\documentclass[12pt,twoside,reqno]{amsart}

\usepackage{a4wide}

\usepackage{amssymb,amsmath,amstext,amsthm,amsfonts,amscd,xcolor}

\usepackage{dsfont}

\usepackage[ansinew]{inputenc}

\usepackage{graphicx}
\usepackage[mathscr]{eucal}
\usepackage{hyperref}

\newcommand{\R}{\mathbb{R}}
\newcommand{\C}{\mathbb{C}}
\newcommand{\N}{\mathbb{N}}

\newcommand{\Su}{\mathbb{S}}

\newcommand{\SL}{{\rm SL}}
\newcommand{\GL}{{\rm GL}}
\newcommand{\gl}{{\rm Mat}}
\newcommand{\Mat}{{\rm Mat}}

\newcommand{\sabs}[1]{\left| #1 \right|} 
\newcommand{\abs}[1]{\bigl| #1 \bigr|} 
\newcommand{\norm}[1]{\lVert#1\rVert} 

\newcommand{\less}{\lesssim}

\newcommand{\ep}{\epsilon} 
\newcommand{\ka}{\kappa}

\newcommand{\Gr}{{\rm Gr}}
\newcommand{\Pp}{\mathbb{P}}
\newcommand{\FF}{\mathscr{F}}
\newcommand{\SO}{{\rm SO}}
\newcommand{\hatv}{\hat{v}}
\newcommand{\hatw}{\hat{w}}
\newcommand{\hatu}{\hat{u}}
\newcommand{\hatp}{\hat{p}}
\newcommand{\hatq}{\hat{q}}

\newcommand{\minexp}{\mathfrak{m}}

\newcommand{\filt}{F}

\newcommand{\Decompsp}{\mathscr{D}}

\newcommand{\mostexp}{\overline{\mathfrak{v}}}
\newcommand{\leastexp}{\underline{\mathfrak{v}}}
\newcommand{\aangle}{\alpha}
\newcommand{\bangle}{\beta}
\newcommand{\sgap}{\sigma}
\newcommand{\rgap}{{\rm gr}}
\newcommand{\rift}{\rho}
\newcommand{\orthC}[1]{\Sigma({#1})}

\newcommand{\Ker}{\rm K}
\newcommand{\Range}{\rm R}

\makeatletter
\newsavebox{\@brx}
\newcommand{\llangle}[1][]{\savebox{\@brx}{\(\m@th{#1\langle}\)}%
  \mathopen{\copy\@brx\mkern2mu\kern-0.9\wd\@brx\usebox{\@brx}}}
\newcommand{\rrangle}[1][]{\savebox{\@brx}{\(\m@th{#1\rangle}\)}%
  \mathclose{\copy\@brx\mkern2mu\kern-0.9\wd\@brx\usebox{\@brx}}}
\makeatother
\newcommand{\linspan}[1]{\llangle {#1} \rrangle}

\newcommand{\normtwo}[1]{
{\left\vert\kern-0.25ex\left\vert\kern-0.25ex\left\vert #1 
    \right\vert\kern-0.25ex\right\vert\kern-0.25ex\right\vert} } 

\newcommand{\proj}{\hat \pi}

\theoremstyle{plain}
\newtheorem{theorem}{Theorem}[section]
\newtheorem{proposition}{Proposition}[section]
\newtheorem{corollary}[proposition]{Corollary}
\newtheorem{lemma}[proposition]{Lemma}
\newtheorem{definition}{Definition}[section]

\numberwithin{equation}{section}

\newtheorem{remark}{Remark}[section]

\title[The Avalanche Principle]{The Avalanche Principle and other estimates on Grassmann manifolds}

\date{}

\begin{document}

\author[P. Duarte]{Pedro Duarte}
\address{Departamento de Matem\'atica and CMAFIO\\
Faculdade de Ci\^encias\\
Universidade de Lisboa\\
Portugal 
}
\email{pmduarte@fc.ul.pt}

\author[S. Klein]{Silvius Klein}
\address{ Department of Mathematical Sciences\\
Norwegian University of Science and Technology (NTNU)\\
Trondheim, Norway\\ 
and IMAR, Bucharest, Romania }
\email{silvius.klein@math.ntnu.no}

\begin{abstract} The main result of this paper, called the {\em Avalanche Principle} (AP),   relates the expansion of a long product of matrices with the product of expansions of the individual matrices. This principle was introduced by  M. Goldstein and W. Schlag  in the context of $\SL(2,\C)$ matrices.  
Besides extending the AP  to matrices of arbitrary dimension, possibly non-invertible, the geometric approach we use here  provides  a relation between  the most expanding (singular) directions of such a long product of matrices and the corresponding singular directions of the first and last matrices in the product. The AP along with other estimates on the action of matrices on Grassmann manifolds will play  a fundamental role in~\cite{LEbook}  
to establish the continuity the Lyapunov exponents and  of the Oseledets decomposition for linear cocycles.

This is the draft of a chapter  in our forthcoming research monograph~\cite{LEbook}.  
\end{abstract}

\maketitle

\section{Grassmann Geometry}
\label{grassmann}
\newcommand{\deltah}{\delta_{{\rm H}}}
\newcommand{\deltamin}{\delta_{{\rm min}}}

Grassmann geometry is the geometric study of manifolds of linear subspaces  of an Euclidean space, and the  action of linear groups (and algebras)  on them.
Its foundations lie on the masterpiece  `Die lineale Ausdehnungslehre' of Hermann Grassmann, whose full geniality is still miscomprehended (see~\cite{Rota}).

\subsection{Projective spaces}
\label{subsection projective spaces}
The projective space is the simplest compact model to study the action of a linear map.
Given a $n$-dimensional Euclidean space $V$, consider the  equivalence relation defined on   $V\setminus\{0\}$ by
$u\equiv v$ if and only if $u=\lambda\,v$ for some $\lambda\neq 0$.
For $v\in V\setminus\{0\}$, the set 
$\hatv:=\{\, \lambda\,v\,\colon\, \lambda\in\R\setminus\{0\}\,\}$ is the equivalence class of the vector $v$ by this relation.
The {\em projective space of $V$} is the quotient  \, $\Pp(V):=\{\, \hatv\,\colon\, v\in V\setminus\{0\}\, \}$ of $V\setminus\{0\}$ by this  equivalence relation.
It is a compact topological space when endowed with the quotient topology.

The unit sphere
$\Su(V):=\{\, v\in V\,\colon\, \norm{v}=1\,\}$ is a compact Riemannian manifold of constant curvature $1$ and diameter $\pi$.
The natural projection $\proj:\Su(V)\to \Pp(V)$, $\proj(v)=\hatv$,
is a (double) covering map. Hence the projective space $\Pp(V)$ 
has a natural smooth  Riemannian structure for which the
covering map $\proj$ is a local isometry.
Thus $\Pp(V)$ is a compact Riemannian manifold with constant curvature $1$ and diameter $\frac{\pi}{2}$.

Given a linear map $g\in \mathcal{L}(V)$ define
$\Pp(g):=\{\, \hatv\in \Pp(V)\,\colon\, g\,v\neq 0\,\}$.
We refer to the linear map $\varphi_g:\Pp(g)\subset \Pp(V)\to\Pp(V)$, $\varphi_g(\hatv) := \proj(\frac{g\,v}{\norm{g\,v}})$, as the {\em projective action of $g$ on } $\Pp(V)$.
If $g$ is invertible then $\varphi_g: \Pp(V)\to\Pp(V)$ is a diffeomorphism with inverse $\varphi_{g^{-1}}: \Pp(V)\to\Pp(V)$.
Through these maps, the group $\GL(V)$, of all linear automorphisms on $V$, acts transitively on the projective space $\Pp(V)$.

We will consider three different metrics on the projective space $\Pp(V)$. The Riemannian distance, $\rho$, measures the length of an arc connecting two points in the sphere.
More precisely, given $u,v\in\Su(V)$,
\begin{equation}
\label{projective Riemannian metric}
 \rho(\hatu,\hat v):=\min\{\, \angle(u,v),  \angle(u,-v)\}\;.
\end{equation}
The second metric, $d$, corresponds to the Euclidean distance measured in the sphere. More precisely, given $u,v\in\Su(V)$,
\begin{equation}
\label{projective Euclidean metric}
 d(\hatu,\hat v):=\min\{\, \norm{u-v},  \norm{u+v}\}   
\end{equation}
measures the smallest chord of the arcs between $u$ and $v$, and between $u$ and $-v$. The third metric, $\delta$, measures the sine of the arc between two points in the sphere. More precisely, given $u,v\in\Su(V)$,
\begin{equation}
\label{projective sine metric}
\delta(\hatu,\hatv):= \frac{\norm{u\wedge v}}{\norm{u}\,\norm{v}}  = \sin(\angle(u,v))\;.
\end{equation} 
The fact that $\delta$ is a metric on $\Pp(V)$ follows from the sine addition law, which implies that
$\sin (\theta+\theta') \leq \sin \theta + \sin\theta'$,
for all $\theta,\theta'\in [0,\frac{\pi}{2}]$.

These three distances are equivalent. For all $\hatu,\hatv\in\Pp(V)$,
\begin{equation}\label{metric relations}
\delta(\hatu,\hatv)=\sin \rho (\hatu,\hatv)\quad \text{ and }\quad d(\hatu,\hatv)={\rm chord} \, \rho (\hatu,\hatv)\;. 
\end{equation}
The inequalities
$$ \frac{2\,\theta}{\pi} \leq \sin \theta \leq {\rm chord}\,\theta= 2\, \sin ({\theta}/{2}) \leq \theta \qquad \forall\, 0\leq \theta\leq \frac{\pi}{2}$$
imply  that
\begin{equation}\label{metric equivalence}
\frac{2}{\pi}\,\rho(\hatu,\hatv)\leq \delta(\hatu,\hatv)  \leq d(\hatu,\hatv) \leq \rho(\hatu,\hatv) \;.  
\end{equation}
Because of ~\eqref{metric relations}, these three metrics determine the same group of isometries on the projective space.

\bigskip

\subsection{Exterior algebra}\label{exterior algebra}
Exterior Algebra was introduced by H. Grassmann in the `Ausdehnungslehre'. We present here an informal description of some of its properties. See the book of Shlomo Stenberg ~\cite{Stenberg-book} for a rigorous treatment of the subject.

Let $V$ be a finite $n$-dimensional Euclidean space.
Given $k$ vectors $v_1,\ldots, v_k\in V$, their $k$-th exterior product is a formal skew-symmetric product
$v_1\wedge \ldots \wedge v_k$, in the sense that
for any permutation $\sigma=(\sigma_1,\ldots,\sigma_k) \in {\rm S}_k$,
$$  v_{\sigma_1} \wedge \ldots \wedge v_{\sigma_k} 
= (-1)^{{\rm sgn}(\sigma)} v_1\wedge \ldots  \wedge v_k
 \;.  $$
These formal products are elements of  an anti-commutative and associative graded algebra  $(\wedge_\ast V, +,\wedge)$,  called the {\em exterior algebra} of $V$.
Formal products $v_1\wedge \ldots \wedge v_k$ are called
  {\em simple $k$-vectors} of $V$.
The {\em $k$-th exterior power of $V$}, denoted by $\wedge_k  V$,
is the linear  span of all  simple $k$ vectors of $V$. 
Elements of $\wedge_k  V$ are called {\em $k$-vectors}.

An easy consequence of this formal definition is that
$v_1\wedge \ldots \wedge v_k=0$ if and only if $v_1,\ldots, v_k$
are linearly dependent. Another simple consequence is that given
two bases $\{v_1,\ldots, v_k\}$ and  $\{w_1,\ldots, w_k\}$ of the same $k$-dimensional linear subspace of $V$, if for some real matrix
$A=(a_{ij})$  we have $w_i=\sum_{j=1}^k a_{ij}\,v_j$
for all $i=1,\ldots, k$, then 
$$ w_1\wedge \ldots \wedge w_k = (\det A)\,v_1\wedge \ldots \wedge v_k\;.$$
More generally, two families  
$\{v_1,\ldots, v_k\}$ and  $\{w_1,\ldots, w_k\}$ of linearly independent vectors  span the same
$k$-di\-men\-sional subspace if and only if for some  real number
$\lambda\neq 0$, $ w_1\wedge \ldots \wedge w_k = \lambda \,v_1\wedge \ldots \wedge v_k$. 
Hence  we identify  the line spanned by a simple $k$-vector $v=v_1\wedge \ldots \wedge v_k$,
 i.e., the projective point   $\hat{v}\in \Pp(\wedge_k  V)$ determined by  $v$,  with the $k$-di\-men\-sional subspace spanned by the vectors $\{v_1,\ldots, v_k\}$,  
denoted hereafter  by $\linspan{ v_1\wedge \ldots \wedge v_k }$.

The subspaces $\wedge_k  V$ induce the grading structure of the exterior algebra $\wedge_\ast V$, i.e.,
 we have the direct sum decomposition $\wedge_\ast V=\oplus_{k=0}^{\dim V} \wedge_k V$ 
with $(\wedge_k V)\wedge (\wedge_{k'} V)\subset \wedge_{k+k'} V$
for all $0\leq k,k'\leq \dim V$.
Geometrically, the exterior product operation
$\wedge:\wedge_k V\times   \wedge_{k'} V \to \wedge_{k+k'} V$
corresponds to the algebraic sum of linear subspaces, in the sense that given families $\{v_1,\ldots, v_k\}$ and  $\{w_1,\ldots, w_k\}$ of linearly independent vectors such that
 $\linspan{ v_1\wedge \ldots \wedge v_k } \cap 
\linspan{ w_1\wedge \ldots \wedge w_{k'} }=0$, then
$$\linspan{ v_1\wedge \ldots \wedge v_k\, \wedge \,  w_1\wedge \ldots \wedge w_{k'} } = 
\linspan{ v_1\wedge \ldots \wedge v_k } +
\linspan{ w_1\wedge \ldots \wedge w_{k'} } \;. $$

Let $\Lambda_k^n$ be the set of all $k$-subsets 
$I=\{i_1,\ldots, i_k\}\subset \{1,\ldots, n\}$,
with $i_1<\ldots <i_k$,
and  order it lexicographically. 
Given a basis $\{e_1, \ldots, e_n\}$ of $V$,
define for each  $I\in\Lambda^n_k$, the $k$-th exterior product $e_I=e_{i_1}\wedge \ldots\wedge e_{i_k}$.
Then the ordered family $\{e_I\,\colon \,I\in\Lambda^n_k\}$ is a basis of $\wedge_k V$. In particular\,
$\dim \wedge_k V=\binom{n}{k}$.

The exterior algebra $\wedge_\ast V$ inherits an Euclidean
structure from $V$. More precisely, there is a unique inner product
on $\wedge_\ast V$ such that for any orthonormal basis $\{e_1, \ldots, e_n\}$ of $V$, the family
$\{\, e_I\,\colon\, I\in \Lambda^n_k, \; 0\leq k\leq n\,\}$
is  an orthonormal basis of the exterior algebra $\wedge_\ast V$.

A  simple $k$-vector
$v_1\wedge \ldots \wedge v_k$ of norm one will be called a {\em unit $k$-vector}. From the previous considerations
the correspondence
$v_1\wedge \ldots \wedge v_k\mapsto \linspan{v_1\wedge \ldots \wedge v_k}$ is one-to-one, between the set of unit $k$-vectors in $\wedge_k V$ and the set of oriented $k$-dimensional linear subspaces of $V$.
In particular, if $V$ is an oriented Euclidean space then the $1$-dimensional space $\wedge_n V$ has a canonical unit $n$-vector,
denoted by $\omega$, and called the {\em volume element} of $\wedge_n V$. In this case there is a unique operator,
 called the {\em Hodge star} operator,
$\ast:\wedge_\ast V\to \wedge_\ast V$ defined by  
$$ v\wedge (\ast w) = \langle v, w\rangle\, \omega,\quad \text{ for all }\;  v,w\in \wedge_\ast V\;. $$
The Hodge star operator maps
$\wedge_k V$ isomorphically, and isometrically, onto $\wedge_{n-k} V$, for all
$0\leq k\leq n$.
Geometrically it corresponds to the
orthogonal complement operation on  linear subspaces, i.e., for any simple $k$-vector,
$$ \linspan{\ast (v_1\wedge \ldots \wedge v_k) } 
= \linspan{ v_1\wedge \ldots \wedge v_k  }^\perp\;. $$

A dual product operation $\vee :\wedge_\ast V\times \wedge_\ast V\to \wedge_\ast V$ can  be defined by
$$ v \vee w := \ast( (\ast v)\wedge (\ast w) ),\quad
\text{ for all }\; v,w\in \wedge_\ast V \;.$$
This operation maps 
$\wedge_k V\times   \wedge_{k'} V$ to $\wedge_{k+k'-n} V$, and describes   the intersection operation on linear subspaces, 
in the sense that given families $\{v_1,\ldots, v_k\}$ and  $\{w_1,\ldots, w_k\}$ of linearly independent vectors with
 $\linspan{ v_1\wedge \ldots \wedge v_k } + 
\linspan{ w_1\wedge \ldots \wedge w_{k'} }= V$, then
$$\linspan{ (v_1\wedge \ldots \wedge v_k) \vee (w_1\wedge \ldots \wedge w_{k'}) } = 
\linspan{ v_1\wedge \ldots \wedge v_k } \cap 
\linspan{ w_1\wedge \ldots \wedge w_{k'} } \;. $$
By duality, this interpretation of the $\vee$-operation  reduces to the previous ones on sums ($\wedge$) and complements ($\ast$).

Any linear map $g:V\to V$ induces 
a linear map $\wedge_k g:\wedge_k V\to \wedge_k V$,
called the {\em $k$-th exterior power} of $g$,
such that  for all $v_1,\ldots, v_k\in V$,
$$\wedge_k g(v_1\wedge \ldots \wedge v_k)=
g(v_1)\wedge \ldots \wedge g(v_k)\;.$$
This construction is functorial
in the sense that for all linear maps $g,g':V\to V$,
$$\wedge_k{\rm id}_V = {\rm id}_{\wedge_k V},  \;\; \wedge_k(g'\circ g)=\wedge_k g'\circ  \wedge_k g\quad \;\; \text{ and }\;\;
\wedge_k g^\ast = (\wedge_k g)^\ast \;,$$
where $g^\ast:V\to V$ denotes the {\em adjoint operator}.

A clear consequence of these properties is that if $g:V\to V$ is an {\em orthogonal automorphism}, i.e.,
$g^\ast\circ g={\rm id}_V$,  then so is
$\wedge_k g: \wedge_k V\to \wedge_k V$.

Consider a matrix $A\in\Mat_n(\R)$.
Given $I,J\in\Lambda^ n_k$,
we denote by  $A_{I\times J}$  the square sub-matrix of $A$
 indexed in  $I\times J$.
If a linear map $g:V\to V$ is represented by  
$A$ relative to a basis $\{e_1, \ldots, e_n\}$,
 then the $k$-th exterior power
$\wedge_k g:\wedge_k V\to \wedge_k V$ is represented
by the matrix $\wedge_k A:= (\det A_{I\times J})_{I,J}$  relative to  the basis $\{e_I\,:\,I\in\Lambda^n_k\}$. The matrix $\wedge_k A$ is called the {\em $k$-th exterior power} of $A$.
Obviously, matrix exterior powers satisfy the same functorial properties
as linear maps, i.e., for all $A,A'\in\Mat_n(\R)$,
 $$\wedge_k{\rm I}_n = {\rm I}_{\binom{n}{k}},  \;\; \wedge_k(A' A)=(\wedge_k A') (\wedge_k g)\quad \;\; \text{ and }\;\;
\wedge_k A^\ast = (\wedge_k A)^\ast \;,$$
where $A^\ast$ denotes the {\em transpose} matrix of $A$.

\bigskip

\subsection{Grassmann manifolds} \label{grassmannians}

Grassmannians, like projective spaces, are compact Riemannian manifolds which stage the action of linear maps.
For each $0\leq k\leq n$,  the {\em Grassmannian} $\Gr_k(V)$ is the space of all $k$-dimensional linear subspaces of $V$.
Notice that the projective space  $\Pp(V)$ and the Grassmannian $\Gr_1(V)$ are the same object if we identify each point $\hatv\in\Pp(V)$ with the line $\langle v\rangle= \{\, \lambda\,v\,\colon\, \lambda\in\R \,\}$. The full Grassmannian $\Gr(V)$ is the union of all Grassmannians $\Gr_k(V)$ with $0\leq k\leq n$. Denote by $\mathcal{L}(V)$   the algebra of linear endomorphisms on $V$, and consider the map
$\pi:\Gr(V)\to \mathcal{L}(V)$, $E\mapsto \pi_E$, that assigns the orthogonal projection $\pi_E$ onto $E$,
to each subspace $E\in\Gr(V)$. This map is one-to-one, and we endow $\Gr(V)$ with the unique topology that makes the map $\pi:\Gr(V)\to \pi(\Gr(V))$ a homeomorphsim. With it, $\Gr(V)$ becomes a compact space, and
each Grassmannian $\Gr_k(V)$ is a closed connected subspace of $\Gr(V)$.

The   group $\GL(V)$ acts transitively on each Grassmannian.
The action of $\GL(V)$ on $\Gr_k(V)$ is given by  
$\cdot:\GL(V)\times \Gr_k(V)\to \Gr_k(V)$,
$(g,E)\mapsto g\,E$. 
The special orthogonal group $\SO(V)$,
of orientation preserving orthogonal automorphisms, acts transitively on   Grassmannians too. All Grassmannians are compact homogeneous spaces.

For each $0\leq k\leq n$,  the {\em Pl\"ucker} embedding is the map
$\psi:\Gr_k(V)\to \Pp(\wedge_k V)$  that to each subspace $E$ in $\Gr_k(V)$
assigns the projective point $\hatv\in \Pp(\wedge_k V)$, 
where $v=v_1\wedge \ldots \wedge v_k$ is any simple $k$-vector formed as exterior product of a basis $\{ v_1,\ldots, v_k \}$  of $E$. This map is one-to-one and  equivariant, i.e.,
for all $g\in\GL(V)$ and $E\in\Gr(V)$,
\begin{equation}\label{Plucker equivariance}
 \psi(g\,E) = \varphi_{ \wedge_k g }\psi(E) \;. 
\end{equation}

We will consider the metrics 
$\rho, d,\delta :\Gr_k(V) \times \Gr_k(V)\to [0,+\infty)$  defined for any given $E,F\in\Gr_k(V)$  by
\begin{align}
\label{Grassmann:arc distance}
\rho(E,F) &:= \rho(\psi(E),\psi(F))  \;,\\
\label{Grassmann:distance}
d(E,F) &:= d(\psi(E),\psi(F))  \;,\\
\label{Grassmann:sine distance}
\delta(E,F) &:= \delta(\psi(E),\psi(F))   \;.
\end{align}
which assign diameter $\frac{\pi}{2}$, $\sqrt{2}$ and $1$, respectively, to the manifold $\Gr_k(V)$.
These distances are preserved by orthogonal linear maps in $\SO(V)$.

We also define the {\em minimum distance} between 
any two subspaces $E,F\in\Gr(V)$, 
$$ \deltamin(E,F):=  \min_{u\in E\setminus\{0\}, v\in F\setminus\{0\}} \delta(\hatu,\hatv)\;,  $$
and the {\em Hausdorff distance} between 
subspaces $E,F\in \Gr_k(V)$,
$$ \deltah(E,F):= \max\left\{ \,
\max_{u\in E\setminus\{0\}} \deltamin(\hatu,F),\,
\max_{v\in F\setminus\{0\}} \deltamin(\hatv,E)\,
\right\}\;. $$

\begin{definition}
\label{def orthog proj}
Given  $E,F\in\Gr(V)$, we denote by
$\pi_F:V\to V$ the orthogonal projection onto $F$,
and by $\pi_{E,F}:E\to F$ the restriction  of $\pi_F$ to $E$.
\end{definition}

\begin{proposition}
\label{delta, deltamin, deltaH}
Given $E,F\in\Gr_k(V)$,
\begin{enumerate}
\item[(a)]\;  $\displaystyle  \delta(E,F) = \sqrt{1-\det(\pi_{E,F})^2} =
\sqrt{1-\det(\pi_{F,E})^2} $,
\item[(b)]\;  $\displaystyle \deltah(E,F)= \norm{\pi_{E,F^\perp}} = \norm{\pi_{F,E^\perp}} $,
\item[(c)]\;  $\displaystyle \deltah(E,F) \leq  \delta(E,F) $.
\end{enumerate}

\end{proposition}

\begin{proof}
Consider the unit $k$-vectors $e=\Psi(E)$ and $f=\Psi(F)$.

For (a) notice first that
$\delta(E,F)=\delta(e,f)=\sqrt{1-\langle e,f\rangle^2}$.
Since the exterior power $\wedge_k \pi_{F,E}:\wedge_k F\to \wedge_k E$ is also an orthogonal projection we have
$\langle e,f\rangle = 
\langle e,\wedge_k \pi_{F,E} (f)\rangle =  \norm{\wedge_k \pi_{F,E}} = \abs{\det(\pi_{F,E}) } $.

Given an orthogonal map $g\in\SO(V)$ such that
$g(F)=E$, we have $g^{-1}(E^\perp)=F^\perp$ 
and $\pi_{E,F^\perp}= g^{-1}\circ \pi_{F,E^\perp}\circ g$. Therefore $\norm{\pi_{E,F^\perp}} = \norm{\pi_{F,E^\perp}} $. 

Item (b) follows because for any unit vector $u\in\hatu$, with $\hatu\in\Pp(E)$,
$$ \norm{\pi_{E,F^\perp} (u)} = \min_{v\in F\setminus\{0\}} \delta(\hatu,\hatv)\;. $$

Since $\pi_{E,F}$ is an orthogonal projection all its singular values are in the range $[0,1]$.
Hence, for any unit vector $u\in E$,
$\norm{\pi_{E,F}(u)}\geq \minexp(\pi_{E,F})\geq \det(\pi_{E,F})$. Thus
\begin{align*}
\norm{\pi_{E,F^\perp} (u)}^2
= 1- \norm{\pi_{E,F} (u)}^2
\leq 1- \det(\pi_{E,F})^2\;,
\end{align*}
and  (c) follows taking the maximum over all unit vectors
$u\in E$.

\end{proof}

Given $k,k'\geq 0$ such that $k+k'\geq n=\dim V$, the
intersection of subspaces is an operation
$\cap:\Gr_{k,k'}(\cap)\subset \Gr_k(V)\times \Gr_{k'}(V)\to \Gr_{k+k'-n}(V)$
where 
\begin{definition} 
\label{def intersection domain}
the domain  is defined by
$$ \Gr_{k,k'}(\cap):=\{\,(E,E')\in \Gr_k(V)\times \Gr_{k'}(V)\,\colon\,
E+E'=V\,  \}\;. $$
\end{definition}

Similarly, given $k,k'\geq 0$ such that $k+k'\leq n=\dim V$, the
algebraic sum of subspaces is operation
$+:\Gr_{k,k'}(+)\subset \Gr_k(V)\times \Gr_{k'}(V)\to \Gr_{k+k'-n}(V)$
where 
\begin{definition} 
\label{def sum domain}
the domain  is defined by
$$ \Gr_{k,k'}(+):=\{\,(E,E')\in \Gr_k(V)\times \Gr_{k'}(V)\,\colon\,
E\cap E' =\{0\}\,  \}\;. $$
\end{definition}
The considerations in subsection~ \ref{exterior algebra}  show that
the Pl\"ucker  embedding satisfies the following relations:
\begin{proposition}
\label{prop }
Given $E\in \Gr_k(V)$, $E'\in\Gr_{k'}(V)$,
consider unit vectors $v\in \Psi(E)$ and $v'\in\Psi(E')$.
\begin{enumerate}
\item[(a)]\, If $(E,E')\in \Gr_{k,k'}(\cap)$ \, then  \,
$ \psi(E\cap E')=\widehat{v\vee v'}$.
\item[(b)]\, If $ (E,E')\in \Gr_{k,k'}(+)$ \, then  \,
$\psi(E + E')=\widehat{v\wedge v'} $.
\end{enumerate}
\end{proposition}

A duality between sums and intersections stems from these facts.  

\begin{proposition}\label{+cap duality}
The orthogonal complement operation $E\mapsto E^\perp$  is a $d$-isometric involution on $\Gr(V)$  which maps $ \Gr_{k,k'}(+)$ to $ \Gr_{n-k,n-k'}(\cap)$
and satisfies for all $(E,E')\in  \Gr_{k,k'}(+)$,
\begin{equation*}
(E+E')^\perp = (E^\perp)\cap (E')^\perp \;. 
\end{equation*}
\end{proposition}

The composition semigroup $\mathcal{L}(V)$ has two partial actions
on  Grassmannians, called the {\em push-forward  action}
and the {\em pull-back action}. Before introducing them a couple facts is needed.

\begin{definition}
\label{def kernel, range}
Given $g\in \mathcal{L}(V)$, we denote by 
$\Ker g:=\{\, v\in V\,\colon\, g\,v=0\,\}$
the the {\em kernel of $g$}, and by 
  $\Range g:=\{\, g\,v\,\colon\, v\in V\,\}$
 the {\em range  of $g$}.
\end{definition}

\begin{lemma}\label{push-forwards,pull-backs}
Given $g\in\mathcal{L}(V)$ and  $E\in\Gr(V)$, 
\begin{enumerate}
\item if $E\cap (\Ker g)=\{0\}$ then the linear map $g\vert_E:E\to g(E)$ is an isomorphism,
and in particular\, $\dim g(E)=\dim E$.
\item if $E + (\Range g)= V$ then the linear map  $g^ \ast\vert_{E^ \perp}:E^ \perp\to g^{-1}(E)^ \perp$ is an isomorphism,
and  in particular\, $\dim g^{-1}(E)=\dim E$.
\end{enumerate}
\end{lemma}

\begin{proof}
The first statement is obvious because
if $E\cap (\Ker g)=\{0\}$ then $\Ker (g\vert_E)=\{0\}$.
If  $E + (\Range g)=V$ then, since $\Ker g^\ast=(\Range g)^\perp$,
we have $E^ \perp\cap (\Ker {g^ \ast})=E^ \perp\cap (\Range {g})^ \perp=(E+\Range g)^ \perp = \{0\}$.
Hence by 1,  the linear map $g^ \ast\vert_{E^ \perp}:{E^ \perp}\to g^ \ast(E^ \perp)$ is an isomorphism.
It is now enough to remark that $ g^ \ast(E^ \perp) = g^{-1}(E)^ \perp$.
In fact, the inclusion  $ g^ \ast(E^ \perp) \subset g^{-1}(E)^ \perp$ is clear.
Since $g^\ast\vert{E^\perp}$ is injective,
$\dim g^ \ast(E^ \perp) = \dim (E^\perp)$. On the other hand, by the transversality condition,
 $g^{-1}(E)$ has dimension
\begin{align*}
\dim g^{-1}(E) &= \dim \left( (g\vert_{(\Ker g)^ \perp})^{-1}(E\cap \Range g)\right) +\dim (\Ker g)\\
&= \dim  (E\cap \Range g) +\dim (\Ker g)\\
&= \dim  (E)+\dim(\Range g)-n +\dim (\Ker g) = \dim(E)\;.
\end{align*}
Hence both  $ g^ \ast(E^ \perp) $ and $ g^{-1}(E)^ \perp$ have  dimension equal to $\dim( E^ \perp)$, and the equality follows.  
\end{proof}

Given $g\in \mathcal{L}(V)$ and $k\geq 0$ such that $k+\dim (\Ker g)\leq n=\dim V$,
the {\em push-forward by $g$} is the map   
$ \varphi_g: \Gr_{k}(g)\subset \Gr_k(V) \to \Gr_{k}(V)$, $E\mapsto g E$,
where 
\begin{definition} 
\label{def push forward domain}
the domain  is defined by
$$ \Gr_{k}(g):=\{\, E \in \Gr_k(V) \,\colon\,
E\cap (\Ker g)=\{0\}\,  \}\;. $$
\end{definition}
Similarly, given $k\geq 0$ such that $k+\dim(\Range g)\geq n=\dim V$, the
{\em pull-back by $g$} is the map   
$  \varphi_{g^{-1}}: \Gr_{k}(g^{-1})\subset \Gr_k(V) \to \Gr_{k}(V)$, $E\mapsto g^{-1} E$,
where 
\begin{definition} 
the domain  is defined by
\label{def pull back domain}
$$ \Gr_{k}(g^{-1}):=\{\, E \in \Gr_k(V) \,\colon\,
E + (\Range g) =V \,  \}\;. $$
\end{definition}

From the proof of proposition~\ref{push-forwards,pull-backs} we obtain a duality between push-forwards and pull-backs which  can be expressed as follows. 
\begin{proposition} \label{push-pull duality}
Given $g\in \mathcal{L}(V)$ and $k\geq 0$ such that $k+\dim (\Range g)\geq n=\dim V$, we have 
$\Gr_k(g^{-1})=\Gr_{n-k}(g^\ast)^\perp$ and for all $E\in \Gr_k(g^{-1})$,
\begin{equation*}
(g^{-1} E)^\perp = g^\ast(E^\perp) \;.
\end{equation*}
\end{proposition}

In  section~\ref{le}  we will  derive modulus of Lipschitz continuity, w.r.t. the metric $\delta$, for  the  sum, intersection, push-forward and pull-back operations.

\subsection{Flag manifolds}

Let $V$ be a finite $n$-dimensional Euclidean space.
Any strictly increasing sequence of linear subspaces
$F_1\subset F_2\subset \ldots \subset F_{k}\subset V$ is called  a {\em flag} in the Euclidean space $V$. 
Formally, flags are denoted as lists  $F=(F_1,\ldots, F_{k})$.
The sequence $\tau=(\tau_1,\ldots, \tau_k)$ of dimensions $\tau_j=\dim F_j$ is called  the {\em signature} of the flag $F$. 
The integer $k$ is called the {\em length} of the flag $F$, and the {\em length} of the signature $\tau$.
Let $\FF(V)$ be the set of all flags in $V$, and define $ \FF_\tau(V)$ to be the space of flags with a given signature $\tau$.  
 Two special cases of flag spaces are the  projective
space $\Pp(V)=\FF_\tau(V)$,  when $\tau=(1)$,
and the Grassmannian  $\Gr_k(V)=\FF_\tau(V)$,  when $\tau=(k)$.

The general linear group $\GL(V)$ acts naturally on $\FF(V)$.
Given  $g\in\GL(V)$ the action of $g$ on $\FF_\tau(V)$ is given by the
map  $\varphi_g:\FF_\tau(V) \to \FF_\tau(V)$, $\varphi_g F=(g F_1,\ldots, g F_{k})$. The special orthogonal subgroup $\SO(V)\subset \GL(V)$
acts transitively on $\FF_\tau(V)$. Hence,
all flag manifolds $\FF_\tau(V)$ are compact homogeneous spaces.
Each of them is a compact connected Riemannian manifold where the
group $\SO(V)$ acts by isometries. 
Since $\FF_\tau(V)\subset \Gr_{\tau_1}(V)\times \Gr_{\tau_2}(V)\times \ldots
\times \Gr_{\tau_k}(V)$,   the product distances
\begin{align}
 \label{rho:tau}
 \rho_\tau(\filt,\filt') &=\max_{1\leq j\leq k} \rho(F_j,F_j')  \\
 \label{d:tau}
 d_\tau(\filt,\filt') &=\max_{1\leq j\leq k} d(F_j,F_j')  \\
 \label{delta:tau}
 \delta_\tau(\filt,\filt') &=\max_{1\leq j\leq k} \delta(F_j,F_j')  
\end{align}
are equivalent to the Riemannian distance on $\FF_\tau(V)$.
With these metrics, the flag manifold  $\FF_\tau(V)$ has diameter $\frac{\pi}{2}$, $\sqrt{2}$ and $1$, respectively.
The group $\SO(V)$ acts isometrically on $\FF_\tau(V)$ with respect to these distances.

Given a signature $\tau=(\tau_1,\ldots, \tau_k)$,
if $n=\dim V$,  we define 
$$\tau^\perp :=(n-\tau_k, \ldots, n-\tau_1)\;. $$
When $\tau=(\tau_1,\ldots, \tau_k)$ we will write
$\tau^\perp=(\tau_1^\perp,\ldots, \tau_k^\perp)$,
where $\tau_j^\perp= n-\tau_{k+1-i}$.
\begin{definition}
\label{def flag orth complement}
Given a  flag $F=(F_1,\ldots, F_k)\in\FF_\tau(V)$, its 
{\em orthogonal complement } is  the ${\tau^\perp}$-flag\, 
$F^\perp : =(F_k^\perp,\ldots, F_1^\perp)$.
\end{definition}

The map $\cdot^\perp:\FF(V)\to \FF(V)$ is an isometric  involution on $\FF(V)$,
 mapping $\FF_{\tau}(V)$ onto $\FF_{{\tau^\perp}}(V)$.
The involution character, $(F^\perp)^\perp = F$
for all $F\in\FF(V)$, is clear.
As explained in section~\ref{exterior algebra}, the Hodge star operator  $\ast:\wedge_k V \to \wedge_{n-k} V$ is an isometry between these Euclidean spaces.
By choice of  metrics on the Grassmannians, see~\eqref{Grassmann:distance}, the Pl\"ucker embeddings are  isometries. Finally, the  Pl\"ucker embedding conjugates the 
orthogonal complement map $\cdot^\perp:\Gr_k(V)\to\Gr_{n-k}(V)$ with the Hodge star operator. Hence for each $0\leq k\leq n$, the map $\cdot^\perp:\Gr_k(V)\to\Gr_{n-k}(V)$ is
an isometry. The analogous conclusion for flags follows from the defintion of  distance $d_\tau$.

%
%
%
%

Given $g\in \mathcal{L}(V)$ and a signature $\tau$ such that $\tau_i+\dim (\Ker g)\leq n$ for all $i$,
the {\em push-forward by $g$} on flags  is the map   
$ \varphi_g: \FF_{\tau}(g)\subset \FF_\tau(V) \to \FF_{\tau}(V)$, $\varphi_g F := (g \,F_1,\ldots, g\,F_k)$,
where 
\begin{definition} 
\label{def push forward flag domain}
the domain  is defined by
\begin{align*}
\FF_{\tau}(g) &:=\{\, F \in \FF_\tau(V) \,\colon\,
F_k\cap (\Ker g)=\{0\}\,  \}\;. 
\end{align*}
\end{definition}

Similarly, given a signature $\tau$  such that $\tau_i+\dim(\Range g)\geq n$ for all $i$, the
{\em pull-back  by $g$} on flags  is the map   
$  \varphi_{g^{-1}}: \FF_{\tau}(g^{-1})\subset \FF_\tau(V) \to \FF_{\tau}(V)$, $\varphi_{g^{-1}} F := (g^{-1} F_1,\ldots, g^{-1} F_k)$, where 
\begin{definition} 
\label{def pull-back flag domain}
the domain  is defined by
\begin{align*}
 \FF_{\tau}(g^{-1}) &:=\{\, F \in \FF_\tau(V) \,\colon\,
F_1 + (\Range g) =V \,  \}\;. 
\end{align*}
\end{definition} 
 
The duality between duality between push-forwards and pull-backs is expressed as follows.

\begin{proposition} \label{push-pull duality: flags}
Given $g\in \mathcal{L}(V)$,\, 
$\FF_\tau(g^{-1})=\FF_{\tau^\perp}(g^\ast)^\perp$ and for all $F\in \FF_\tau(g^{-1})$,
\begin{equation*}
(\varphi_{g^{-1}} F)^\perp = \varphi_{g^\ast}(F^\perp) \;.
\end{equation*}
\end{proposition}

\bigskip

\section{Singular Value Geometry}
\label{svg}

Singular value geometry refers here to the geometry of the {\em singular value decomposition} (SVD)
of a linear endomorphism $g:V\to V$ on some Euclidean space $V$. It also refers to some geometric properties of the action of $g$ on Grassmannians and flag manifolds related to the singular value decomposition of $g$.

\subsection{Singular value decomposition}
Let $V$ be an Euclidean space of dimension $n$.
\begin{definition}
\label{singular values}
Given $g\in\mathcal{L}(V)$, the {\em singular values} of $g$ are the square roots of the eigenvalues of the 
quadratic form $Q_g:V\to\R$, $Q_g(v)=\norm{g\,v}^2=\langle g v, gv\rangle$, i.e., the eigenvalues of the
positive semi-definite self-adjoint operator $\sqrt{g^\ast g}$.
\end{definition}
Given  $g\in \mathcal{L}(V)$, let 
$$ s_1(g) \geq s_2(g)\geq \ldots \geq s_n(g)\geq 0\;,$$
denote the sorted singular values of $g$. The adjoint   $g^ \ast$ has the same singular values as $g$
because the operators $\sqrt{g^\ast g}$ and $\sqrt{g\, g^\ast}$ are conjugate.

The largest singular value, $s_1(g)$, is the square root of the maximum value of $Q_g$ over the unit sphere,
i.e., $s_1(g)=\max_{\norm{v}=1} \norm{g\,v}=\norm{g}$ is the operator norm of $g$.
Likewise, the least singular value, $s_n(g)$, is the square root of the minimum value of $Q_g$ over the unit sphere,
i.e., $s_n(g)=\min_{\norm{v}=1} \norm{g\,v}$. This number also denoted by  $\minexp(g)$ is called the {\em least expansion} of $g$.
If $g$ is invertible  $\minexp(g)=\norm{g^{-1}}^{-1}$,
while otherwise  $\minexp(g)=0$. 

\begin{definition}
\label{singular vectors}
The eigenvectors of the quadratic form $Q_g$, i.e., of the positive semi-definite self-adjoint operator $\sqrt{g^\ast g}$,
are called  the {\em singular vectors} of $g$.
\end{definition}
By the spectral theory of self-adjoint operators, for any $g\in\mathcal{L}(V)$
there exists an orthonormal basis consisting of singular vector of $g$.

\begin{proposition}
Given  $g\in\mathcal{L}(V)$ and  $v\in V$ be a unit singular vector of $g$ such that  $g^\ast g\,v=\lambda^2\,v$, there exists a unit vector
$w\in V$  such that
\begin{enumerate}
\item[(a)] $g\,v=\lambda\, w$,
\item[(b)] $g\,g^\ast\,w=\lambda^2\,w$, i.e., $w$ is a singular vector of $g^\ast$.
\end{enumerate}
\end{proposition}

\begin{proof}
Let $v\in V$ be a unit singular vector of $g$.
Then $g^\ast g\,v=\lambda^2\,v$ and
$\lambda^2=\langle\lambda^2\,v,v\rangle=\langle g^\ast g\,v,v\rangle=\norm{g\,v}^ 2$,
which implies that $\lambda=\norm{g\,v}$.
Since  $(g\,g^\ast)\,(g\,v)=g\,(g^\ast g)\,v=\lambda^2\,g\,v$, if $\lambda\neq 0$ then 
setting $w= g\,v/\norm{g\,v}=\lambda^{-1}\,g\,v$, we have   $(g\,g^\ast)\,w= \lambda^2\,w$, which proves that
$w$ is a singular vector of $g^\ast$. By definition $g\,v=\lambda\,w$.
When $\lambda=0$,
take $w$ to be any unit vector in $\Ker g^\ast$. Notice that $\dim(\Ker g)=\dim(\Ker g^ \ast)$.
In this case $v$ and $w$ are singular vectors of $g$ and $g^\ast$, respectively, such that $g\,v=0= \lambda\,w$. 
\end{proof}

 By the previous proposition, given $g\in\mathcal{L}(V)$ there exist two orthonormal singular vector  basis  of $V$,
 $\{v_1(g),\ldots, v_n(g)\}$  and   $\{v_1(g^\ast),\ldots, v_n(g^\ast)\}$ 
for $g$ and  $g^\ast$, respectively, such that
\begin{equation*} 
 g\, v_j(g)= s_j(g)\, v_j(g^\ast) \quad \text{ for all }\; 1\leq j \leq n\;.
\end{equation*} 
Denote by $D_g$ the diagonal matrix with diagonal entries $s_j(g)$, $1\leq j \leq n$, seen as an operator $D_g\in\mathcal{L}(\R^ n)$. Define the linear maps $U_g, U_{g^\ast}:\R^n\to V$ by $U_g(e_j)=v_ j(g)$ and $U_{g^\ast}(e_j)=v_ j(g^\ast)$, for all  $1\leq j \leq n$, where the $e_j$ are the vectors of the canonical basis in $\R^n$. Then the following decomposition holds 
$$ g= U_{g^\ast}\,D_g\, (U_g)^\ast\;,$$
known as the {\em singular value decomposition} (SVD) of $g$.

We say that  $g$ has a {\em simple singular spectrum}
if its $n$ singular values are all distinct. When $g$ has simple singular spectrum, the singular vectors
$v_j(g)$ and $v_j(g^\ast)$  above are uniquely determined up to a sign, and in particular they determine well-defined projective points 
$\mostexp_j(g), \mostexp_j(g^\ast)\in\Pp(V)$.

\begin{definition}\label{def singular basis}
Given $g\in\mathcal{L}(V)$, we call {\em singular basis} of $g$ to any orthonormal basis
$\{v_1,\ldots, v_m\}$ of $V$ ordered such that
$\norm{g\,v_i } = s_i(g)$, 
for all  $i=1,\ldots, m$.
\end{definition}

\subsection{Gaps and most expanding directions}

Consider a linear map $g\in\mathcal{L}(V)$ and a number $0\leq k \leq \dim V$.

\begin{definition}
\label{gap ratio}
The {\em $k$-th gap ratio } of $g$ is defined to be
\begin{equation*}
\rgap_k(g):=\frac{s_k(g)}{s_{k+1}(g)}\geq 1 \;. 
\end{equation*}
\end{definition}
We will write $\rgap(g)$ instead of 
$\rgap_1(g)$.
\begin{definition}
\label{def has a gap ratio}
We say that $g$ {\em has a first singular gap} when
$\rgap(g)>1$. More generally, we say that $g$ {\em has a $k$  singular gap} when $\rgap_k(g)>1$. 
\end{definition}
In some occasions  it is convenient to work with the inverse quantity, denoted by 
\begin{equation}
\label{inverse gap ratio}
\sgap_k(g):=\rgap_k(g)^{-1}\leq 1 \;. 
\end{equation}

\begin{proposition}\label{sing vals and ext pws}
For any $1\leq k \leq \dim V$, 
$\norm{\wedge_k g} = s_1(g)\,\ldots\, s_k(g)$.
\end{proposition}

\begin{proof}
Let $n=\dim V$. Consider orthonormal singular vector basis
$\{v_1,\ldots, v_n\}$  and   $\{v_1^\ast,\ldots, v_n^\ast\}$ 
for $g$ and  $g^\ast$, respectively, such that
$$  g\, v_j = s_j \, v_j^\ast,\; \text{ where }\; s_j=s_j(g) \quad \text{ for all }\; 1\leq j \leq n\;.$$
Given $I=\{i_1,\ldots, i_k\}\in \Lambda^n_k$, with
$i_1<\ldots < i_k$, we have
$$(\wedge_k g)(v_{i_1}\wedge \ldots \wedge v_{i_k})
= (s_{i_1} \ldots s_{i_k} )\, (v_{i_1}^\ast\wedge \ldots \wedge v_{i_k}^\ast)\;. $$
Therefore, the $k$-vectors  $v_I=v_{i_1}\wedge \ldots \wedge v_{i_k}$ and  $v_I^\ast=v_{i_1}^\ast\wedge \ldots \wedge v_{i_k}^\ast$ form two orthonormal singular vector basis for $\wedge_k g$ and $\wedge_k g^\ast$, respectively,
while the products $s_I= s_{i_1} \ldots s_{i_k}$ are the singular values of both $\wedge_k g$ and $\wedge_k g^\ast$.
Since the largest singular value is attained with $I=\{1,\ldots, k\}$, $\norm{\wedge_k g} = s_1 \,\ldots\, s_k$.

\end{proof}

\begin{corollary}
For any $1\leq k <\dim V$, \; 
$$\rgap_k(g) = \frac{\norm{\wedge_k g}^2}{\norm{\wedge_{k-1} g}\, \norm{\wedge_{k+1} g}}\;.$$
\end{corollary}
 
Given $g\in\mathcal{L}(V)$, if $\rgap(g)>1$ then
 the singular value  $s_1(g)=\norm{g}$ is simple.
\begin{definition}
\label{def medir}
We denote by $\mostexp(g)\in\Pp(V)$ the associated singular direction, and refer to it as the  {\em $g$-most expanding direction}.
\end{definition}

By definition we have 
\begin{equation}\label{varphi g:mostexp proj}
\varphi_g \mostexp(g)= \mostexp(g^\ast)\;.
\end{equation}   
More generally, given $1\leq k\leq  \dim V$, 
if $\rgap_k(g)>1$
\begin{definition}
\label{def medir k}
 we define the
 {\em $g$-most expanding $k$-subspace}  to be
 $$  \mostexp_k(g):=\Psi^{-1}\left( \mostexp(\wedge_k g)  \right)\; ,$$
 where $\Psi$ stands for the Pl\"ucker embedding 
 defined in subsection~\ref{grassmannians}.
\end{definition}
The subspace $\mostexp_k(g)$ is the direct sum of all singular directions associated with the singular values 
$s_1(g),\ldots, s_k(g)$.
We have
\begin{equation}\label{varphi g:mostexp grassmann}
 \varphi_g \mostexp_k(g)= \mostexp_k(g^\ast)\;. 
\end{equation}

Analogously, let $n=\dim V$ and assume $\rgap_{n-k}(g)>1$.
\begin{definition}
\label{def ledir}

We define the {\em $g$-least expanding $k$-subspace} as
 $$ \leastexp_k(g) := \mostexp_{n-k}(g)^\perp \;.$$  
\end{definition}
The subspace $\leastexp_k(g)$ is the direct sum of all singular directions associated with the singular values 
$s_{n-k+1}(g),\ldots, s_n(g)$.
 Again we have
\begin{equation}\label{varphi g:leastexp grassmann}
 \varphi_g \leastexp_k(g)= \leastexp_k(g^\ast)\;. 
\end{equation}

\bigskip

Let $\tau=(\tau_1,\ldots,\tau_k)$ be a signature with $1\leq \tau_1<\ldots <\tau_k\leq \dim V$.
\begin{definition}
\label{def tau gap ratio}
We define the {\em $\tau$-gap ratio} of $g$ to be
$$ \rgap_\tau(g):= \min_{1\leq j\leq k} \rgap_{\tau_j}(g)\;. $$
When $\rgap_\tau(g)>1$ we say that $g$ {\em has a $\tau$-gap pattern}.
\end{definition}
Note that $\rgap_\tau(g)>1$ means that $g$ has a 
$\tau_j$ singular gap for $1\leq j\leq k$.
Recall that $\FF_\tau(V)$ denotes the space of all $\tau$-flags, i.e.,
flags $F=(F_1,\ldots, F_k)$ such that $\dim (F_j)=\tau_j$ for $j=1,\ldots, k$.

\begin{definition}
\label{def most expanding flag}
If  $\rgap_\tau(g)>1$   the {\em most expanding $\tau$-flag}  is   defined to be
$$  \mostexp_\tau(g):=( \mostexp_{\tau_1}(g),\ldots,  \mostexp_{\tau_k}(g))\in  \FF_\tau(V) \;. $$
\end{definition}

Given $g\in\mathcal{L}(V)$ the domain of its push-forward action on $\FF_\tau(V)$ is  
\begin{definition}
\label{def flag push-forward domain}\quad  
$\displaystyle  \FF_\tau(g):=\{\, F\in\FF_\tau(V)\,\colon\,
F_k\cap \Ker_g=0\,\}\;.$ 
\end{definition}
The {\em push-forward} of a flag $F\in \FF_\tau(g)$ by  $g$ is defined to be
$$ \varphi_g F = g\, F :=
(g\,F_1,\ldots, g\,F_k)\;. $$

\begin{proposition} Given $g\in\mathcal{L}(V)$ such that $\rgap_\tau(g)>1$, the push-forward induces a map
$\varphi_g: \FF_\tau(g)\to\FF_\tau(g^\ast)$ such that
 $ \varphi_g \mostexp_\tau(g)= \mostexp_\tau(g^\ast)$. 
\end{proposition}

\begin{proof}
Given $F\in\FF_\tau(g)$, we have 
$F_j\cap \Ker_g=0$ for all $j=1,\ldots, k$.
Hence $\dim g F_j= \dim F_j=\tau_j$ for all $j$,
which proves that $\varphi_g F \in \FF_\tau(V)$. To check that
$\varphi_g F \in \FF_\tau(g^\ast)$ we need to show that $g F_k\cap \Ker_{g^\ast}=0$.
Assume $g\,v\in \Ker_{g^\ast}$, with $v\in F_k$,
and let us see that $g\,v=0$. By assumption $g^\ast g\, v=0$, which implies $(g\,g^\ast)\,g\,v=0$.
Since the self-adjoint map $g\,g^\ast$ induces an automorphism on $\Range_g$, we conclude that $g\,v=0$.

The second statement follows from~\eqref{varphi g:mostexp grassmann}. 

\end{proof}

Given $g\in\mathcal{L}(V)$, the domain of its pull-back action on $\FF_\tau(V)$ is  
\begin{definition}
\label{def flag pull-back domain}
\quad  
$\displaystyle  \FF_\tau^{-1}(g):=\{\, F\in\FF_\tau(V)\,\colon\,
F_1 + \Range_g=V\,\}\;.$ 
\end{definition}
The {\em pull-back} of a flag $F\in \FF_\tau(g)$ by  $g$ is defined to be
$$ \varphi_g F = g^{-1} F :=
(g^{-1} F_1,\ldots, g^{-1} F_k)\;. $$


\begin{definition}
\label{def least expanding flag}
If  $\rgap_{\tau^\perp}(g)>1$   the {\em least expanding $\tau$-flag}  is  defined as
$$  \leastexp_\tau(g):=( \leastexp_{\tau_1}(g),\ldots,  \leastexp_{\tau_k}(g))\in  \FF_\tau(V) \;. $$
\end{definition}

\begin{proposition}
\label{prop relation least most expanding subspaces}
If  $\rgap_{\tau}(g)>1$ then \,
$\leastexp_{\tau^\perp} (g) =  
\mostexp_{\tau}(g)^\perp$.
\end{proposition}

\begin{proof}
Let $\{v_1,\ldots, v_n\}$ be a singular basis of $g$.
Since this basis is orthonormal,
\begin{align*}
\leastexp_{n-k}(g) =\langle v_{k+1},\ldots, v_n \rangle = \langle v_1,\ldots, v_k\rangle^\perp = \mostexp_k(g)^\perp\;.
\end{align*}
Hence
$$ \leastexp_{\tau^\perp}(g) =
( \leastexp_{n-\tau_k}(g),\ldots, \leastexp_{n-\tau_1}(g) ) = ( \mostexp_{\tau_1}(g),\ldots,
\mostexp_{\tau_k}(g) )^\perp = \mostexp_\tau(g)^\perp\;.$$

\end{proof}

\begin{proposition} Given $g\in\mathcal{L}(V)$ such that $\rgap_{\tau^\perp}(g)>1$, the pull-back induces a map
$\varphi_g^{-1}:\FF_\tau^{-1}(g)\to \FF_\tau^{-1}(g^\ast)$ such that  $ \varphi_g^{-1} \leastexp_\tau(g)= \leastexp_\tau(g^\ast)$. 
\end{proposition}

\begin{proof}
Given $F\in\FF_\tau^{-1}(g)$, we have 
$F_j + \Range_g=V$ for all $j=1,\ldots, k$.
Hence $\dim g^{-1}F_j= \dim F_j=\tau_j$ for all $j$,
which proves that $\varphi_g^{-1} F \in \FF_\tau(V)$. To check that
$\varphi_g^{-1} F \in \FF_\tau^{-1}(g^\ast)$ just notice that $ g^{-1} F_1+ \Range_{g^\ast} \supseteq
\Ker_g + \Ker_g^\perp = V$.

The second statement follows from~\eqref{varphi g:leastexp grassmann}
and proposition~\ref{prop relation least most expanding subspaces}. 

\end{proof}

We end this subsection proving that the orthogonal complement involution conjugates the push-forward action by
$g\in\mathcal{L}(V)$ with the pull-back action by the adjoint map $g^\ast$.

\begin{proposition} \label{pfwd-pbck:conjugation}
Given $g\in\mathcal{L}(V)$ such that $\rgap_{\tau^\perp}(g)>1$, 
the action of $\varphi_{g}^{-1}$ on $\FF_\tau(V)$ is
conjugate to the action of $\varphi_{g^\ast}$ on
$\FF_{\tau^\perp}(V)$ by the orthogonal complement involution. 
More precisely, we have
$\FF_{\tau}^{-1}(g)=\FF_{\tau^\perp}(g^\ast)^\perp$
and $\FF_{\tau}^{-1}(g^\ast)=\FF_{\tau^\perp}(g)^\perp$, and the following diagram commutes
$$ \begin{CD}
\FF_{\tau^\perp}(g^\ast)  @>\varphi_{g^\ast} >> \FF_{\tau^\perp}(g) \\
@V \cdot^\perp VV @VV \cdot^\perp V  \\
\FF_{\tau}^{-1}(g)   @>>\varphi_{g}^{-1}> \FF_{\tau}^{-1}(g^\ast)
\end{CD} \;. $$
\end{proposition}

\begin{proof}
To see that $\FF_{\tau}^{-1}(g)=\FF_{\tau^\perp}(g^\ast)^\perp$, notice that the following equivalences hold:
\begin{align*}
F\in \FF_{\tau}^{-1}(g)\quad & \Leftrightarrow \quad  F_1+\Range_g=V  \\
 & \Leftrightarrow \quad  F_1^\perp\cap \Ker_{g^\ast} = 0  \quad   \Leftrightarrow \quad 
 F^\perp\in \FF_{\tau^\perp}(g^\ast)\;.
\end{align*}
Exchanging the roles of $g$ and $g^\ast$ we obtain the relation
$\FF_{\tau}^{-1}(g^\ast)=\FF_{\tau^\perp}(g)^\perp$.

Finally, notice it is enough to prove the diagram's commutativity at the Grassmannian level. For that we use proposition~\ref{push-pull duality}.

\end{proof}

\subsection{Angles and expansion }
\label{subsection angles}

Throughout this subsection let $\hatp,\hatq\in \Pp(V)$, and $p\in\hatp$, $q\in \hatq$ denote   representative vectors.
The projective  distance $\delta(\hatp,\hatq)$ was defined by
$$ \delta(\hatp,\hatq):=\sqrt{1-\frac{\langle p,q\rangle^2}{\norm{p}^2\norm{q}^2}} = \frac{\norm{p\wedge q}}{\norm{p}\,\norm{q}}
=\sin \rho(\hatp,\hatq)  \;. $$

The complementary quantity plays a special role in the sequel.
\begin{definition}
\label{def aangle(p,q)}
The $\aangle$-angle between $\hat p$ and $\hat q$ is defined to be
$$  \aangle(\hatp,\hatq ):= \frac{\vert \langle p,q\rangle\vert }{\norm{p}\,\norm{q} } =\cos \rho(\hatp,\hatq)\;. $$
\end{definition}

In order to give a geometric meaning to this angle we define the {\em projective orthogonal complement} of  $\hatp\in\Pp(V)$ as
$$ \orthC{\hatp} :=\{\, \hat{x}\in\Pp(V) \,\colon\, \langle x, p\rangle=0\quad \text{ for }\; x\in\hat x \,\}\;. $$
The number $\aangle(\hatp,\hatq )$ is the sine of the minimum angle between
$\hatp$ and $\orthC{\hatq}$.
\begin{proposition}
For any $\hat p, \hat q\in \Pp(V)$,
\begin{align}\label{aangle:sine rho} 
&\aangle(\hatp,\hatq) =\sin \rho_{{\rm min}}(\hatp,\orthC{\hatq}) =\deltamin(\hatp,\orthC{\hatq})\\
\label{aangle:delta:orth}
& \aangle(\hatp,\hatq)=0\;  \Leftrightarrow \;
\delta(\hatp,\hatq)=1 \; \Leftrightarrow \; p\perp q\;.
\end{align}
\end{proposition}

These concepts extend naturally to Grassmannians and flag manifolds. 

\begin{definition}
\label{grassmann rho delta alpha}
Given $E,F\in\Gr_k(V)$, we define the $\aangle$-angle  between them
\begin{align*}
\aangle(E,F)= \aangle_k(E,F):=\aangle(\Psi(E),\Psi(F))\;,
\end{align*}
where $\Psi:\Gr_k(V)\to\Pp(\wedge_k V)$ denotes the Pl\"ucker embedding (see subsection ~\ref{grassmannians}). 
\end{definition}

\begin{definition}
\label{def Grassmann orth complement}
We say that two $k$-subspaces $E,F\in\Gr_k(V)$  {\em are orthogonal}, and we write $E\perp F$,\,
iff \, $\aangle(E,F)=0$.
\end{definition}

The {\em Grassmannian orthogonal complement} of $F$ is defined as
$$ \orthC{F}:=\{\, E\in \Gr_k(V)\,\colon\, \aangle(E,F)=0\,\}\;. $$

As before, the number $\aangle(E,F)$ measures the sine of the minimum angle between
$E$ and $\orthC{F}$.
\begin{proposition}
\label{aangle:sine rho:grassmann} 
For any $E, F\in \Gr_k(V)$,  
\begin{equation*}
\aangle(E,F)=\sin \rho_{{\rm min}}(E,\orthC{F}) =\deltamin(E,\orthC{F})\;.
\end{equation*}
\end{proposition}

Next we characterize  the angle $\aangle(E,F)$.
Consider the notation of definition~\ref{def orthog proj}.

\begin{proposition}\label{prop: alpha = det Pi(E F)}
Given $E,F\in\Gr_k(V)$,   
\begin{enumerate}
\item[(a)] $\aangle(E,F)= \aangle(E^\perp, F^\perp)$,
\item[(b)] $\aangle(E,F)= \abs{\det(\pi_{E,F})} = \abs{\det(\pi_{F,E})}$,
\item[(c)] $E\perp F$\, iff \,
 there exists a pair $(e,f)$ of unit vectors
such that $e\in E\cap F^\perp$ and $f\in F\cap E^\perp$,
\item[(d)] $\deltamin(E,F^\perp)\geq \aangle(E,F)$.
\end{enumerate}
\end{proposition}

\begin{proof} Given $E,F\in\Gr_k(V)$, take orthonormal
basis $\{u_1,\ldots, u_k\}$ and $\{v_1,\ldots, v_k\}$ of $E$ and $F$, respectively, and consider the associated unit $k$-vectors $u= u_1\wedge  \ldots \wedge u_k$ and
$v= v_1\wedge  \ldots \wedge v_k$, so that  $u\in \Psi(E)$ and $v\in\Psi(F)$.

Using the Hodge star operator we obtain unit vectors  $\ast u\in \Psi(E^\perp)$ and $\ast v\in \Psi(F^\perp)$.
Hence
$$\alpha(E^\perp,F^\perp)=\abs{\langle\ast  u, \ast v \rangle}
=\abs{\langle u,  v \rangle} =\alpha(E,F)\;, $$
which proves (a). Also
\begin{align*}
\aangle(E,F) & :=
  \abs{ \langle \, u_1\wedge  \ldots \wedge u_k, \,
v_1\wedge   \ldots \wedge v_k\, \rangle}\\
&=\abs{\det \left( \begin{array}{cccc}
\langle u_1,v_1\rangle & \langle u_1,v_2\rangle & \ldots & \langle u_1,v_k\rangle\\
\langle u_2,v_1\rangle & \langle u_2,v_2\rangle & \ldots & \langle u_2,v_k\rangle\\
\vdots & \vdots & \ddots & \vdots \\
\langle u_k,v_1\rangle & \langle u_k,v_2\rangle & \ldots & \langle u_k,v_k\rangle 
\end{array}\right)}\\
&= \abs{\det (\pi_{E,F})}
\;. 
\end{align*}

For the second equality write $u_i=w_i +\sum_{j=1}^k \langle u_i,v_j\rangle\,v_j$ with $w_i\in F^\perp$ and use the anti-symmetry of the exterior product. For the third equality remark that the matrix with entries $\langle u_i,v_j\rangle$ represents $\pi_{E,F}$ w.r.t. the given orthonormal basis.
By symmetry,  $\aangle(E,F)= \abs{\det(\pi_{F,E})}$. This proves (b).

From these relations,
$\aangle(E,F)=0$ \, $\Leftrightarrow$\, $\Ker (\pi_{E,F})\neq \{0\}$ \, $\Leftrightarrow$\, $\Ker (\pi_{F,E})\neq \{0\}$,
which explains (c).

By proposition~\ref{delta, deltamin, deltaH} (b), and because all singular values of $\pi_{E,F}$ are in  $[0,1]$,
$$ \deltamin(E,F^\perp) = \norm{\pi_{E,F}}
\geq \abs{\det( \pi_{E,F} ) } = 
\aangle_k(E,F) \;,$$
which proves (d).

\end{proof}

Finally, we extend $\aangle$-angle to flags.
Consider a signature $\tau$ of length $k$.

\begin{definition}
\label{def tau aangle}
Given flags $F,G\in 
 \FF_\tau(V)$, define
\begin{align*}
\aangle(F,G) = \aangle_\tau(F,G) &:=
\min_{1\leq j\leq k} \aangle(F_j,G_j)\;.
\end{align*}
\end{definition}

\begin{definition}
\label{def Flag orth complement}
We say that two $\tau$-flags $F,G\in\FF_\tau(V)$  {\em are orthogonal}, and we write $F\perp G$,\,
iff \, $F_j\perp G_j$ for some $j=1,\ldots, k$ .
\end{definition}

Comparing the two definitions, for all $F, G\in \FF_\tau(V)$
$$ \aangle_\tau(F,G)=0\quad \Leftrightarrow\quad  G\perp F  \;. $$
Hence, the {\em orthogonal
flag hyperplane} of $F$ is defined as
$$ \orthC{F}:=\{\, \Sigma(F):=\{\, G\in\FF_\tau(V)\,:\, 
 \aangle(G,F)=0\,\}\;. $$

As in the previous cases, the number $\aangle_\tau(F,G)$ measures the sine of the minimum angle between
$F$ and $\orthC{G}$.
\begin{proposition}
\label{aangle:sine rho:grassmann} 
For any $F, G\in \FF_\tau(V)$,  
\begin{equation*}
\aangle(E,F)=\sin \rho_{{\rm min}}(F,\orthC{G}) =\deltamin(F,\orthC{G})\;.
\end{equation*}
\end{proposition}

\medskip

Consider  a sequence of linear maps $g_0, g_1, \ldots, g_{n-1} \in \mathcal{L}(V)$. The following quantities, called {\em expansion rifts}, measure
the break of expansion in the composition 
$g_{n-1}\,\ldots\, g_1\, g_0$ of the  maps $g_j$.
\begin{definition}
\label{def rift}
The first  expansion rift of the sequence above is the number
$$\rift (g_0, g_1, \ldots, g_{n-1}) := \frac{\norm{g_{n-1} \ldots g_1 g_0}}{ \norm{g_{n-1}} \ldots \norm{g_1} \norm{g_0}}\in [1,+\infty)\,. $$
Given $1\leq k\leq \dim V$, the $k$-th expansion rift is
$$\rift_k (g_0, g_1, \ldots, g_{n-1}) := \rift (\wedge_k g_0, \wedge_k g_1, \ldots, \wedge_k g_{n-1})\,. $$
Given a signature $\tau=(\tau_1,\ldots, \tau_k)$,
 the $\tau$-expansion rift is defined as
$$\rift_\tau (g_0, g_1, \ldots, g_{n-1}) := \min_{1\leq j\leq k}
\rift_{\tau_j} (g_0, g_1, \ldots, g_{n-1})\,. $$
\end{definition}

\medskip

The key concept of this section is that of angle between linear maps. The quantity  $\aangle(g,g')$, for instance,
 is the sine of the angle between
$\varphi_g(\mostexp(g))=\mostexp(g^\ast)$ and $\orthC{\mostexp(g')}$.
As we will see, this angle is a lower bound on the expansion rift of two linear maps $g$ and $g'$.

\begin {definition}\label{alpha:def}
Given  $g,g'\in\mathcal{L}(V)$, we  define
\begin{align*}
\aangle (g,g')&:=\aangle (\mostexp(g^\ast),\mostexp(g'))
\quad  &\text{ if }\; g \; \text{ and }\; g' \; \text{ have a first gap ratio} \\
\aangle_k(g,g') &:=\aangle (\mostexp_{k}(g^\ast),\mostexp_{k}(g'))\quad &\text{ if }\; g \; \text{ and }\; g' \; \text{ have a }\; k \; \text{ gap ratio } \\
\aangle_\tau(g,g') &:=\aangle (\mostexp_{\tau}(g^\ast),\mostexp_{\tau}(g')) \quad &\text{ if }\; g \; \text{ and }\; g' \; \text{ have a }\; \tau \; \text{ gap pattern}.
\end{align*}
\end {definition}

\medskip

The following exotic operation is introduced to obtain an upper bound
on the expansion rift $\rift(g,g')$.
Consider the algebraic operation $ a\oplus b := a+b -a\,b $ on the set $[0,1]$.
Clearly
$([0,1],\oplus)$ is a commutative semigroup isomorphic to $([0,1],\cdot)$. In fact, the transformation $\Phi:([0,1],\oplus)\to ([0,1],\cdot)$, $\Phi(x):= 1-x$,
is a semigroup isomorphism. We summarize some properties of this
operation.

\begin{proposition} \label{oplus:prop}
For any $a,b,c\in [0,1]$,
\begin{enumerate}
\item[(1)] $0\oplus a = a$,
\item[(2)] $1\oplus a = 1$,
\item[(3)] $a\oplus b = (1-b)\,a+b = (1-a)\,b+a $,
\item[(4)] $a\oplus b <1$\; $\Leftrightarrow$\; $a<1$ and $b<1$,
\item[(5)] $a\leq b$ \; $\Rightarrow$\; $a\oplus c\leq b\oplus c$,
\item[(6)] $b>0$\; $\Rightarrow$\; 
$({a} {b}^{-1}\oplus c)\,b\leq a\oplus c$,
\item[(7)] $a\,c + b\,\sqrt{1-a^2}\,\,\sqrt{1-c^ 2}    \leq \sqrt{a^2 \oplus b^2}$.
\end{enumerate}
\end{proposition}

\begin{proof} Items (1)-(6) are left as exercises.
For the last item consider the function 
$f:[0,1]\to [0,1]$
defined by $f(c):= a\,c + b\,\sqrt{1-a^2}\,\,\sqrt{1-c^ 2}$.
A simple computation shows that
$$ f'(c)=a-\frac{b\,c\,\sqrt{1-a^2}}{\sqrt{1-c^2}}$$
The derivative $f'$  has a zero at $c= a/\sqrt{a\oplus b}$,
and one can check that this zero is a global maximum of $f$.
Since
$f( a/\sqrt{a\oplus b} )=\sqrt{a^2\oplus b^2}$,
item (7) follows. 
\end{proof}

\begin {definition}\label{beta:def}
Given $g,g'\in\mathcal{L}(V)$ with $\tau$-gap patterns,  
the  upper $\tau$-angle between $g$ and $g'$ is defined to be
\begin{equation*}  
\bangle_\tau(g,g'):=\sqrt{ \rgap_\tau(g)^{-2} \oplus \aangle_\tau(g,g')^ 2 \oplus \rgap_\tau(g')^{-2}}\;. 
\end{equation*}
We will  write  $\bangle_k(g,g')$  when $\tau=(k)$, and 
  $\bangle(g,g')$ when $\tau=(1)$.
\end{definition}

Next proposition relates norm expansion by $g$ and distance contraction by $\varphi_g$ with angles and gap ratios.

\begin{proposition}\label{prop expansion aangle}
Given $g\in\mathcal{L}(V)$ with
$\sgap(g)<1$,
a point $\hatw\in \Pp(V)$ and a unit vector $w\in\hatw$,
\begin{enumerate}
\item[(a)]\quad $\displaystyle \aangle(\hatw,\mostexp(g))\,\norm{g}  \leq \norm{g\,w} \leq \norm{g} \, \sqrt{  \aangle(\hatw,\mostexp(g))^2 \oplus \sigma(g)^{2}} $,
\smallskip

\item[(b)]\quad $\displaystyle \delta( \varphi_g(\hatw), \mostexp(g^\ast) )\leq  \frac{\sigma(g)}{\aangle(\hatw,\mostexp(g))}\,\delta(\hatw, \mostexp(g)) \,  $.

\end{enumerate}
\end{proposition}

\begin{proof}
Let us write $\alpha= \aangle(\hatw,\mostexp(g))$ and $\sigma= \sgap(g)$.
Take a unit vector $v\in \mostexp(g)$ 
such that $\angle(v,w)$ is non obtuse. Then $w = \alpha \, v + u$ with $u\perp v$ and $\norm{u}=\sqrt{1-\alpha^2}$. Choosing a unit vector
 $v^\ast\in\mostexp(g^\ast)$, we have $g w = \alpha\,\norm{g}\,v^\ast + g u$ with $g u\perp v^\ast$ and $\norm{g u}\leq \sqrt{1-\alpha^2} \,s_2(g)= \sqrt{1-\alpha^2} \,\sigma\,\norm{g}$.
We define the number $0\leq \kappa\leq \sigma$  so that
$\norm{g u}  = \sqrt{1-\alpha^2} \,\kappa\,\norm{g}$.
Hence
\begin{align*}
\alpha^2\,\norm{g}^2 &\leq \alpha^2\,\norm{g}^2 +
\norm{g u}^2 = \norm{g w}^2\;,
\end{align*}
and also
\begin{align*}
 \norm{g w}^2 & = \alpha^2\,\norm{g}^2 +
\norm{g u}^2 =
\norm{g}^2\,\left(\alpha^2 + (1-\alpha^2) \kappa^{2} \right)\\
&= \norm{g}^2 \, \left(  \alpha^2 \oplus \kappa^{2}\right)
\leq 
 \norm{g}^2 \, \left(  \alpha^2 \oplus \sigma^{2}\right) \;,
\end{align*}
which proves (a).

Item (b) follows from
\begin{align*}
\delta\left(  \varphi_g(\hatw), \mostexp(g^\ast)  \right) 
&= \frac{\norm{ g\,v \wedge g w}}{\norm{g v}\,\norm{g w}}
= \frac{\norm{ g\,v \wedge g u}}{\norm{g}\,\norm{g w}} = \frac{\norm{ v^\ast \wedge g u}}{ \norm{g w}} \\
& = \frac{\norm{ g u}}{ \norm{g w}}
\leq \frac{\sigma\,\sqrt{1-\alpha^2}\,\norm{g}}{\alpha\,\norm{g} } = \frac{\sigma\,\delta(\hatw, \mostexp(g))}{\alpha }\;.
\end{align*}

\end{proof}

Next proposition relates the expansion rift
$\rift(g,g')$ with the angle $\aangle(g,g')$
and the upper angle $\bangle(g,g')$.

\begin{proposition}\label{prod:2:lemma}
Given $g, g' \in\mathcal{L}(V)$ with a $(1)$-gap pattern,
$$ \aangle(g,g') \leq \frac{\norm{g'\,g}}{\norm{g'}\,\norm{g}}
\leq \bangle(g,g')   $$
\end{proposition}

\begin{proof}
Let $\alpha:=\aangle(g,g')=\aangle(\mostexp(g^\ast),\mostexp(g'))$ and take unit vectors $v\in\mostexp(g)$,  $v^\ast\in\mostexp(g^\ast)$
and  $v'\in\mostexp(g')$  such that
$\langle v^\ast, v'\rangle= \alpha >0$ and 
$g\,v=\norm{g}\,v^\ast$.

Since $\varphi_g(\mostexp(g))=\mostexp(g^\ast)$,
$w=\frac{g\,v}{\norm{g\,v}}$ is a unit vector in $\hatw = \mostexp(g^\ast)$. Hence, applying proposition~\ref{prop expansion aangle} (a) to $g'$ and $\hatw$, we  get
$$ \aangle(g,g')\,\norm{g'}
= \aangle(\hatw, \mostexp(g')) \,\norm{g'}
\leq \norm{ \frac{g'\,g\,v}{\norm{g\,v}} } \leq \frac{\norm{g'\,g}}{\norm{g}}\;, $$
which proves the first inequality.

For the second, consider $\hatw\in\Pp(g)$ and a unit vector $w\in \hatw$ such that
$a:=\langle w, v\rangle = \aangle(\hatw, \mostexp(g)) \geq 0$.
Then $w = a\, v+ \sqrt{1-a^2}\, u$,
where $u$ is a unit vector orthogonal to $v$. It follows that 
$g\, w = a\, \norm{g}\,v^\ast + \sqrt{1-a^2}\, g\,u$ with $g\,u\perp v^\ast$, and $\norm{g\, u} = \kappa\,\norm{g}$ for some  $0\leq \kappa\leq \sgap(g)$. Therefore
$$\frac{\norm{g\, w}^2}{\norm{g}^2} 
 = a^2 + (1-a^2)\,\kappa^2 =  a^2\oplus \kappa^2 \;. $$
and
$$ \frac{g\, w}{\norm{g\,w}}
= \frac{a}{\sqrt{a^2\oplus \kappa^2}}\, v^\ast
+ \frac{\sqrt{1-a^2}}{\sqrt{a^2\oplus \kappa^2}}\,\frac{g\,u}{\norm{g}}\;. $$
The vector  $v'$ can be written as
 $v'=	\alpha\, v^\ast + w'$ with $w'\perp v^\ast$ and $\norm{w'}=\sqrt{1-\alpha^2}$.
Set now $b:= \aangle( \varphi_g(\hatw), \mostexp(g'))$.  Then
\begin{align*}
 b = \abs{\langle 
 \frac{ g\,w}{\norm{g\, w}}, v' \rangle}   &\leq   \frac{\alpha\,a}{\sqrt{a^2\oplus\kappa^2}} 
  + \frac{\sqrt{1-a^2} }{\sqrt{a^2\oplus\kappa^2}}\,\frac{\abs{\langle g\,u,  v' \rangle}}{\norm{g}}
 \\
 &\leq \frac{\alpha\,a}{\sqrt{a^2\oplus\kappa^2}} 
 + \frac{\kappa \, \sqrt{1-a^2}}{\sqrt{a^2\oplus\kappa^2}}\,
 \abs{ \langle \frac{g\,u}{\norm{g\,u}},  w'  \rangle }\\
  &\leq \frac{\alpha\,a}{\sqrt{a^2\oplus\kappa^2}} 
 + \frac{\kappa \, \sqrt{1-a^2}}{\sqrt{a^2\oplus\kappa^2}}\,
 \norm{w'}\\
 &\leq \frac{\alpha\,a}{\sqrt{a^2\oplus\kappa^2}} 
 + \frac{\kappa \, \sqrt{1-a^2}\, \sqrt{1-\alpha^2} }{\sqrt{a^2\oplus\kappa^2}} \leq \frac{\sqrt{\alpha^2\oplus\kappa^2}}{\sqrt{a^2\oplus\kappa^2}}
  \;.
 \end{align*}
We use Lemma~\ref{oplus:prop} (7) on the last inequality.
Finally, by proposition~\ref{prop expansion aangle} (a)
\begin{align*}
\norm{ g'\,g\,w} &\leq \norm{g'}\,\sqrt{b^2\oplus \sgap(g')^2}\,\norm{g\,w}\\
&\leq \norm{g'}\,\norm{g}\,\sqrt{b^2\oplus \sgap(g')^2}\,\sqrt{a^2\oplus \kappa^2}\\
&\leq \norm{g'}\,\norm{g}\,\sqrt{\kappa^2\oplus \alpha^2\oplus \sgap(g')^2 }
\leq \bangle(g,g')\,\norm{g'}\,\norm{g} \;,
\end{align*} 
where on the two last inequalities use items (6) and (5)  of lemma~\ref{oplus:prop}.

\end{proof}

\begin{corollary} 
Given $g, g' \in\mathcal{L}(V)$ with a $(k)$-gap pattern,
$$ \aangle_k(g,g') \leq \frac{\norm{ \wedge_k (g'\, g) }}{\norm{\wedge_k g'}\,\norm{\wedge_k  g}}
\leq \bangle_k(g,g')   $$
\end{corollary}

\begin{proof} 
Apply proposition~\ref{prod:2:lemma} to the composition 
$(\wedge_k g')\,(\wedge_k g)$. Notice that by definition~\ref{def medir k}, the Pl\"ucker embedding satisfies $\Psi(\mostexp_{k}(g) )= \mostexp(\wedge_k g)$. Hence
$$ \aangle_k(g,g') =\aangle(\mostexp_{k}(g^\ast),\mostexp_{k}(g'))
= \abs{\langle \mostexp(\wedge_k g), \mostexp(\wedge_k g') \rangle} =\aangle (\wedge_k g, \wedge_kg')\;.
$$

\end{proof}

Next lemmas show how close the   bounds
$\aangle(g,g')$ and $\bangle(g,g')$ are, to each other, and to the rift $\rift(g,g')$.

\begin{lemma} \label{alpha:beta:bound}
Given   $g,g'\in\mathcal{L}(V)$ with $(1)$-gap patterns,
$$ 1\leq \frac{\beta (g,g')}{\aangle (g,g')}\leq \sqrt{ 1+
\frac{\rgap (g)^{-2}\oplus \rgap (g')^{-2}}{ \aangle (g,g')^2} }\;. $$
\end{lemma}

\begin{proof} 
Just notice that
$$ \frac{\sqrt{\kappa^2\oplus \alpha^2\oplus (\kappa')^2}}{\alpha} \leq 
\sqrt{\frac{\alpha^2 +  (\kappa^2\oplus  (\kappa')^2) }{\alpha^2}}
= \sqrt{1 + \frac{ \kappa^2\oplus  (\kappa')^2  }{\alpha^2}} \;.$$

\end{proof}

\begin{proposition}
\label{prop angle rift}
Given $g, g' \in \mathcal{L}(V)$ with a $(1)$-gap pattern
$$ \aangle(g,g')\geq \rift(g,g')\,\sqrt{1-\frac{\rgap(g)^{-2} + \rgap(g')^{-2} }{\rift(g,g')^2} } \;.$$
\end{proposition}

\begin{proof}
By proposition~\ref{prod:2:lemma}

$$\rift(g,g')^2\leq \bangle(g,g')^2\leq 
\aangle(g,g')^2 + \sgap(g)^2 + \sgap(g')^2 \;,$$
which implies the claimed inequality.

\end{proof}

These inequalities then imply the following more general  fact.

\begin{proposition} \label{svp:lemma:norm}
Given $g_0, g_1,\ldots, g_{n-1}\in\mathcal{L}(V)$, if for all $1\leq i\leq n-1$ the linear maps
$g_i$ and $g^{(i)}= g_{i-1}\ldots g_0$
have  $(1)$-gap patterns, then
$$ \prod_{i=1}^{n-1}
 \aangle (g^{(i)},g_i)  
\leq \frac{\norm{ g_{n-1}\ldots g_1 g_0} }{
\norm{g_{n-1}}\ldots  \norm{g_{1}}
\norm{g_{0}}} \leq \prod_{i=1}^{n-1}
\bangle (g^{(i)},g_i)  $$
\end{proposition}

\begin{proof}
By definition
$g^{(n-1)}=g_{n-1}\ldots g_1 g_0$,  and by convention   $g^{(0)}={\rm id}_V$.
Hence  $\norm{ g_{n-1}\ldots g_1 g_0} = \prod_{i=0}^{n-1}\frac{\norm{g^{(i+1)}}}{\norm{g^{(i)}}}$. This implies that
\begin{align*}
\frac{\norm{ g_{n-1}\ldots g_1 g_0}}{\norm{g_{n-1}}\ldots  \norm{g_{1}}}
&=\left(\prod_{i=0}^{n-1}\frac{1}{\norm{g_i}}\right)\,
\left(\prod_{i=0}^{n-1}\frac{\norm{g^{(i+1)}}}{\norm{g^{(i)}}}\right)\\
&= \prod_{i=0}^{n-1}\frac{\norm{g_i\,g^{(i)}}}{\norm{g_i}\,\norm{g^{(i)}}}\;.
\end{align*}
It is now enough to apply proposition~\ref{prod:2:lemma} to each factor.
\end{proof}

\bigskip

\section{Lipschitz Estimates}
\label{le}

\newcommand{\hide}[1]{}
\newcommand{\drel}{d_{{\rm rel}} }

In this section we will derive some
inequalities describing quantities such as the contracting behaviour of a linear endomorphism on the projective space, 
the Lipschitz dependence of a projective action on the acting linear endomorphism, the continuity of most expanding directions as functions of a linear map, and the Lipschitz modulus of continuity for sum and intersection operations on flag manifolds.

\subsection{  Projective action }

\begin{proposition}\label{proj:lip}
Given $p,q\in V\setminus\{0\}$,
$$ \norm{\frac{p}{\norm{p}} - \frac{q}{\norm{q}} } \leq
\max\{ \frac{1}{\norm{p}}, \frac{1}{\norm{q}} \}\,
 \norm{p-q} \;.$$
\end{proposition}

\begin{proof}
Given to vectors $u,v\in V$ with $\norm{u}\geq \norm{v}=1$
we have
$$\norm{ \frac{u}{\norm{u}} - \frac{v}{\norm{v}} }\leq 
\norm{u-v} \;. $$
Assume for instance that $\norm{p}\geq \norm{q}$,
so that
$$\max\{ \norm{p}^{-1}, \norm{q}^{-1} \} = \norm{q}^{-1}\;.$$
Applying the previous inequality with
$u= \frac{p}{\norm{q}}$ and
$v = \frac{q}{\norm{q}}$, we get

\begin{align*} 
\norm{ \frac{p}{\norm{p}} -
\frac{q}{\norm{q}} } & = \norm{ \frac{u}{\norm{u}} - \frac{v}{\norm{v}} }
 \leq \norm{u-v} = \norm{\frac{p}{\norm{q}}-
\frac{q}{\norm{q}}} \\
 & =
 \norm{q}^{-1} \,
 \norm{p-q}  =
\max\{ \norm{p}^{-1}, \norm{q}^{-1} \}\,
 \norm{p-q}\;.
\end{align*} 
\end{proof}

Given a linear map $g\in\mathcal{L}(V)$, the projective action of $g$ is given by the map
$\varphi_g:\Pp(g)\to\Pp(g^\ast)$,
$\varphi_g(\hatp):=\widehat{g\,p}$.

For any non collinear vectors $p,q\in V$ with $\norm{p}=\norm{q}=1$,
define
$$v_{p}(q):= \frac{ q-  \langle p, q\rangle \,p}{\norm{q-  \langle p, q\rangle \,p}} $$ 
to be the versor of the orthogonal projection
of $q$ onto  $p^\perp$.

\begin{proposition} \label{Lip:proj:action}
Given $g\in\mathcal{L}(V)$, and points $\hatp\neq \hatq$ in $\Pp(V)$,
$$ \frac{\delta(\varphi_g(\hatp), \varphi_g(\hatq))}{\delta(\hatp,\hatq)} = \frac{\norm{g p \wedge g v_p(q)}}{\norm{g\,p}\,\norm{g\,q}}\;. $$
\end{proposition}

\begin{proof}
Let $p\in \hatp$ and $q\in\hatq$ be unit vectors such that
$\theta=\angle(p,q)\in [0,\frac{\pi}{2}]$.
We can write $q= (\cos\theta)\,p +  (\sin\theta)\,v_p(q)$.
Hence
$$ \delta(\hatp,\hatq)=\norm{p\wedge q}= (\sin\theta)\,\norm{p\wedge v_p(q)}= \sin\theta\;,$$
and 
$$ \delta(\varphi_g(\hatp), \varphi_g(\hatq))=
\frac{\norm{ g\,p \wedge g\,q} }{\norm{g\,p}\,\norm{g\,q} } =
(\sin\theta)\,\frac{\norm{g p\wedge g v_p(q)}}{\norm{g\,p}\,\norm{g\,q} } \;.$$
 
\end{proof}

Given a  unit vector $v\in V$, $\norm{v}=1$,
denote by $\pi_v,\pi_v^\perp:V\to V$
the orthogonal projections $\pi_v(x):=\langle v,x\rangle\, v$,
respectively  $\pi_v^ \perp(x):=x -\langle v,x\rangle\, v$.

\begin{lemma}\label{diff:proj}
Given  $u,v\in V$  non-collinear  with $\norm{u}=\norm{v}=1$,
denote by $P$ the plane spanned by $u$ and $v$. Then
\begin{enumerate}
\item[(a)] is $\pi_v-\pi_u$ a self-adjoint endomorphism,
\item[(b)]  $\Ker (\pi_v-\pi_u)= P^ \perp$,
\item[(c)] the restriction $\pi_v-\pi_u:P\to P$ is anti-conformal
with similarity factor $\abs{\sin\angle (u,v)}$,
\item[(d)] $\norm{\pi^\perp_v-\pi^\perp_u}=\norm{\pi_v-\pi_u}\leq \norm{v-u}$.
\end{enumerate}
\end{lemma}

\begin{proof}
Item  (a) follows because orthogonal projections are self-adjoint operators.

Given $w\in P^\perp$, we have $\pi_u(w)=\pi_v(w)=0$, which implies
$w\in \Ker (\pi_u-\pi_v)$. Hence $P^\perp \subset  \Ker (\pi_u-\pi_v)$. Since $u$ and $v$ are non-collinear, 
$\pi_u-\pi_v$ has rank $2$. Thus  $\Ker (\pi_u-\pi_v) = P^\perp$, which proves (b).

For (c) we may assume that
$V=\R^2$  and consider $u=(u_1,u_2)$, $v=(v_1,v_2)$, with $u_1^ 2+u_2^ 2=v_1^ 2+v_2^ 2=1$.
The projections $\pi_u$ and $\pi_v$ are represented by the matrices
$$ U= \left(\begin{array}{cc}
u_1^2 & u_1 u_2 \\
u_1 u_2 & u_2^ 2
\end{array} \right)\quad \text{ and }\quad
V = \left(\begin{array}{cc}
v_1^2 & v_1 v_2 \\
v_1 v_2 & v_2^ 2
\end{array} \right)$$
w.r.t. the canonical basis. Hence  $\pi_v-\pi_u$ is given by
$$ V-U= \left(\begin{array}{cc}
v_1^2-u_1^2 & v_1 v_2 - u_1 u_2 \\
v_1 v_2 - u_1 u_2 & v_2^ 2 - u_2^ 2
\end{array} \right) = \left(\begin{array}{cc}
\beta & \alpha \\
\alpha & -\beta 
\end{array} \right)$$
where $\alpha=v_1 v_2 - u_1 u_2$ and $\beta= v_1^2-u_1^2 = -(v_2^ 2 - u_2^ 2)$.
This proves that the restriction of $\pi_v-\pi_u$ to the plane $P$ is
anti-conformal. The similarity factor of this map is
$$ \norm{\pi_v-\pi_u} = \norm{\pi_v(u)-u}=\norm{\pi_v^\perp(u) }
=\abs{\sin \angle(u,v)}$$
Finally, since $u - \langle v,u\rangle\,v\perp v$,
\begin{align*}
  \norm{\pi^\perp_v-\pi^\perp_u}^2 &=\norm{\pi_v-\pi_u}^2 = \norm{\pi_v^\perp(u) }^2\\
&= \norm{ u - \langle v,u\rangle\,v}^2\\
&= \norm{u-v}^2- \norm{ v - \langle v,u\rangle\,v}^2\\
&\leq  \norm{u-v}^2\;. 
\end{align*}

\end{proof}

Given a point $\hatp\in\Pp(V)$, we identify the tangent to the projective space at $\hatp$ as  $T_{\hatp}\Pp(V)=p^\perp$,
for any representative $p\in\hatp$.

\begin{proposition}\label{derivative varphig}
Given $g\in \mathcal{L}(V)$, $\hat{x} \in\Pp(g)$,
and a representative $x\in\hat{x}$, 
the derivative of the map
$\varphi_g:\Pp(g)\to \Pp(g^\ast)$ at $\hat{x}$  is given by 
$$ (D\varphi_g)_{\hat{x}}\, v = 
\frac{g\,v-\langle \frac{g\,x}{\norm{g\,x}}, \, g\,v\rangle\, \frac{g\,x}{\norm{g\,x}} }{\norm{g\,x}} =
 \frac{1}{\norm{g\,x}}\, \pi_{ g x/\norm{g x} }^ \perp (g\,v) $$
\end{proposition}

\begin{proof}
The sphere $\Su(V):=\{\, v\in V\,\colon\,
\norm{v}=1\,\}$ is a double covering space of  $\Pp(V)$, whose covering map is the canonical projection $\proj:\Su(V)\to\Pp(V)$.
With the identification $T_{\hatp}\Pp(V)=p^\perp$, the derivative of $\proj$, $D{\proj}_x:T_x\Su(V)\to T_{\hat{x}}\Pp(V)$, is the identity linear map.
The map $\varphi_g$  lifts
to the map defined on the sphere  by
$\widetilde{\varphi}_g(x):=\frac{g\,x}{\norm{g\,x}}$.
Hence we can identify the derivatives $(D\varphi_g)_{\hat{x}}$ and
$(D\widetilde{\varphi}_g)_x$. The explicit expression for
$(D\widetilde{\varphi}_g)_x v$ follows by a simple calculation.

\end{proof}

We will use the following closed ball notation
 $$B^{(d)}(\hatp,r):=\{\, \hat{x}\in\Pp(V)\,\colon\,
d(\hat{x},\hatp)\leq r\,\} \;, $$
where the superscript emphasizes the distance in matter.
Given a projective map $f:X\subset \Pp(V)\to\Pp(V)$,
we denote by ${\rm Lip}_d(f)$ the least Lipschitz 
constant of $f$ with respect to the distance $d$.
Next proposition refers to the projective metrics
$\delta$ and $\rho$ defined in subsection~\ref{subsection projective spaces}.

\begin{proposition} \label{proj:contr} Given  $0<\kappa<1$ and $g\in\mathcal{L}(V)$ such that $\rgap(g)\geq \kappa^{-1}$,

\begin{enumerate}
\item[(1)] $\varphi_g\left( B^{(\delta)}(\mostexp(g),r) \right)\subset B^{(\delta)}(\mostexp(g^\ast), \kappa\,r/\sqrt{1-r^ 2})$, \, for any $0<r <1$,

\item[(2)] $\varphi_g\left( B^{(\rho)}(\mostexp(g),a) \right)\subset B^{(\rho)}(\mostexp(g^\ast), \kappa\,\tan a)$,  \, for any $0<a <\frac{\pi}{2}$,

\item[(3)] ${\rm Lip}_\rho( \varphi_g\vert_{ B^{(\delta)}(\mostexp(g),r)} ) \leq 
\kappa\, \frac{ r+\sqrt{1-r^2} }{1-r^2}$,  \, for any $0<r <1$.
\end{enumerate}
\end{proposition}

\begin{proof}
Item (1) of this proposition follows from 
proposition~\ref{prop expansion aangle} (b), because   
$$ \delta(\hatw,\mostexp(g))<r \quad \text{ implies} \quad
\aangle(\hatw,\mostexp(g))=\sqrt{1-\delta(\hatw,\mostexp(g))^2}
\geq \sqrt{1-r^2}\;.$$

Item (2) reduces to (1), because we have 
$\delta(\hatu,\hatv)  =\sin \rho(\hatu,\hatv)$, which implies that  $B^{(\rho)}(\hatv, a)= B^{(\delta)}(\hatv, \sin a)$.

To prove (3), take unit vectors $v\in \mostexp(g)$ and $v^\ast\in \mostexp(g^\ast)$ 
such that $g\,v=\norm{g}\,v^\ast$.
Because $v$ is a $g$-most expanding vector,  
$\norm{\pi^\perp_{v^\ast}\circ g} =\norm{g\circ \pi_v^\perp}\leq s_2(g)\leq \kappa\,\norm{g}$.
Given $\hat x$ such that $\delta(\hat x, \mostexp(g))<r$,
and a unit vector $x\in\hat x$, by proposition~\ref{prop expansion aangle} (a)  
$$ \frac{\norm{g}}{\norm{g x}}\leq \frac{1}{\aangle(\hat x,\mostexp(g))} \leq \frac{1}{\sqrt{1-r^2}}\;. $$
Using item (b) of the same proposition we get
$$ \norm{ \widetilde{\varphi}_g(x) -v^\ast} \leq
\delta(\varphi_g(\hat x), \mostexp(g^\ast))\leq\frac{\sgap(g)}{\aangle( \hat x, \mostexp(g))}\,\delta( \hat x, \mostexp(g))
\leq \frac{\kappa\, r}{\sqrt{1-r^2}}
$$
By proposition~\ref{derivative varphig} we have
$$ (D\varphi_g)_x\, v=\frac{1}{\norm{g x}}\,
\pi_{v^\ast}^\perp(g\,v) + \frac{1}{\norm{g x}}\left( \pi_{\widetilde{\varphi}_g(x)}^\perp - \pi_{v^\ast}^\perp\right)(g\,v)\;.
$$
Thus, by lemma~\ref{diff:proj} (d), 
\begin{align*}
\norm{ (D\varphi_g)_x } &\leq
\frac{\kappa\,\norm{g}}{\norm{g x}}
 + \frac{\norm{ \widetilde{\varphi}_g(x) - v^\ast }\,\norm{g}}{\norm{g x}} \\
 &\leq
\frac{\kappa}{\sqrt{1-r^2}}
 + \frac{\kappa\,r }{1-r^2} 
 = \frac{\kappa\,(r+\sqrt{1-r^2}) }{1-r^2}   \;.
\end{align*}
Since $B^{(\delta)}(\mostexp(g),r)$ is a convex Riemannian  disk,
by the mean value theorem $\varphi_g\vert_ {B^{(\delta)}(\mostexp(g),r)}$ has Lipschitz constant
$\leq \frac{\kappa\,(r+\sqrt{1-r^2}) }{1-r^2}$ with respect to  distance $\rho$.  
\end{proof}

\bigskip

\subsection{ Operations on flag manifolds }

As before let $V$ be a $n$-dimensional Euclidean space.
Recall that the Grassmann manifold $\Gr_k(V)$ identifies through the Pl\"ucker embedding with a submanifold of $\mathbb{P}(\wedge_k V)$. Up to a sign,  $E\in \Gr_k(V)$
is identified with the unit $k$-vector $e = e_1\wedge \ldots \wedge e_k$
associated to any orthonormal basis $\{e_1,\ldots, e_k\}$ of $E$.
Recall that the Grassmann distance~\eqref{Grassmann:distance}  on $\Gr_k(V)$  can be characterized by
$$ d(E_1, E_2):= \min\{ \norm{e_1-e_2}, \norm{e_1+e_2}\}\; , $$
where $e_j$  is a unit $k$-vector  of $E_j$, for $j=1,2$.

\begin{definition}
Given  $E,F\in\Gr(V)$, we say that $E$ and $F$ are
$(\cap)$ transversal \, iff\,  $E+F=V$.
Analogously, we say  that $E$ and $F$ are
$(+)$ transversal \, iff\,  $E\cap F =\{0\}$.
\end{definition}

The following numbers quantify the transversality of two linear subspaces.

\begin{definition} Given $E\in \Gr_r(V)$ and $F\in \Gr_s(V)$,
consider a unit $r$-vector $e$ of $E$,
 a unit $s$-vector $f$ of $F$, a unit $(n-r)$-vector $e^\perp$ of $E^\perp$,
and  a unit $(n-s)$-vector $f^\perp$ of $F^\perp$.
We define
\begin{align*}
 \theta_+(E,F) &:= \norm{e \wedge f}\;,  \\
 \theta_\cap(E,F) &:= \norm{e^\perp \wedge f^\perp } \;.
\end{align*}
Since the chosen unit vectors are unique up to a sign, these quantities are well-defined.
\end{definition}

\begin{remark}
\label{rmk dim constraint}
If $r + s>n$ then $\theta_+(E,F)=0$.
Similarly, if $r + s< n$ then  $\theta_\cap(E,F)=0$.
\end{remark}

\begin{remark}
\label{rmk rel between transversalities}
Given $E,F\in\Gr(V)$,\, 
$ \theta_\cap(E,F) =  \theta_+(E^\perp,F^\perp)$.
\end{remark}
 
Next proposition establishes a Lispchitz modulus of continuity
for the sum and intersection operations on Grassmannians in terms of
the previous quantities.

\begin{proposition} 
\label{sum:inters:modulus cont}
Given $r,s\in\N$ and  $E,E'\in\Gr_r(V)$, $F,F'\in\Gr_s(V)$,
\begin{enumerate}
\item[(1)]\; $\displaystyle  d(E + F, E' + F')  \leq \max\left\{  \frac{1}{\theta_+ (E,F)},
 \frac{1}{\theta_+ (E',F')} \right\} \,( d(E,E') + d(F,F') ) \;,$ 
 
\smallskip

\item[(2)]\; $\displaystyle d(E \cap F, E' \cap  F')  \leq \max\left\{  \frac{1}{\theta_\cap (E,F)},
 \frac{1}{\theta_\cap (E',F')} \right\} \,( d(E,E') + d(F,F') ) \;.$ 
\end{enumerate}
\end{proposition}

\begin{proof} (1)\; Consider unit $r$-vectors $e$ and $e'$ representing the 
subspaces $E$ and $E'$  respectively. Consider also 
 unit $s$-vectors $f$ and $f'$ representing the 
subspaces $F$ and $F'$  respectively.
By Proposition~\ref{proj:lip}
\begin{align*}
d(E+F,E'+F') &= \norm{ \frac{e\wedge f}{\norm{e\wedge f}} 
- \frac{e'\wedge f'}{\norm{e'\wedge f'}}  } \\
&\leq K \, \norm{ e\wedge f - e'\wedge f' }\\
&\leq K \, ( \norm{ e\wedge (f-f') } +
\norm{ (e - e')\wedge f' } )\\
&\leq K \, ( \norm{ e- e' } +
\norm{f -  f' } )\\
\end{align*}
where 
$K = \max\{\norm{e\wedge f}^{-1},
\norm{e'\wedge f'}^{-1} \} =
\max\{\theta_{+}(E,F)^{-1},
\max\{\theta_{+}(E',F')^{-1} \}$.

\noindent
 (2)  reduces to (1) by duality (see Proposition~\ref{+cap duality}).

\end{proof}

Next proposition gives an  alternative  characterization  of  the transversality measurements $\theta_+(E,F)$ and $\theta_\cap(E,F)$.
 Let, as before, $\pi_E:V\to E$ denote the orthogonal projection  onto a subspace $E\subset V$,
and define   the restriction   $\pi_{E,F}:= \pi_F\vert_E:E\to F$.

\begin{proposition} 
\label{prop theta continuity modulus}
Given $E\in\Gr_r(V)$ and $F\in\Gr_s(V)$,
\begin{enumerate}
\item[(1)]   \; $\theta_+(E,F) =  \abs{\det (\pi_{E,F^\perp})} =  \abs{\det (\pi_{F,E^\perp})}$.
 
\item[(2)]   \; $\theta_\cap(E,F) =  \abs{\det (\pi_{E^\perp,F})} =  \abs{\det (\pi_{F^\perp,E})}$.
\end{enumerate}
\end{proposition}

\begin{proof}
Notice that $E\cap F = \Ker (\pi_{E,F^\perp})=  \Ker(\pi_{F,E^\perp})$.
If $E\cap F\neq \emptyset$ then the three terms in (1) vanish. Otherwise $\pi_{E,F^\perp}$ and $\pi_{F,E^\perp}$ are isomorphisms.
Take an orthonormal basis $\{f_1,\ldots, f_s,  f_{s+1},\ldots, f_{s+r},\ldots, f_n\}$
such that $\{ f_1,\ldots, f_s\}$ spans $F$ and the family of vectors
$\{f_1,\ldots, f_r, f_{s+1},\ldots, f_{s+r}\}$ spans $E+F$.
Consider the unit $s$-vector
$f=f_1\wedge\ldots \wedge f_s$ of $F$,
and  a unit $r$-vector $e=e_1\wedge \ldots \wedge e_r$ of $E$. Then
\begin{align*}
\theta_+(E,F) &= \norm{(e_1\wedge \ldots \wedge e_r)\wedge (f_1\wedge \ldots \wedge f_s)}\\
 &= \norm{\pi_{E, F^\perp}(e_1)\wedge \ldots \wedge  \pi_{E, F^\perp}(e_r)\wedge  f_1\wedge \ldots \wedge f_s }\\
 &= \abs{\det (\pi_{E, F^\perp})}\, \norm{f_{s+1}\wedge \ldots \wedge f_{s+r}\wedge  f_1\wedge \ldots \wedge f_s }
 = \abs{\det (\pi_{E, F^\perp})}\;.
\end{align*}
Reversing the roles of $E$ and $F$, and because $\norm{e\wedge f}$ is symmetric in $e$ and $f$,
we obtain $\theta_+(E,F) = \abs{\det \pi_{F, E^\perp}}$, which proves (1).

By  duality and remark~\ref{rmk rel between transversalities}, item (2) reduces to (1). 
\end{proof}

The measurement on the $(\cap)$ transversality admits the following lower bound in terms of the  angle  in definition~\ref{grassmann rho delta alpha}. 

\begin{proposition}
Given $E\in \Gr_r(V)$  and $F\in \Gr_s(V)$,
if $E+F=V$ then
$$ \theta_{\cap}(E,F)\geq \aangle_r(E, E\cap F + F^\perp)\;.$$
\end{proposition}

\begin{proof}
Combining lemmas~\ref{theta:monot} and ~\ref{theta aangle} we have
\begin{align*}
\theta_{\cap}(E,F) & \geq \theta_{\cap}(E,F\cap(E\cap F)^\perp) = \aangle_r(E, (F\cap(E\cap F)^\perp)^\perp)\\
&= \aangle_r(E,  (E\cap F) + F^\perp )\;. 
\end{align*}

\end{proof}

\begin{lemma}\label{theta:monot}
Given $E\in \Gr_r(V)$, $E'\in \Gr_{r'}(V)$ and $F\in \Gr_s(V)$
such that $r+s\geq n$ and $E\subseteq E'$  then
$\displaystyle \theta_{\cap}(E',F)\geq  \theta_{\cap}(E,F)$.
\end{lemma}

\begin{proof} Because $E\subset E'$, we have
$\pi_{F^ \perp, E} = \pi_{E', E}\circ \pi_{F^ \perp, E'}$. Hence
\begin{align*}
\theta_{\cap}(E,F) &= \abs{\det(\pi_{F^ \perp, E}) } = \abs{\det(\pi_{\pi_{E'}(F^ \perp), E})}\,\abs{\det( \pi_{F^ \perp, E'})} \\
&\leq \abs{\det( \pi_{F^ \perp, E'})}  = \theta_{\wedge}(E',F) \;,
\end{align*}
where $\abs{\det(\pi_{\pi_{E'}(F^ \perp), E})} \leq 1$ because   $\norm{ \pi_{E} } \leq 1$.  
\end{proof}

\begin{lemma}\label{theta aangle}
Given $E, E'\in \Gr_r(V)$, \;
$\displaystyle \theta_{\cap}(E',E^ \perp) =  \alpha_r(E',E)$.
\end{lemma}

\begin{proof}
Take orthonormal basis $\{v_1,\ldots, v_r\}$ of $E$, and 
$\{v_1',\ldots, v_r'\}$ of $E'$. Then
\begin{align*}
\theta_{\cap}(E',E^\perp) &= \abs{ \det (\pi_{E',E} ) } \\
&= \abs{\langle  \wedge_r \pi_{E,E'} (v_1 \wedge \ldots \wedge v_r), v_1'\wedge \ldots \wedge v_r'\rangle } \\
&= \abs{\langle \pi_{E'}(v_1)\wedge \ldots \wedge \pi_{E'}(v_r), v_1'\wedge \ldots \wedge v_r'\rangle } \\
&= \abs{\langle v_1\wedge \ldots \wedge v_r, v_1'\wedge \ldots \wedge v_r'\rangle }
=\alpha_r(E,E')\;.
\end{align*}
 
\end{proof}

Next proposition gives a modulus of lower semi-continuity for the   transversality measurement $\theta_{\cap}$.

\begin{proposition}\label{theta cap: E->E0}
Given $E, E_0\in \Gr_r(V)$ and $F, F_0\in \Gr_s(V)$,
$$\theta_{\cap}(E,F)\geq \theta_{\cap}(E_0,F_0)
- d(E,E_0)-d(F,F_0)\;.$$
\end{proposition}

\begin{proof}
Consider unit vectors
$e\in\Psi(E^\perp)$, $f\in\Psi(F^\perp)$,
$e_0\in\Psi(E_0^\perp)$ and $f_0\in\Psi(F_0^\perp)$,
chosen so that 
\begin{align*}
d(E,E_0) &=d(E^\perp,E_0^\perp)=\norm{e-e_0}\;,\\ 
d(F,F_0) &=d(F^\perp,F_0^\perp)=\norm{f-f_0}\;.
\end{align*}
Hence
\begin{align*}
\theta_{\cap}(E,F) &= \norm{e\wedge f} \geq 
\norm{e_0\wedge f_0} -\norm{e\wedge f - e_0\wedge f_0}\\
&\geq \theta_{\cap}(E_0,F_0) - 
\norm{ e\wedge (f-f_0)} -\norm{(e-e_0)\wedge f_0}\\
&\geq \theta_{\cap}(E_0,F_0) - 
\norm{f-f_0} -\norm{e-e_0}\\
&\geq \theta_{\cap}(E_0,F_0) - 
d(F,F_0) -d(E,E_0)\;.
\end{align*}

\end{proof}

The exterior product is  a continuous operation.
A lower bound on its modulus of continuity can be expressed in terms of the angle between the arguments.

\begin{proposition}\label{prop norm{u wedge v} >= norm{u} norm{v}}
Given $E,F\in\Gr_k(V)$, and  families of vectors
$\{u_1,\ldots, u_k\}\subset E$ and $\{u_{k+1},\ldots, u_{k+i}\}\subset  F^\perp$  with $1\leq i\leq m-k$,
\begin{enumerate}
\item[(a)] $\displaystyle \norm{u_1\wedge \ldots \wedge u_k\wedge u_{k+1}\wedge \ldots \wedge u_{k+i}} \leq  \norm{u_1\wedge \ldots \wedge u_k}\,\norm{u_{k+1}\wedge \ldots \wedge u_{k+i}}$,
\item[(b)]  $\displaystyle  \norm{u_1\wedge \ldots \wedge u_k\wedge u_{k+1}\wedge \ldots \wedge u_{k+i}} \geq \aangle(E,F)\,\norm{u_1\wedge \ldots \wedge u_k}\,\norm{u_{k+1}\wedge \ldots \wedge u_{k+i}}$.
\end{enumerate}
\end{proposition}

\begin{proof}
Since $\pi_{F^\perp,E^\perp}$ is an orthogonal projection, all its singular values  are in  $[0,1]$.
Thus, because $\abs{\det (\pi_{F^\perp,E^\perp})}$ is the product of all singular values, while $\minexp(\wedge_i \, \pi_{F^\perp,E^\perp})$ is the product of the 
$i$ smallest singular values, we have    
$$\abs{\det (\pi_{F^\perp,E^\perp})}
\leq  \minexp(\wedge_i \, \pi_{F^\perp,E^\perp})
\leq \norm{ \wedge_i \, \pi_{F^\perp,E^\perp} }\leq 1\;.$$
Hence
\begin{align*}
\norm{u_1\wedge \ldots \wedge u_k\wedge u_{k+1}\wedge \ldots \wedge u_{k+i}}
&= \norm{u_1\wedge \ldots \wedge u_k\wedge \pi_{F^\perp,E^\perp}(u_{k+1})\wedge \ldots \wedge \pi_{F^\perp,E^\perp}(u_{k+i})} \\
&= \norm{ u_1 \wedge \ldots u_k}\,
\norm{ \pi_{F^\perp,E^\perp}(u_{k+1})\wedge \ldots \pi_{F^\perp,E^\perp}(\wedge u_{k+i})} \\
&\leq  \norm{\wedge_i \, \pi_{F^\perp,E^\perp}}\, \norm{ u_1\wedge \ldots \wedge u_k}\,\norm{ u_{k+1}\wedge \ldots \wedge u_{k+i}} \\
&\leq    \norm{ u_1\wedge \ldots \wedge u_k}\,\norm{ u_{k+1}\wedge \ldots \wedge u_{k+i}} \;,
\end{align*}
which proves (a).
By proposition~\ref{prop: alpha = det Pi(E F)} we have
$$\aangle(E,F)=\aangle(F^\perp,E^\perp)=\abs{\det (\pi_{F^\perp,E^\perp})}
\leq  \minexp( \wedge_i (\pi_{F^\perp,E^\perp}))\;.
$$
Thus
\begin{align*}
\norm{u_1\wedge \ldots \wedge u_k\wedge u_{k+1}\wedge \ldots \wedge u_{k+i}}
&= \norm{u_1\wedge \ldots \wedge u_k\wedge \pi_{F^\perp,E^\perp}(u_{k+1})\wedge \ldots \wedge \pi_{F^\perp,E^\perp}(u_{k+i})} \\
&= \norm{ u_1 \wedge \ldots u_k}\,
\norm{ \pi_{F^\perp,E^\perp}(u_{k+1})\wedge \ldots \pi_{F^\perp,E^\perp}(\wedge u_{k+i})} \\
&\geq  \minexp(\wedge_i \, \pi_{F^\perp,E^\perp})\, \norm{ u_1\wedge \ldots \wedge u_k}\,\norm{ u_{k+1}\wedge \ldots \wedge u_{k+i}} \\
&\geq  \alpha(E,F)\, \norm{ u_1\wedge \ldots \wedge u_k}\,\norm{ u_{k+1}\wedge \ldots \wedge u_{k+i}} \;,
\end{align*}
which proves (b).

\end{proof}

Of course the angle function $\aangle$ is Lipschitz continuous.

\begin{proposition}
\label{prop aangle continuity}
Given $u,u',v,v'\in \Pp(V)$,
$$ \abs{\aangle(u,v) -\aangle(u',v')} \leq d(u,u') + d(v,v') \;. $$
\end{proposition}

\begin{proof}
Exercise.
\end{proof}

\bigskip

The intersection of complementary flags satisfying the appropriate transversality conditions determines a decomposition of the Euclidean space $V$.
We end this subsection proving a modulus of continuity for this intersection operation.

Consider a signature $\tau=(\tau_1,\ldots, \tau_k)$ of length $k$ with $\tau_k<\dim V$.
We make the convention that
$\tau_0=0$ and $\tau_{k+1}=\dim V$.

\begin{definition}\label{def decomposition}
A $\tau$-decomposition is a family of linear subspaces $E_{\cdot}=\{E_i\}_{1\leq i\leq k+1}$ in $\Gr(V)$ such that $V=\oplus_{i=1}^{k+1} E_i$ and
$\dim E_i=\tau_i-\tau_{i-1}$ for all
$1\le i\leq k+1$. 
\end{definition}
Let $\Decompsp_\tau(V)$ denote
the space of all $\tau$-decompositions, which is a metric space with the distance
$$ d_\tau(E_{\cdot}, E_{\cdot}')=
\max_{1\leq i \leq k+1} d_{\tau_i-\tau_{i-1}}(E_i,E_i')\;, $$
and where $d_{\tau_i-\tau_{i-1}}$ stands for the distance~\eqref{Grassmann:distance} in $\Gr_{\tau_i-\tau_{i-1}}(V)$.

 Given  two flags $F \in\FF_\tau(V)$ and $F' \in\FF_{\tau^\perp}(V)$,  we will define a decomposition, denoted by $F\sqcap F'$, formed out of intersecting the components of these flags. 
 For that we introduce the following a measurement.

\begin{definition}
 \label{def sqcap tranvsersality}
 Given  two flags $F \in\FF_\tau(V)$ and $F' \in\FF_{\tau^\perp}(V)$,  let
$$ \theta_{\sqcap}(F,F'):=\min_{1\leq i\leq k}
\theta_{\cap}(F_i,F_{k-i+1}')\;.$$
\end{definition}

Notice that $\dim F_i=\tau_i$ and
$\dim F_{k-i+1}'=\tau_{k-i+1}^\perp=\dim V -\tau_i$, i.e., the subspaces $F_i$ and $F_{k-i+1}'$ have complementary dimensions. We will refer to this quantity as the measurement of the transversality between the flags $F$ and $F'$.

In the next proposition we complete $F$ and $F'$ to full flags of length $k+1$ setting $F_{k+1}=F_{k+1}'=V$. 
Assume also that $\tau_0=0$ and $\tau_{k+1}=\dim V$. 
\begin{proposition} 
\label{label theta sqcap >0 decomp well-def}
If $\theta_{\sqcap}(F,F')>0$ then the following is a direct sum decomposition in the space $\Decompsp_\tau(V)$, 
$$ V=\bigoplus_{i=1}^{k+1} F_i\cap F_{k-i+2}' \;, $$
with $\dim (F_i\cap F_{k-i+2}') = \tau_{i}-\tau_{i-1}$ for all $1\leq i\leq k+1$.
\end{proposition}

\begin{proof}
Since the subspaces $F_i$ and $F_{k-i+1}'$ have complementary dimensions, the relation $\theta_{\cap}(F_i,F_{k-i+1}')>0$ implies that 
\begin{equation}\label{V=Fi+Fk-i+1'}
V=F_i\oplus F_{k-i+1}'\;.
\end{equation}
By lemma~\ref{theta:monot},
 $\theta_{\cap}(F_i,F_{k-i+2}')\geq \theta_{\cap}(F_i,F_{k-i+1}')>0$. Therefore
$F_i + F_{k-i+2}'=V$ and
\begin{align*}
\dim(F_i\cap F_{k-i+2}') &=
\tau_i + \tau_{k-i+2}^\perp-\dim V\\
&=
\tau_i + (\dim V - \tau_{i-1})-\dim V = \tau_i-\tau_{i-1}\;.
\end{align*}
We prove by finite induction in $i=1,\ldots, k+1$
that
\begin{equation}\label{Fi oplus}
F_i= \bigoplus_{j\leq i} F_j\cap F_{k-j+2}'\;.
\end{equation}
Since $F_{k+1}=V$ the proposition will follow from this relation at $i=k+1$.

For $i=1$,  ~\eqref{Fi oplus} reduces to $F_1=F_1\cap V$. The induction step follows
from
$$ F_{i+1}=F_{i}\oplus \left( F_{i+1}\cap F_{k-i+1}'\right) \;.$$
Since the following dimensions add up
\begin{align*} 
\dim F_{i+1}=\tau_{i+1} & = \tau_{i} + (\tau_{i+1}-\tau_{i}) \\
& = \dim F_{i} +\dim (F_{i+1}\cap F_{k-i+1}') \;,
\end{align*}
it is enough to see that
$$ F_{i}\cap \left( F_{i+1}\cap F_{k-i+1}'\right) =  F_{i}\cap F_{k-i+1}' =\{0\}\;, $$
which holds because of~\eqref{V=Fi+Fk-i+1'}.

\end{proof}

Hence, by the previous proposition we can define

\begin{definition} 
\label{def sqcap decomp operation} 
Given flags $F\in \FF_\tau(V)$ and $F'\in\FF_{\tau^\perp}(V)$ such that $\theta_{\sqcap}(F,F')>0$ we define
$F\sqcap F':= \{ F_i\cap F_{k-i+2}'\}_{1\leq i\leq k+1}$ and call it the intersection decomposition of the flags $F$ and $F'$.
\end{definition}

Next proposition gives a modulus of lower semi-continuity for the   transversality measurement $\theta_{\sqcap}$.

\begin{proposition} 
\label{theta cap: F,F'->F0,F0'}.
Given $F, F_0\in \FF_\tau(V)$ and $F', F_0'\in \FF_{\tau^\perp}(V)$,
$$\theta_{\sqcap}(F,F')\geq \theta_{\sqcap}(F_0,F_0')
- d_{\tau}(F,F_0)-d_{\tau^\perp}(F',F_0')\;.$$
\end{proposition}

\begin{proof}
Apply proposition~\ref{theta cap: E->E0}.

\end{proof}

The modulus of continuity for the intersection map 
$\sqcap:\FF_\tau(V)\times \FF_{\tau^\perp}(V)\to \Decompsp_\tau(V)$ is established below.

\begin{proposition} 
\label{decomp:modulus cont}
Given flags $F_1, F_2\in\FF_\tau(V)$ and $F_1', F_2'\in\FF_{\tau^\perp}(V)$,
$$  d_\tau(F_1\sqcap F_1', F_2\sqcap F_2')  \leq \max\left\{  \frac{1}{\theta_\sqcap (F_1,F_1')},
 \frac{1}{\theta_\sqcap (F_2,F_2')} \right\} \,( d_{\tau}(F_1,F_2) + d_{\tau^\perp}(F_1',F_2') ) \;.$$ 
\end{proposition}

\begin{proof}
The proof reduces to apply proposition~\ref{sum:inters:modulus cont}.

\end{proof}

Any two linear maps $g_0,g_1\in\mathcal{L}(V)$
having  $\tau$-gap ratios, and such that $\aangle_\tau(g_0,g_1)>0$, determine a  $\tau$-decomposition of $V$ as intersection of the image by $\varphi_g$  of the $g_0$ most expanding $\tau$-flag with the $g_1$ least expanding $\tau^\perp$-flag. Recall definitions~\ref{def most expanding flag} and~\ref{def least expanding flag}.
The corresponding intersection measurement is bounded below by the angle $\aangle_\tau(g_0,g_1)$.

\begin{proposition}
\label{aangle bound}
Given $g_0,g_1\in\mathcal{L}(V)$, if $\rgap_\tau(g_0)>1$ and $\rgap_\tau(g_1)>1$ then
$$\theta_{\sqcap}(\leastexp_{\tau^\perp}(g_1),
\mostexp_\tau(g_0^\ast) )\geq \aangle_\tau(g_0,g_1)\;.$$
In particular, if $\aangle_\tau(g_0,g_1)>0$ the flags $\mostexp_\tau(g_0^\ast)$ and $\leastexp_{\tau^\perp}(g_1)$ determine the decomposition $\mostexp_\tau(g_0^\ast)\sqcap  \leastexp_{\tau^\perp}(g_1) \in\Decompsp_\tau(V)$.
\end{proposition}

\begin{proof} Let $n=\dim V$.
Consider the flags $F= \mostexp_{\tau}(g_0^\ast)$ and $F' = \leastexp_{{\tau}^\perp}(g_1)$.
We have $F_i= \mostexp_{\tau_i}(g_0^\ast)$ \, and \,
$F_{k-i+1} = \leastexp_{\tau_{k-i+1}^\perp}(g_1) = \leastexp_{n-\tau_{i}}(g_1) = \mostexp_{\tau_i}(g_1)^\perp$. Hence by lemma~\ref{theta aangle},
$$ \theta_{\cap}(F_i,F_{k-i+1}') = 
\theta_{\cap}(\mostexp_{\tau_i}(g_0^\ast),\mostexp_{\tau_i}(g_1)^\perp) = \aangle_{\tau_i }(\mostexp_{\tau_i}(g_0^\ast),\mostexp_{\tau_i}(g_1)) = \aangle_{\tau_i }(g_0,g_1)\;, $$
and taking the minimum,\,
$\theta_{\sqcap}(F,F')\geq \aangle_{\tau}(g_0,g_1)$.

\end{proof}

\subsection{ Dependence on the linear map}

We establish a modulus of Lipschitz continuity for the most expanding direction of a linear endomorphism  with a gap between its first and second singular values. For any $0<\kappa<1$, consider the set
$\mathcal{L}_\kappa :=\{\, g\in \mathcal{L}(V)\,:\,
\rgap(g)\geq \frac{1}{\kappa} \,\}$. We denote by $\mostexp:\mathcal{L}_\kappa\to\Pp(V)$ the map that assigns the $g$-most expanding direction to each $g\in \mathcal{L}_\kappa$.

The {\em relative distance} between
 linear maps $g,g'\in \mathcal{L}(V)\setminus\{0\}$ is defined as
$$ \drel(g,g'):=\frac{ \norm{g-g'} }{ \max\{ \norm{g}, \norm{g'} \}}\;. $$
Notice that this relative distance is not a metric.
It does not satisfy the triangle inequality. We  introduce it just to lighten the notation.

\begin{proposition}\label{lipschitz:eigendir} 
The map  $\mostexp:\mathcal{L}_\kappa\to\Pp(V)$ is locally Lipschitz. 

More precisely, given $0<\kappa<1$ there exists $\varepsilon_0 >0$ such that   for any $g_1,g_2\in \mathcal{L}_\kappa$ satisfying $\drel(g_1,g_2)\leq \varepsilon_0 $,
$$d(\mostexp(g_1), \mostexp(g_2))\leq \frac{16}{1-\kappa^2}\, \drel(g_1,g_2) \;. $$
\end{proposition}

\begin{proof} 
Let $g\in \mathcal{L}_\kappa$ and $\lambda>0$.
The singular values (resp. singular vectors) of $g$ are the eigenvalues (resp. eigenvectors) of $\sqrt{g^\ast\, g}$.
Hence  $s_j(\lambda\,g)= \lambda\, s_j(g)$ , for all $j$. We also have 
$\mostexp(\lambda g)=\mostexp(g)$ and $\rgap(\lambda\, g)=\rgap(g)$.

Consider the subspace
$\mathcal{L}_\kappa(1):=\{\, g\in \mathcal{L}_\kappa\, \colon\, \norm{g}=1\,\}$.
The projection $g\mapsto g/\norm{g}$ takes
$\mathcal{L}_\kappa$ to $\mathcal{L}_\kappa(1)$.
It also satisfies $\mostexp(g/\norm{g})=\mostexp(g)$ 
and 
$$ \norm{ \frac{g_1}{\norm{g_1}} - \frac{g_2}{\norm{g_2}} }\leq 
2\, \drel(g_1,g_2) \;. $$
Hence we can focus our attention on the 
restricted map
$\mostexp:\mathcal{L}_\kappa(1)\to\Pp(V)$.
 
Let $\mathcal{L}^+_\kappa(1)$ denote the subspace of  $g\in \mathcal{L}_\kappa(1)$ such that
$g=g^ \ast\geq 0$, i.e., $g$ is positive semi-definite.

Given $g \in\mathcal{L}_\kappa(1)$,
we have $\norm{g^\ast\, g}=1=\norm{g}$, $\rgap(g^\ast g)=\rgap(g)^2$  and  $\mostexp(g^\ast g)=\mostexp(g)$.
Also, for all $g_1, g_2\in\mathcal{L}_\kappa(1)$,
\begin{align*}
\norm{g_1^\ast\,g_1 -  g_2^\ast\,g_2 } &\leq
\norm{g_1^\ast}\,\norm{g_1 - g_2 } +
\norm{g_1^\ast -  g_2^\ast } \,\norm{g_2}\\
&= ( \norm{g_1^\ast} + \norm{g_2})\,\norm{g_1 - g_2 } 
\leq 2\,\norm{g_1-g_2}\;.
\end{align*}
Hence,  the mapping $g\mapsto g^\ast\,g$  takes $\mathcal{L}_\kappa(1)$  to $\mathcal{L}^+_{\kappa^2}(1)$ and  has Lispschitz constant  $2$.
Therefore, it is enough to prove that the 
restricted map
$\mostexp:\mathcal{L}^+_{\kappa^ 2}(1)\to\Pp(V)$
has (locally) Lipschitz constant $4\,(1-\kappa^2)^{-1}$. 

Let $\delta_0$ be a small positive number and take
$0<\varepsilon_0\ll \frac{ \delta_0}{4}$.
The size of $\delta_0$  will be fixed throughout the rest of the proof
according to necessity.
Take $h_1,h_2\in\mathcal{L}_{\kappa^2}^+(1)$
such that $\norm{h_1-h_2}<\varepsilon_0$  and  set
 $\hatp_0:=\mostexp(h_1)$. By Proposition~\ref{proj:contr} we have
$$\varphi_{h_1}\left(B(\hatp_0,\delta_0) \right) \subset  B \left(\hatp_0, \frac{\kappa^2\delta_0}{\sqrt{1-\delta_0^2}} \right)
\subset B(\hatp_0, \delta_0  )  \;, $$
where all balls refer to the projective sine-metric $\delta$ 
defined in~\eqref{projective sine metric}.
The second inclusion holds if $\delta_0$ is chosen small enough.
Take any $\hatp\in B(\hatp_0, \delta_0)$ and choose unit vectors
$p\in\hatp$ and $p_0\in\hatp_0$ such that $\langle p,p_0\rangle>0$.
Then $p = \langle p,p_0\rangle\,p_0 + w$, with $w\in p_0^\perp$,
$h_1(p_0)= p_0$   and
$h_1(w)\in p_0^\perp$. Hence
\begin{align*}
\norm{h_1(p)} & = \norm{\langle p,p_0\rangle\,p_0 + 
h_1(w) } \geq  \langle p,p_0\rangle  \\
& = \sqrt{1-\norm{p\wedge p_0}^2}\geq \sqrt{1-\delta_0^2} \geq   1/2 \;,
\end{align*}
and again, assuming $\delta_0$ is small,
$$
\norm{h_2(p)}  \geq \norm{h_1(p)} -\norm{h_1-h_2} \geq\sqrt{1-\delta_0^2}    -\varepsilon_0 
\geq  1/2 \;.
$$
Thus, by Lemma~\ref{varphi:gi} below, for all $\hatp\in B(\hatp_0,\delta_0)$,
$$ d(\varphi_{h_1}(\hatp), \varphi_{h_2}(\hatp)) \leq 2\,  \norm{h_1-h_2}  \;.  $$
Choosing $\varepsilon_0$ small enough, $\frac{\kappa^2\,\delta_0}{\sqrt{1-\delta_0^2}} + 2\,\varepsilon_0  < \delta_0$. This implies that
$$\varphi_{h_2}\left(B(\hatp_0,\delta_0) \right) \subset B(\hatp_0,\delta_0) \;. $$
By Proposition~\ref{proj:contr} we  know that 
$T_1=\varphi_{h_1}\vert_{B(\hatp_0,\delta_0)}$ has Lispchitz
constant $\kappa'= \kappa^2\,\frac{\delta_0+\sqrt{1-\delta_0^2}}{1-\delta_0^2}\approx \kappa^2$, and assuming $\delta_0$ is small enough we have $\frac{1}{1-\kappa'}\leq \frac{2}{1-\kappa^2}$.
Notice that although the Lispchitz constant in this proposition refers to the Riemannian metric $\rho$, since the ratio  ${\rm Lip}_\delta(T_1)/{\rm Lip}_\rho(T_1)$ approaches $1$ as $\delta_0$ tends to $0$, we can assume that
${\rm Lip}_\delta(T_1)\leq \kappa'$.
Thus, by Lemma~\ref{Ti} below applied to $T_1$ and $T_2=\varphi_{h_2}\vert_{B(\hatp_0,\delta_0)}$,
we have $d(T_1,T_2)\leq 2\,\norm{h_1-h_2}$ and
$$
d(\mostexp(h_1),\mostexp(h_2)) \leq \frac{1}{1-\kappa'}\, d(T_1,T_2)\leq  \frac{4}{1-\kappa^2}\,
 \norm{h_1-h_2}  \;.
$$

\end{proof}

\begin{lemma} \label{Ti}
Let $(X,d)$ be a complete metric space,
$T_1:X\to X$ a Lipschitz contraction  
  with ${\rm Lip}(T_1)<\kappa<1$,  
$x_1^\ast=T_1 (x_1^\ast)$ a fixed point,  
and  $T_2:X\to X$  any other map with a fixed point $x_2^\ast=T_2 (x_2^\ast)$. Then
$$ d(x_1^\ast, x_2^\ast)\leq \frac{1}{1-\kappa}\, d(T_1,T_2) \;,$$
where   $d(T_1,T_2):=\sup_{x\in X} d(T_1(x), T_2 (x))$.
\end{lemma}

\begin{proof}
\begin{align*}
 d(x_1^\ast, x_2^\ast) & =
  d(T_1(x_1^\ast), T_2(x_2^\ast))\\
  &\leq
  d(T_1(x_1^\ast), T_1(x_2^\ast)) +  d(T_1(x_2^\ast), T_2(x_2^\ast))\\
   &\leq
  \kappa\, d(x_1^\ast, x_2^\ast) +  d(T_1, T_2)\;,
\end{align*}
which implies that
$$  d(x_1^\ast, x_2^\ast)  \leq \frac{1}{1-\kappa}\,  d(T_1, T_2)\;. $$ 
\end{proof}

\begin{lemma}
\label{lemma wedge i g1 minus wedge i g2}
Given $g_1,g_2\in\mathcal{L}(V)$, for any $1\leq i \leq \dim V$,
$$ \norm{ \wedge_i g_1 -  \wedge_i g_2 }
\leq i\, \max\{ 1, \norm{g_1}, \norm{g_2} \}^{i-1}\,\norm{g_1-g_2}\;. $$
\end{lemma}

\begin{proof}
Given any unit $i$-vector $v_1\wedge \ldots \wedge v_i\in \wedge_i V$, determined by an orthonormal family of vectors $\{v_1,\ldots, v_i\}$,
\begin{align*}
& \norm{(\wedge_i g_1)(v_1\wedge \ldots \wedge v_i)
- (\wedge_i g_2)(v_1\wedge \ldots \wedge v_i) }
 = \\
 & \qquad = \norm{ (g_1 v_1)\wedge \ldots \wedge (g_1 v_i) 
- (g_2 v_1)\wedge \ldots \wedge (g_2 v_i) }\\
& \qquad \leq \sum_{j=1}^i 
\norm{ (g_1 v_1)\wedge \ldots \wedge (g_1 v_{j-1}) \wedge (g_1 v_j  
-  g_2 v_j)\wedge (g_2 v_{j+1}) \wedge \ldots \wedge (g_2 v_i)  } 		\\
& \qquad \leq \sum_{j=1}^i 
\norm{g_1}^{j-1}\,\norm{g_2}^{i-j} \, 
\norm{ g_1 v_j -  g_2 v_j} 		\\
&\qquad \leq i\, \max\{1, \norm{g_1}, \norm{g_2} \}^{i-1}\,\norm{g_1-g_2}\;.
\end{align*}

\end{proof}

Given a dimension $1\leq l \leq \dim V$ and  $0<\kappa<1$, consider the set
$$ \mathcal{L}_{l, \kappa} :=\{\, g\in \mathcal{L}(V)\,:\,
\rgap_l(g)\geq  \kappa^{-1} \,\}\;,$$
and define
$$ C_l(g_1,g_2):=  \frac{l\,\max\{ 1, \norm {g_1},\norm{g_2}\}^{l-1}}{\max\{\norm{ 1, \wedge_l g_1}, \norm{\wedge_l g_2} \}}\;. $$

\begin{corollary}
\label{coro lipschitz:medir} 
The map  $\mostexp:\mathcal{L}_{l,\kappa}\to\Gr_l(V)$ is locally Lipschitz. 

More precisely, given $0<\kappa<1$ there exists $\varepsilon_0 >0$ such that   for any $g_1,g_2\in \mathcal{L}_{l,\kappa}$ such that $\norm{g_1-g_2}\leq \varepsilon_0 \,C_l(g_1,g_2)^{-1}$, we have
$$d(\mostexp_l(g_1), \mostexp_l(g_2))\leq \frac{16}{1-\kappa^2}\,C_l(g_1,g_2)\,
\norm{g_1-g_2}\;. $$
\end{corollary}

\begin{proof}
By lemma~\ref{lemma wedge i g1 minus wedge i g2},
$\drel(\wedge_l g_1, \wedge_l g_2) \leq C_l(g_1,g_2)\,\norm{g_1-g_2}$.
Apply proposition~\ref{lipschitz:eigendir} 
to the linear maps $\wedge_l g_j:\wedge_l V \to \wedge_l V$, $j=1,2$.

\end{proof}

Given $g\in\mathcal{L}(V)$ having  $k$ and $k+r$ gap ratios, if a subspace $E\in\Gr_k(V)$ close to the $g$ most expanding subspace $\mostexp_k(g)$ then the restriction $g\vert_{E^\perp}$ has a $r$-gap ratio and the most expanding $r$-dimensional subspace of $g\vert_{E^\perp}$ is close to the intersection of
$\mostexp_{k+r}(g)$ with $E^\perp$. Next proposition expresses this fact in a quantitative way.

\begin{proposition}\label{prop dist med g | E perp}
Given $\varkappa>0$ small enough, and integers $1\leq k < k+r \leq \dim V$, there exists $\delta_0>0$ such that for all $g\in\mathcal{L}(V)$
and   $E\in\Gr_k(V)$, if
\begin{enumerate}
\item[(a)]\; $\sgap_k(g)<\varkappa$ \, and \,  $\sgap_{k+r}(g)<\varkappa$, 
\item[(b)]\;  $\delta(E,\mostexp_k(g))<\delta_0$
\end{enumerate}
then
\begin{enumerate}
\item[(1)]\; $\sgap_r(g\vert_{E^\perp})\leq 2\,\varkappa$,
\item[(2)]\; $\displaystyle \delta\left(\, \mostexp_r(g\vert_{E^\perp}), \, 
\mostexp_{k+r}(g)\cap E^\perp\,
\right) \leq \frac{20}{1-4\,\varkappa^2}\,\delta(E,\mostexp_k(g))$.
\end{enumerate}
\end{proposition}

\begin{proof}
Consider the compact space
$$ \mathcal{K}_r=\{\, h\in \mathcal{L}(V)\,\colon\,
\norm{h}\leq 1 \; \text{ and }\; \sgap_r(h)\leq \varkappa\,\}\;. $$
By uniform continuity of $\sgap_r$ on $\mathcal{K}_r$ there exists $\delta_0>0$ such that for all
$h\in \mathcal{L}(V)$ if there exists $h_0\in  \mathcal{K}_r$
with $\norm{h-h_0}<\delta_0$ then $\sgap_r(h)\leq 2\,\varkappa$.

Recall that $\pi_F$   denotes the orthogonal projection onto a linear subspace $F\subset V$.

Given $g \in \mathcal{L}(V)$ such that (a) holds,
consider the   map $h=\frac{g}{\norm{g}}\circ \pi_{\mostexp_k(g)^\perp}$. We have $h\in \mathcal{K}_r$ because 
$\sgap_r(h)= \sgap_r(g \circ \pi_{\mostexp_k(g)^\perp})
= \sgap_{k+r}(g)<\varkappa$.

Given $E\in \Gr_k(V)$ such that (b) holds,
we define $h_E=\frac{g}{\norm{g}}\circ \pi_{E^\perp}$.
Then
$$ \norm{h-h_E}\leq \norm{ \pi_{\mostexp_k(g)^\perp} - \pi_{E^\perp} } \lesssim \delta(\mostexp_k(g)^\perp, E^\perp) = \delta(E, \mostexp_k(g)) <\delta_0 \;, $$
which implies that
$\sgap_r(g\vert_{E^\perp})=\sgap_r(h_E) \leq 2\,\varkappa$,
and hence proves (1).

For  (2) we use the following triangle inequality
\begin{align*}
\delta( \mostexp_r(g\vert_{E^\perp}), \mostexp_{k+r}(g)\cap E^\perp ) &\leq  \; \delta( \mostexp_r( h_E), \mostexp_r(h) )  \\
&\quad   + \delta( \mostexp_r(h), 
\mostexp_{k+r}(g)\cap \mostexp_k(g)^\perp ) \\
&\quad   + \delta( \mostexp_{k+r}(g)\cap \mostexp_k(g)^\perp,\, 
\mostexp_{k+r}(g)\cap E^\perp  ) \\
&\leq  \left( \frac{16\,r}{1-4\,\varkappa^2} + 0 + 1\right)\,\delta(E,\mostexp_k(g))\\
&\leq \frac{20}{1-4\,\varkappa^2} \,\delta(E,\mostexp_k(g))\; .
\end{align*}
The bound on the first distance is obtained through
corollary~\ref{coro lipschitz:medir}, with $C_r(h_E,h)=r$.
The second distance is zero.
Finally the bound on the third distance comes from
proposition~\ref{sum:inters:modulus cont} (2), using that
$\theta_\cap( \mostexp_{k+r}(g), \mostexp_{k}(g)^\perp )=1$,
because $\mostexp_{k}(g)\subset \mostexp_{k+r}(g)$.

\end{proof}

\begin{lemma}\label{varphi:gi}
Given $g_1, g_2\in \mathcal{L}(V)$, $\hatp\in\Pp(g_1)\cap \Pp(g_2)$   and any unit vector $p\in\hat p$,
$$ d(\varphi_{g_1}(\hatp), \varphi_{g_2}(\hatp)) \leq
\max\{ \frac{1}{\norm{g_1\,p}}, \frac{1}{\norm{g_2\,p}} \}\,
 \norm{g_1-g_2} \;.$$
\end{lemma}

\begin{proof} Assume $p\in V$ is a unit vector such that $\hatp\in \Pp(g_1)\cap \Pp(g_2)$.
Applying  proposition~\ref{proj:lip} to the non-zero vectors
$g_1 \,p$ and $g_2\,p$, we get

\begin{align*}
d(\varphi_{g_1}(\hatp), \varphi_{g_2}(\hatp)) &\leq
\norm{ \frac{g_1 \,p}{\norm{g_1 \,p}} -
\frac{g_2 \,p}{\norm{g_2 \,p}} } \\
 &\leq \max\{ \norm{g_1\,p}^{-1}, \norm{g_2\,p}^{-1} \}\,
 \norm{g_1\, p-g_2\, p} \\
  &\leq \max\{ \norm{g_1\,p}^{-1}, \norm{g_2\,p}^{-1} \}\,
 \norm{g_1 -g_2} \;.
\end{align*} 
\end{proof}

The final four lemmas of this subsection apply to invertible linear maps in $\GL(V)$. They express  the continuity of the map  $g\mapsto \varphi_g$ with values in the space of Lipschitz or H\"older continuous maps on the projective space. These facts will be useful in ~\cite{LEbook}.

\begin{lemma}\label{delta:1}
Given $g_1, g_2\in \GL(V)$, and $\hatp\neq \hatq$ in $\Pp(V)$,
$$ \abs{ \frac{\delta(\varphi_{g_1}(\hatp), \varphi_{g_1}(\hatq))}{\delta(\hatp,\hatq)} - \frac{\delta(\varphi_{g_2}(\hatp), \varphi_{g_2}(\hatq))}{\delta(\hatp,\hatq)}  }
\leq C(g_1,g_2)\,  \norm{g_1 - g_2}\;, $$
where \, $C(g_1,g_2):= (\norm{g_1^{-1}}^ 2 + \norm{g_2}^2\,\norm{g_1^{-1}}^2\,\norm{g_2^{-1}}^2)\,(\norm{g_1}+  \norm{g_2})$.
\end{lemma}
 
\begin{proof}
Given $p\in\hat p$ and $q\in\hat q$, by proposition~\ref{Lip:proj:action}
\begin{align*}
& \abs{ \frac{\delta(\varphi_{g_1}(\hatp), \varphi_{g_1}(\hatq))}{\delta(\hatp,\hatq)} - \frac{\delta(\varphi_{g_2}(\hatp), \varphi_{g_2}(\hatq))}{\delta(\hatp,\hatq)} }  = 
\abs{  \frac{ \norm{g_1 p \wedge g_1 v_p(q) } }{ \norm{g_1 p} \norm{g_1 q} } - 
\frac{ \norm{g_2 p \wedge g_2 v_p(q) } }{ \norm{g_2 p} \norm{g_2 q} }   }\\
&\qquad \leq \frac{  \norm{g_1 p \wedge g_1 v_p(q)  -  g_2 p \wedge g_2 v_p(q)   } }{ \norm{g_1 p} \norm{g_1 q} }  \\
&\qquad  \quad + \, \abs{ \frac{1}{\norm{g_1 p} \norm{g_1 q}} - \frac{1}{\norm{g_2 p} \norm{g_2 q}}   } \, \norm{  g_2 p \wedge g_2 v_p(q) }  \\
&\qquad \leq  \norm{g_1^ {-1}}^2\, \norm{g_1 p \wedge (g_1 v_p(q)  -  g_2 v_p(q))   }  + 
 \norm{g_1^{-1}}^2\, \norm{(g_1 p -  g_2 p) \wedge g_2 v_p(q)   } \\
&\qquad \quad + \,  \norm{g_1^{-1}}^2\,\norm{g_2^{-1}}^2\,( \norm{g_1 p}\,\abs{\norm{g_1 q}- \norm{g_2 q}}  + 
\norm{g_2 q}\,\abs{\norm{g_1 p} - \norm{g_2 p}} )\,\norm{g_2}^2  \\
&\qquad \leq \norm{g_1^ {-1}}^2\, (\norm{g_1} + \norm{g_2})\, \norm{g_1-g_2}   \\
&\qquad \quad + \, \norm{g_2}^2\, \norm{g_1^{-1}}^2\,\norm{g_2^{-1}}^2\,(\norm{g_1}+  \norm{g_2})\, \norm{g_1-g_2} \\
&\qquad  = (\norm{g_1^{-1}}^ 2 + \norm{g_2}^2\, \norm{g_1^{-1}}^2\,\norm{g_2^{-1}}^2)\,(\norm{g_1}+  \norm{g_2})\,\norm{g_1-g_2}  \;.
\end{align*}

\end{proof}

\begin{lemma}\label{delta:3} Given $g\in\GL(V)$ and $\hatp\neq \hatq$ in $\Pp(V)$,
$$ \frac{1}{\norm{g}^{2}\,\norm{g^{-1}}^{2}} \leq \frac{\delta(\varphi_g(\hatp), \varphi_g(\hatq))}{\delta(\hatp,\hatq)} \leq \norm{g}^2\,\norm{g^{-1}}^2 \;.$$
\end{lemma}

\begin{proof}
Given $\hatp\neq \hatq$ in $\Pp(V)$ consider unit vectors
$p\in\hatp$, $q\in \hatq$ and set $v=v_p(q)$.
We have $\norm{p}=\norm{q}=\norm{v}=1$ and $\langle p,v\rangle=0$. This last relation implies $\norm{p\wedge v}=1$. Hence
$$ \norm{g p \wedge g v}= \norm{(\wedge_2 g)(p\wedge v)}\geq \norm{(\wedge_2 g)^{-1}}^{-1}\geq \norm{g^{-1}}^ {-2} \;. $$
Analogously
$$ \norm{g p \wedge g v}= \norm{(\wedge_2 g)(p\wedge v)}\leq \norm{\wedge_2 g}\leq \norm{g}^ 2 \;. $$
We also have
$$    \norm{g^{-1}}^{-2} \leq \norm{g\,p}\,\norm{g\,q} \leq \norm{g}^2\;. $$
To finish the proof combine these inequalities with proposition~\ref{Lip:proj:action}.

\end{proof}

\bigskip

Given $g\in\GL(V)$, we define
\begin{equation} \label{ell}
 \ell(g):= \max\{ \log \norm{g},\log \norm{g^{-1}} \}\;. 
\end{equation}

\begin{lemma}\label{Lipschitz:Proj}
For every $g\in\GL(V)$ and $\hatp\neq \hatq$ in $\Pp(V)$,
$$ -4\,\ell(g) \leq \log \left[\frac{\delta(\varphi_g(\hatp), \varphi_g(\hatq))}{\delta(\hatp,\hatq)}\right]\leq 4\,\ell(g)\;.$$
\end{lemma}
\begin{proof} 
Follows from ~lemma~\ref{delta:3}. \end{proof}

\begin{lemma}\label{delta:2}
Given $g_1, g_2\in\GL(V)$, $0<\alpha\leq 1$ and $\hatp\neq \hatq$ in $\Pp(V)$,
$$ \abs{ \left(\frac{\delta(\varphi_{g_1}(\hatp), \varphi_{g_1}(\hatq))}{\delta(\hatp,\hatq)}\right)^ \alpha - 
\left(\frac{\delta(\varphi_{g_2}(\hatp), \varphi_{g_2}(\hatq))}{\delta(\hatp,\hatq)} \right)^ \alpha }
\leq C_1(g_1,g_2)\,  \norm{g_1 - g_2}\;, $$
where $C_1(g_1,g_2)= \alpha\, \max\{\norm{g_1}\,\norm{g_1^{-1}}, \norm{g_2}\,\norm{g_2^{-1}}\}^{2(1-\alpha)}\, C(g_1,g_2)$,
and $C(g_1,g_2)$ stands for the constant in lemma~\ref{delta:1}.
\end{lemma}
 
\begin{proof}
Setting $\Delta_1:=\frac{\delta(\varphi_{g_1} \hatp, \varphi_{g_1} \hatq)}{\delta(\hatp,\hatq)}$ and $\Delta_2:=\frac{\delta(\varphi_{g_2} \hatp, \varphi_{g_2} \hatq)}{\delta(\hatp,\hatq)}$,
from lemmas~\ref{delta:1} and~\ref{delta:3} we get
\begin{align*}
\abs{ \Delta_1^\alpha - \Delta_2^\alpha  } &\leq \alpha\, \max\{\Delta_1^{\alpha-1}, \Delta_2^{\alpha-1}\} \,\abs{\Delta_1-\Delta_2}\\
&\leq  \alpha\, \max\{ \norm{g_1}\,\norm{g_1^{-1}}, \norm{g_2}\,\norm{g_2^{-1}} \}^{2(1-\alpha)} \,\abs{\Delta_1-\Delta_2}\\
&\leq  \alpha\, \max\{ \norm{g_1}\,\norm{g_1^{-1}}, \norm{g_2}\,\norm{g_2^{-1}} \}^{2(1-\alpha)} \, C(g_1,g_2) \,\norm{ g_1- g_2}\;.
\end{align*}

\end{proof}

\bigskip

\section{Avalanche Principle}
\label{ap}
\newcommand{\asp}[1]{\stackrel{{#1}}{\asymp}}
\newcommand{\epsh}{\varepsilon_{{\rm sh}}}
\newcommand{\desh}{\delta_{{\rm sh}}}
\newcommand{\kash}{\kappa_{{\rm sh}}}

Consider a long chain of $n$ linear maps
$g_0:V_0\to V_1$, $g_1:V_1\to V_2$, etc,  between Euclidean spaces $V_i$ of the same dimension $m$. The  AP  relates
the expansion $\norm{g_{n-1}\, \ldots \, g_1 \,g_0}$ of the composition $g_{n-1}\, \ldots \, g_1 \,g_0$  with the product  of the individual expansions $\norm{g_{n-1}}\,\ldots\, \norm{g_1}\,\norm{g_0}$. Given two
 quantities $M_n$ and $N_n$ depending on a large number $n\in\N$, we say in rough terms that they are $\varepsilon$-asymptotic, and write $M_n\asp{\varepsilon}  N_n$, when $e^{-n\,\epsilon}\leq M_n/N_n \leq e^{n\,\epsilon}$.
 In general it is not true  that $\norm{g_{n-1}\, \ldots \, g_1 \,g_0}\asp{\varepsilon} \norm{g_{n-1}}\,\ldots\, \norm{g_1}\,\norm{g_0}$ 
for some small $\varepsilon>0$,  unless some atypically sharp alignment of the singular directions of the linear maps $g_j$ occurs.
Given the chain of linear maps
$g_0, g_1,\ldots, g_{n-1}$,
its {\em rift} $\rift(g_0,\ldots, g_ {n-1}):= \frac{\norm{g_{n-1}\,\ldots \,g_0}}{\norm{g_{n-1}}\,\ldots\, \norm{g_0}}\in [0,1]$ measures the break of expansion in the 
composition $g_{n-1}\,\dots \, g_1\,g_0$.
The AP says that given any such chain   $g_0, g_1,\ldots, g_{n-1}$, where
the gap ratio\footnote{ratio between the first and second singular value} of each map $g_j$  is large, and
the rift of any pair of consecutive maps is never too small,
then the rift of the composition behaves  multiplicatively, in the sense that for some small number $\varepsilon>0$,
$$ \rho(g_0,g_1,\ldots, g_{n-1})\asp{\varepsilon} \rho(g_0,g_1)\,\rho(g_1,g_2)\,\ldots\, \rho(g_{n-2},g_{n-1}) \;, $$
or, equivalently,
$$ \frac{\norm{g_{n-1}\, \ldots \, g_1 \,g_0} \, \norm{g_1}\, \ldots\, \norm{g_{n-2}} }{  \norm{g_1\,g_0}\, \ldots\, \norm{g_{n-1}\,g_{n-2}} }\asp{\varepsilon} 1 \;. $$

The AP was introduced by M. Goldstein and W. Schlag~\cite[proposition 2.2]{GS-Holder} 
as a technique to obtain  H\"oder continuity of the integrated density of states for
quasi-periodic Schr\"odinger cocycles.  In the original version, the AP  applies to chains of unimodular matrices in $\SL(2,\C)$, and the length of the chain is assumed to be bounded by some  lower bound on the norms of the matrices. Notice that for unimodular matrices, the gap ratio and the norm are two equivalent measurements.
Still in this unimodular setting,  for matrices in $\SL(2,\R)$, J. Bourgain and S. Jitomirskaya~\cite[lemma 5]{B-J} have greatly relaxed the constraint on the  length of the chain of matrices, and later J. Bourgain~\cite[lemma 2.6]{B-d} has completely removed it, at the cost of slightly weakening the conclusion of the AP.

Later, W. Schlag ~\cite[lemma 1]{Schlag} has generalized the AP to  invertible matrices in $\GL(m,\C)$. Recently, C. Sadel has shared with the authors an earlier draft of ~\cite{AJS},
containing his version of the AP for $\GL(m,\C)$ matrices. Both these higher dimensional  APs  assume some bound on the length of the chains of matrices.
A higher dimensional  AP without this assumption was proven by the authors~\cite[theorem 3.1]{DK2} for invertible real matrices.

We present here the proof of a more general AP, that holds for  (possibly non-invertible) matrices in $\Mat(m,\R)$. As a by-product of the geometric approach used in the proof,
we also obtain a quantitative control on the most expanding directions of the product matrix, something essential to prove the continuity of the Oseledets decomposition.

\subsection{Contractive shadowing}

We prove here a {\em shadowing lemma} saying that under some conditions a loose pseudo-orbit of a chain of contracting maps is shadowed by a true orbit of the mapping sequence. In particular, a closed pseudo-orbit is shadowed by a periodic orbit of the mapping chain.

Given a metric space  $(X,d)$, denote the closed $r$-ball around $x\in X$ by 
$$B(x,\varepsilon):=\{\, z\in X\,\colon\, d(z,x)\leq \varepsilon \,\}\;. $$
Given an open set $X^0\subset X$, define
$$ X^0(\varepsilon):=\{\, x\in X^0\,\colon\, d(x,\partial X^0)\geq \varepsilon\,\}\;,$$
where $\partial X^0$ denotes the topological boundary of $X^0$ in $(X,d)$.

\begin{lemma}[shadowing lemma] \label{shadow:lemma}
Consider  $\varepsilon>0$  and
$0<\delta<\kappa<1$   such that $\delta/(1-\kappa)<\varepsilon<1/2$.

Given a family $\{ (X_j,d_j)\}_{0\leq j\leq n}$ of compact metric spaces with diameter $1$,   a  chain of continuous mappings  $\{g_j:X_j^0  \to X_{j+1}\}_{0\leq j\leq n-1}$ defined on open sets $X^0_j\subset X_j$,
and a sequence of points
$x_j\in X_j$,  assume that for every $0\leq j\leq n-1$:
\begin{enumerate}
\item[(a)]  $x_j\in X^0_j$ and $d(x_j,\partial X^0_j)=1$,
\item[(b)] $g_j$ has Lipschitz constant
$\leq \kappa$ on $X^0_j(\varepsilon)$, 
\item[(c)] $g_j(x_j)\in X^0_{j+1}(2\,\varepsilon)$,
\item[(d)] $g_j(X^0_j(\varepsilon))\subset B(g_j(x_j),\delta)$.
\end{enumerate}
Then, setting $g^{(n)} :=g_{n-1}\circ \ldots\circ g_1\circ g_0$, the following hold:
\begin{enumerate}
\item[(1)] the composition $g^{(n)}$ is defined on $B(x_0,\varepsilon)$ and ${\rm Lip}(g^{(n)}\vert_{B(x_0,\varepsilon)})\leq \kappa^n$,

\item[(2)]   $d (\, g_{n-1}(x_{n-1}), \, g^{(n)}(x_0)\, )\leq 
\frac{\delta}{1-\kappa}$,

\item[(3)] if $x_0=g_{n-1}(x_{n-1})$ then
$g^{(n)}(B(x_0,\varepsilon))\subset B(x_0,\varepsilon)$ and  there is 
a point $x^\ast\in B(x_0,\varepsilon)$
such that $g^{(n)}(x^\ast)=x^\ast$ and
$d\left(  x_0, x^\ast \right)\leq 
\frac{\delta}{(1-\kappa)(1-\kappa^n)}$.
\end{enumerate}
\end{lemma}

\begin{proof}

The proof's inductive scheme is better understood
with the help of figure~\ref{chain:orbits},
where we set  $z^i_j:= (g_{j-1}\circ \ldots \circ g_{i+1}\circ g_i)(x_i)$ for $i\leq j\leq n$. Of course we have to prove that all  points $z^i_j$ are well-defined.

\begin{figure}[h]
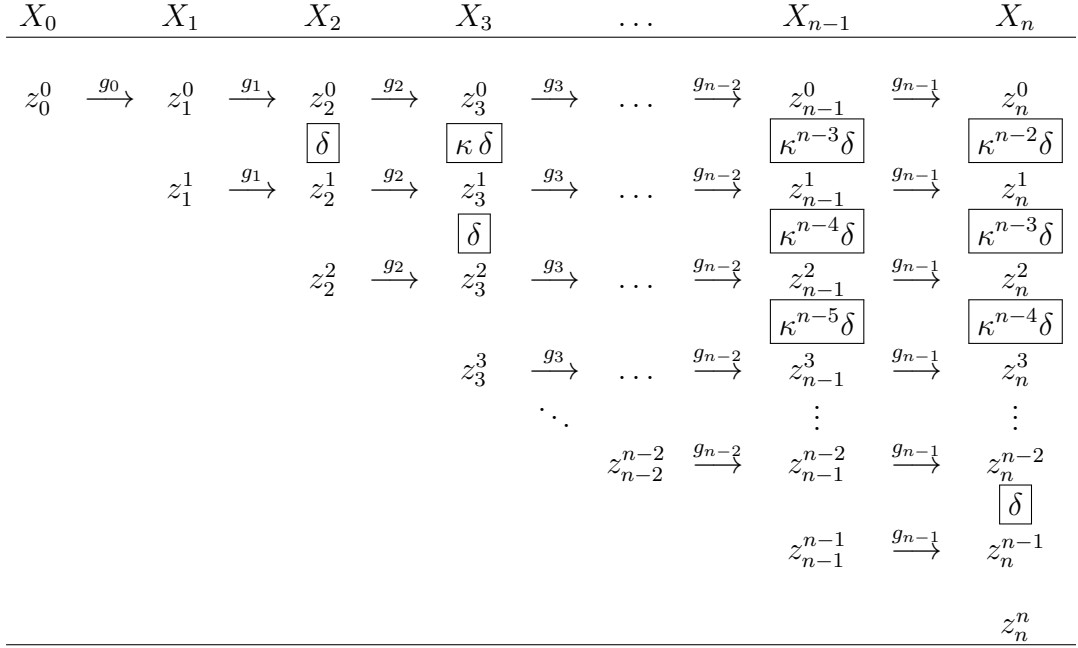

$$\begin{array}{cccccccccccccc}
X_0 &  & X_1  &  & X_2  &  & X_3  & & \ldots & & X_{n-1} & & X_n\\
\hline \\
z^0_0 & \stackrel{g_0}{\longrightarrow} & z^0_1
& \stackrel{g_1}{\longrightarrow} & z^0_2
& \stackrel{g_2}{\longrightarrow} & z^0_3
& \stackrel{g_3}{\longrightarrow} & \; \ldots \;  & \stackrel{g_{n-2}}{\longrightarrow} & z^0_{n-1} &  \stackrel{g_{n-1}}{\longrightarrow} & z^0_n \\
& &   & & \boxed{\delta} & & \boxed{\kappa\,\delta} & & & & \boxed{\kappa^{n-3}\delta} & & \boxed{\kappa^{n-2}\delta}\\
  &  & z^1_1
& \stackrel{g_1}{\longrightarrow} & z^1_2
& \stackrel{g_2}{\longrightarrow} & z^1_3
& \stackrel{g_3}{\longrightarrow}  & \; \ldots \;  & \stackrel{g_{n-2}}{\longrightarrow} & z^1_{n-1} &  \stackrel{g_{n-1}}{\longrightarrow} & z^1_n \\
& &  & & & & \boxed{\delta} &  & & & \boxed{\kappa^{n-4}\delta} & & \boxed{\kappa^{n-3}\delta}\\
  &  &  &  & z^2_2
& \stackrel{g_2}{\longrightarrow} & z^2_3
& \stackrel{g_3}{\longrightarrow}  & \; \ldots \;  & \stackrel{g_{n-2}}{\longrightarrow} & z^2_{n-1} &  \stackrel{g_{n-1}}{\longrightarrow} & z^2_n \\
& &  & & & & & & & &   \boxed{\kappa^{n-5}\delta} & & \boxed{\kappa^{n-4}\delta} \\
  &  &  &  &  &  & z^3_3
& \stackrel{g_3}{\longrightarrow}  & \; \ldots \;  & \stackrel{g_{n-2}}{\longrightarrow} & z^3_{n-1} &  \stackrel{g_{n-1}}{\longrightarrow} & z^3_n \\
& &  & & & & & \ddots & & & \vdots & & \vdots\\
 & & & & & & & & z^{n-2}_{n-2} & \stackrel{g_{n-2}}{\longrightarrow}  & z^{n-2}_{n-1} &  \stackrel{g_{n-1}}{\longrightarrow} & z^{n-2}_n \\
& &  & & & & & & & &  & & \boxed{\delta} \\
 & & & & & & & & & & z^{n-1}_{n-1} &  \stackrel{g_{n-1}}{\longrightarrow} & z^{n-1}_n \\
& &  & & & & & & & & & &  \\
  & & & & & & & & & & & & z^{n}_n \\
\hline
\end{array} $$
\caption{Family of orbits for the chain of mappings $\{g_j:X^0_j\to X_{j+1}\}_j$.
}
\label{chain:orbits}
\end{figure}

The boxed expressions represent upper bounds on the distance between the points respectively above and below the box.
The $i$-th row represents the orbit of $x_i\in X_i$ by the   chain of mappings $\{g_j\}_{j\geq i}$. All points in the $j$-th column belong to the space $X_j$.

To explain the last upper bound at the bottom of each column,
first notice that  $z^i_i=x_i$.
By (a),  $z^{i-1}_{i}=g_{i-1}(x_{i-1})$ is well-defined,
and by (c), $z^{i-1}_{i}\in X^0_{i}(2\,\varepsilon)\subset X^0_{i}(\varepsilon)$.
Likewise $z^{i-2}_{i-1}\in   X^0_{i-1}(\varepsilon)$,
and    $z_i^{i-2}=g_{i-1}(g_{i-2}(x_{i-2}))$ is well-defined.
Then by (d) we have
\begin{equation}\label{delta:dist}
  d(z^{i-1}_i, z_i^{i-2}) =  d(g_{i-1}(x_{i-1}),g_{i-1}(g_{i-2}(x_{i-2}))) \leq \delta\;. 
\end{equation}
All other bounds are obtained applying (b) inductively.
More precisely, we prove by induction in the column index $j$
that 
\begin{enumerate}
\item[(i)] all points $z^i_j$ in the $j$-th column are well-defined and  belong to $ X^0_j(\varepsilon)$,  
\item[(ii)]  distances between between consecutive points in the  column $j$ are bounded by the expressions in figure~\ref{chain:orbits}, i.e., 
for all $1\leq i\leq j-1$,
\begin{equation}\label{zij:claim}
d(z_j^{i-1},z_j^{i})\leq \kappa^{j-i-1}\,\delta \;.   
\end{equation} 
\end{enumerate}

The initial  inductive steps, $j=0,1,2$, follow  from (a), (c) and ~\eqref{delta:dist}.
Assume now that the points $z_j^i$ in $j$-th column satisfy (i) and (ii). Then their images $z^i_{j+1}=g_{j}(z_j^i)$ are well-defined.
By (b) we have for all  $1\leq i\leq j-1$,
$$d(z_{j+1}^{i-1},z_{j+1}^{i}) =
d(g_j(z_j^{i-1}), g_j(z_j^{i}))  \leq
\kappa\,d(z_j^{i-1},z_j^{i})\leq \kappa^{j-i}\,\delta \;. $$
Together with ~\eqref{delta:dist} this proves  (ii) for the column $j+1$. To prove (i) consider any  $1\leq i\leq j$.
By (c) and the triangle inequality,
\begin{align*}
d( z^i_{j+1}, \partial X^0_{j+1}(\varepsilon))  &\geq
d( z^j_{j+1}, \partial X^0_{j+1}(\varepsilon)) - d( z^i_{j+1}, z^j_{j+1})\\
&\geq
d( g_j(x_j), \partial X^0_{j+1}(\varepsilon)) - \sum_{l=i+1}^{j} d( z^{l-1}_{j+1}, z^{l}_{j+1})\\
&\geq
2\,\varepsilon - \sum_{l=i+1}^{j} \kappa^{j-l}\,\delta\geq
2\,\varepsilon-\frac{\delta}{1-\kappa} \geq \varepsilon\;.
\end{align*}
This proves (i) for the column $j+1$, and concludes the induction.


\smallskip

Conclusion (1) follows from (b) and the following claim, to be proved  by induction in $i$.

For every $i=0,1,\ldots, n-1$, \,
$g^{(i)}( B(x_0,\varepsilon))\subset X^0_{i}(\varepsilon)$, where $g^{(i)}=g_{i-1}\circ \ldots \circ g_0$.

Consider first the case $i=0$. 
Given $x\in  B(x_0,\varepsilon)$, \,
$$d(x,\partial X^0_0)\geq d(x_0,\partial X^0_0)-d(x,x_0)\geq 1-\varepsilon>\varepsilon\;.$$
This implies that $d(g_0(x),g_0(x_0))\leq\kappa\,d(x,x_0)$.
Thus
\begin{align*}
 d( g_0(x), \partial X^0_1) &\geq
 d(  g_0(x_0), \partial X^0_1) - d(g_0(x_0), g_0(x)) \geq 2\,\varepsilon -  d(g_0(x_0), g_0(x))\\
 &\geq 2\,\varepsilon -  \kappa\, d(x_0,x)\geq 
 2\,\varepsilon -  \kappa\, \varepsilon >\varepsilon\;
\end{align*}
which proves that $g_0( B(x_0,\varepsilon) )\subset X^0_1(\varepsilon)$.

Assume now that for every $l\leq i-1$,
$$(g_l\circ \ldots\circ g_0)( B(x_0,\varepsilon))\subset X^0_{l+1}(\varepsilon)\;.  $$  
By (b), $g^{(i)}$ acts
as a $\kappa^i$ contraction on $B(x_0,\varepsilon)$
and $g^{(i)}( B(x_0,\varepsilon) )\subset X^0_{i}(\varepsilon)$.
Thus for every $x\in B(x_0,\varepsilon)$,
\begin{align*}
 d( g^{(i+1)}(x), \partial X^0_{i+1}) &\geq
 d(  g_i(x_i), \partial X^0_{i+1}) - d(g_i(x_i), g^{(i+1)}(x))\\
 &\geq 2\,\varepsilon -  d(z_{i+1}^0, z_{i+1}^i) -
 d( z_{i+1}^0,  g^{(i+1)}(x))\\
  &\geq 2\,\varepsilon -  \sum_{l=0}^{i-1} d(z_{i+1}^l, z_{i+1}^{l+1}) -
 d( g^{(i+1)}(x_0),  g^{(i+1)}(x))\\
 &\geq 2\,\varepsilon -   (\delta+\kappa\, \delta +\ldots+\kappa^{i-1}\, \delta)
  -  \kappa^{i}\, d(x_0, x)  \\
 &\geq 2\,\varepsilon -  (\delta+ \kappa\, \delta +  \ldots+\kappa^{i-1}\, \delta)  - \kappa^{i}\, \varepsilon \\
 &\geq 2\,\varepsilon -  (1-\kappa)\,\varepsilon\,(1 + \kappa +  \ldots+\kappa^{i-1} ) - \kappa^{i}\, \varepsilon  
=  \varepsilon\;
\end{align*}
which proves that $g^{(i+1)}( B(x_0,\varepsilon))\subset  X^0_{i+1}(\varepsilon)$, and establishes the claim above.

Thus $g^{(n)}$ is well-defined on $B(x_0,\varepsilon)$,
and, because of assumption (b), $g^{(n)}$ is a $\kappa^n$ Lipschitz contraction on this ball.
This proves (1).

Item (2) follows by~\eqref{zij:claim}. In fact
$$
d(g_{n-1}(x_{n-1}), g^{(n)}(x_0) )  =
d(z^{n-1}_n , z^0_n ) \leq \sum_{l=1}^{n-1} d(z^{l}_n , z^{l-1}_n )  \leq \sum_{l=1}^{n-1} \kappa^{n-l-1}\,\delta \leq \frac{\delta}{1-\kappa}\;.
$$

Finally we prove (3). Assume   $x_0=g_{n-1}(x_{n-1})$.

It is enough to see that
$g^{(n)}( B(x_0,\varepsilon))\subset B(x_0,\varepsilon)$,
because by (1) $g^{(n)}$ acts as a $\kappa^n$-contraction in the closed ball $B(x_0,\varepsilon)$. The conclusion on the existence of a fixed point, as well as the proximity bound, follow from the classical fixed point theorem for Lipschitz contractions.

Given $x\in  B(x_0,\varepsilon)$,
we know from the previous calculation that \,
$$d(x_0, g^{(n)}(x_0))<\delta + \kappa\,\delta  + \ldots + \kappa^{n-2}\,\delta \;.$$
Hence
\begin{align*}
 d( g^{(n)}(x), x_0) &\leq
 d(g^{(n)}(x), g^{(n)}(x_0)) + d(g^{(n)}(x_0), x_0)\\
 &\leq \kappa^{n-1}\, d(x, x_0) + \delta + \kappa\,\delta  + \cdots + \kappa^{n-2}\,\delta\\
  &\leq \delta + \kappa\,\delta  + \cdots + \kappa^{n-2}\,\delta +
  \kappa^{n-1}\,\varepsilon \\
  &\leq (1-\kappa)\,\varepsilon\,(1 + \kappa + \cdots + \kappa^{n-2} ) +
  \kappa^{n-1}\,\varepsilon \\
  & =(1-\kappa)\,\varepsilon\, \frac{1-\kappa^{n-1}}{1-\kappa} +
  \kappa^{n-1}\,\varepsilon = \varepsilon\;.
\end{align*}
Thus $g^{(n)}(x)\in B(x_0,\varepsilon)$, which proves that
$g^{(n)}(B(x_0,\varepsilon))\subset B(x_0,\varepsilon)$.

\end{proof}

\begin{figure}
\begin{center}
\includegraphics*[scale={.6}]{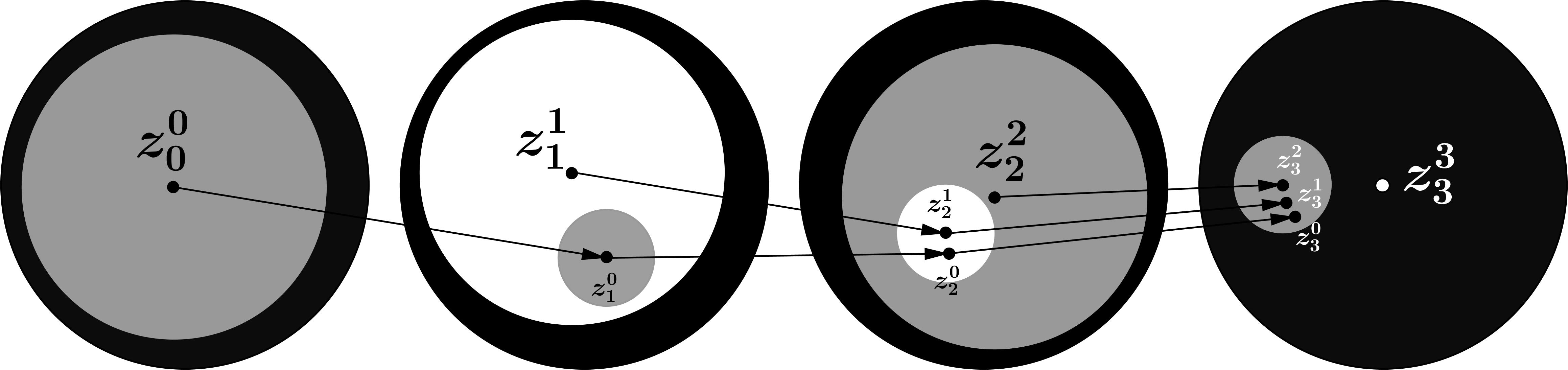} 
\end{center}
\caption{Shadowing property for a chain of contractive mappings}
\end{figure}

\subsection{Statement and proof of the AP}

In the AP's statement and proof we will use the notation introduced in subsection~\ref{subsection angles}.
Given a chain of linear mappings
$\{g_j:V_{j}\to V_{j+1}\}_{0\leq j\leq n-1}$ we denote the composition of the first $i$ maps by 
$g^{(i)}:= g_{i-1}\, \ldots \, g_1\, g_0$.

\begin{theorem}[Avalanche Principle] \label{Theorem:AP} There exists a constant
$c>0$ such that
 given   $0<\varepsilon<1$,  \;$0<\kappa\leq c\,\varepsilon^ 2$ 
and  a chain of linear mappings
$\{g_j:V_{j}\to V_{j+1}\}_{0\leq j\leq n-1}$ between Euclidean spaces $V_j$, \,   if 
\begin{enumerate}
\item[(a)] $\sgap(g_i)\leq \kappa$,\,  
for $0\leq i\leq n-1$, and 
\item[(b)] $\aangle(g_{i-1}, g_{i})\geq \varepsilon$,\; 
for $1\leq i\leq n-1$,
\end{enumerate}
then 
\begin{enumerate}

\item[(1)]  $d(\mostexp(g^{(n)}), \mostexp(g_{0})) 
 \lesssim    \kappa\,\varepsilon^{-1}$\,, 

\smallskip

\item[(2)]  $d(\mostexp(g^{(n)\ast}), \mostexp(g_{n-1}^\ast)) 
  \lesssim   \kappa\,\varepsilon^{-1}$\,,

\smallskip

\item[(3)]  $  \sgap(g^{(n)})  \leq   \left(\frac{\kappa\,(4+2\,\varepsilon)}{\varepsilon^{2}}\right)^n$\,,

\smallskip

\item[(4)] $\displaystyle 
 \abs{\log \norm{g^{(n)}} +
\sum_{i=1}^{n-2} \log \norm{g_i} -
\sum_{i=1}^{n-1}
\log  \norm{g_i\,g_{i-1}} } \lesssim  n\,\frac{\kappa}{\varepsilon^2}
\; .$
\end{enumerate}
\end{theorem}

\begin{remark}[On the assumptions]
\label{rmk on the assumptions}
Assumption (a) says that the (first) gap ratio of each $g_j$ is large,
$\rgap(g_j)\geq \kappa^{-1}$.
By  propositions~\ref{prod:2:lemma} and ~\ref{prop angle rift}, assumption (b) is equivalent to a condition on the rift,
$\rift(g_{j-1},g_j)\geq\varepsilon$\, for all $j=1,\ldots,n-1$.
\end{remark}

\begin{remark}[On the conclusions]
\label{rmk on the conclusions}
Conclusions (1) and (2) say that the most expanding direction $\mostexp(g^{(n)})$ of the product $g^{(n)}$, and its image 
$\varphi_{g^{(n)}} \mostexp(g^{(n)})$, are respectively $\kappa/\varepsilon$-close to the  most expanding direction $\mostexp(g_0)$ of $g_0$, and to the image $\varphi_{g_{n-1}} \mostexp(g_{n-1})$ of the  most expanding direction of $g_{n-1}$.
Conclusion (3) says that the composition map  $g^{(n)}$ has a large gap ratio.
Finally, conclusion (4) is equivalent to
$$  e^{-n\, C\,\kappa\,\varepsilon^{-2}}\leq \frac{\norm{g_{n-1}\, \ldots \, g_1 \,g_0} \, \norm{g_1}\, \ldots\, \norm{g_{n-2}} }{  \norm{g_1\,g_0}\, \ldots\, \norm{g_{n-1}\,g_{n-2}} }\leq  e^{n\, C\,\kappa\,\varepsilon^{-2}} \;, $$
for some universal constant $C>0$.
These inequalities describe the  asymptotic almost multiplicative behavior of the rifts
$$ \rho(g_0,g_1,\ldots, g_{n-1})\asp{ C\,\kappa/\varepsilon^{2}} \rho(g_0,g_1)\,\rho(g_1,g_2)\,\ldots\, \rho(g_{n-2},g_{n-1}) \;. $$
\end{remark}


\begin{proof}
The strategy of the proof is to look at the  contracting action  of linear mappings $g_j$ on the projective space.

For each $j=0,1,\ldots, n$ consider the compact metric space
$X_j=\Pp(V_j)$ with the normalized Riemannian distance,
$d(\hatu,\hatv)=\frac{2}{\pi}\,\rho(\hatu,\hatv)$, and define for
$0\leq j <n$
\begin{align*}
 X^0_j &:=\{\, \hatv\in X_j\,\colon\,
\aangle(\hatv,\mostexp(g_j))>0\,\}\;,\\
Y^0_j &:=\{\, \hatv\in X_j\,\colon\,
\aangle(\hatv,\mostexp(g_{j-1}^\ast))>0\,\}\;.
\end{align*}  
The domain  of the projective map $\varphi_{g_j}:\Pp(g_j)\subset X_j\to X_{j+1}$ clearly contains the open set $X^0_j$.
Analogously, the domain  of $\varphi_{g_{j-1}^\ast}:\Pp(g_j^\ast)\subset X_{j} \to X_{j-1}$   contains  $Y^0_j$.
We will apply lemma~\ref{shadow:lemma} to   chains
of projective maps formed out of the mappings $\varphi_{g_j}:X^0_j\to X_{j+1}$ and their adjoints
$\varphi_{g_{j-1}^\ast}:Y^0_j\to X_{j-1}$.

Take positive numbers $\varepsilon$ and $\kappa$ such that
$0<\kappa\ll \varepsilon^2$, let $r := \sqrt{1- \varepsilon^2/4 } $,
and define the following input parameters for the application of lemma~\ref{shadow:lemma},
\begin{align*}
\epsh &:= \frac{1}{\pi}\,\arcsin \varepsilon\;, \\
\kash &:= \kappa\,\frac{r+\sqrt{1-r^2}}{1-r^2}
 \asymp  \frac{ 4\,\kappa }{\varepsilon^2}\;, \\
 \desh &:=\frac{  \kappa\,r }{\sqrt{1-r^2}}
 \asymp  \frac{ 2\,\kappa }{\varepsilon}\;.
\end{align*}
A simple calculation shows that there exists
$0<c<1$ such that for any $0<\varepsilon<1$ and  $0<\kappa\leq c\,\varepsilon^2$, the pre-conditions
$0<\desh<\kash<1$  and $\frac{\desh}{1-\kash}<\epsh<1/2$ of the shadowing lemma are satisfied.

Define $x_j=\mostexp(g_j)$ and
 $x_j^\ast=\mostexp(g_{j-1}^\ast)$. This lemma is going to be applied to the following chains of maps and sequences of points
\begin{align*}
\text{(A)} \quad &\varphi_{g_0}, \ldots,
\varphi_{g_{n-1}}, \varphi_{g_{n-1}^\ast}, \ldots, \varphi_{g_0^\ast}\;, \quad
x_0, \ldots, x_{n-1}, x_n^\ast,
\ldots, x_1^\ast \;,\\
\text{(B)} \quad &\varphi_{g_{n-1}^\ast}, \ldots,  \varphi_{g_0^\ast}, \varphi_{g_0}, \ldots,
\varphi_{g_{n-1}} \;, \quad
x_n^\ast,
\ldots, x_1^\ast,x_0,\ldots, x_{n-1} \;,
\end{align*}
from which we will infer the conclusions (1) and (2).
Let us  check now that assumptions (a)-(d) of
lemma~\ref{shadow:lemma} hold in both cases (A) and (B).

By definition
$\partial X^0_j:=\{\, \hatv\in X_j\,\colon\, \aangle(\hatv, x_j)=0\,\}=\{\, \hatv\in X_j\,\colon\,  \hatv\perp x_j \,\}$.
Hence,   if $\hatv\in \partial X^0_j$ then  $d(x_j,\hatv)=1$, which proves that $d(x_j,\partial X^0_j)=1$.
Analogously, $\partial Y^0_j =\{\, \hatv\in X_j\,\colon\,  \hatv\perp  x_j^\ast \,\}$ and $d(x_j^\ast,\partial Y^0_j)=1$. Therefore assumption (a) holds.

By definition of $ X^0_j(\varepsilon)$,
\begin{align*}
\hatv\in X^0_j(\varepsilon) \; &\Leftrightarrow \;
d(\hatv, \partial X^0_j) \geq \varepsilon
\; \Leftrightarrow \;
\rho(\hatv, \partial X^0_j) \geq \frac{\pi}{2}\, \varepsilon \\
&\Leftrightarrow \; \delta(\hatv, \partial X^0_j) =\aangle(\hatv, x_j) \geq \sin\left(\frac{\pi}{2}\,\varepsilon\right)  \\
& \Leftrightarrow \; \delta(\hatv, x_j) \leq \cos\left(\frac{\pi}{2}\,\varepsilon\right) \;.
\end{align*}
Similarly, by definition of $Y^0_j(\varepsilon)$,
$$ \hatv\in Y^0_j(\varepsilon) \;  \Leftrightarrow \; \delta(\hatv, x_j^\ast) \leq \cos\left(\frac{\pi}{2}\,\varepsilon\right)\;. $$
Thus, because
$$ \cos\left( \frac{\pi}{2}\,\epsh\right)
=\cos\left(\frac{1}{2}\,\arcsin \varepsilon \right) \leq \sqrt{1-\frac{\varepsilon^2}{4}}= r\;,$$
we have \, 
$X^0_j(\epsh)\subset B^{(\delta)}(x_j,r)$
 \, and \, 
$Y^0_j(\epsh)\subset B^{(\delta)}(x_j^\ast,r)$, and assumption (b) holds by proposition~\ref{proj:contr} (3).

By the gap assumption,
\begin{align*}
\aangle(\varphi_{g_j}(x_j), x_{j+1})
= \aangle( \mostexp(g_j^\ast), \mostexp(g_{j+1}))
= \aangle(g_j,g_{j+1})\geq \varepsilon \;.
\end{align*}
Therefore
\begin{align*}
d(\varphi_{g_j}(x_j), \partial X^0_{j+1})
&= \frac{2}{\pi}\,\arcsin \delta(\varphi_{g_j}(x_j), \partial X^0_{j+1}) = \frac{2}{\pi}\,\arcsin \aangle(\varphi_{g_j}(x_j), x_{j+1})\\
&\geq \frac{2}{\pi}\,\arcsin \varepsilon = 2\,\epsh\;.
\end{align*}
Similarly, by the gap assumption,
\begin{align*}
\aangle(\varphi_{g_{j-1}^\ast}(x_j^\ast), x_{j-1}^\ast)
= \aangle( \mostexp(g_{j-1} ), \mostexp(g_{j-1}^\ast))
=  \aangle(g_{j+1}^\ast, g_j^\ast) = \aangle(g_j,g_{j+1})\geq \varepsilon \;,
\end{align*}
and in the same way we infer  that
\begin{align*}
d(\varphi_{g_{j-1}^\ast}(x_j^\ast), \partial Y^0_{j-1})\geq  \frac{2}{\pi}\,\arcsin \varepsilon =  2\,\epsh\;.
\end{align*}
This proves that   (c) of the shadowing lemma holds.
Notice that in both cases (A) and (B), the assumption (c) holds trivially for the middle points,  because
$\varphi_{g_{n-1}}(x_{n-1})=x_n^\ast\in Y_n^0(2\,\epsh)$ and  
$\varphi_{g_{0}^\ast}(x_{1}^\ast)=x_0\in X_0^0(2\,\epsh)$.

It was proved above that $X^0_j(\epsh)\subset B^{(\delta)}(x_j,r)$
 \, and \, 
$Y^0_j(\epsh)\subset B^{(\delta)}(x_j^\ast,r)$.
By~\eqref{metric equivalence} we have
$d(\hatu,\hatv)\leq \delta(\hatu,\hatv)$.
Thus by proposition~\ref{proj:contr} (1),
$$ \varphi_{g_j}( X_j^0(\epsh))\subset B^{(\delta)}(x_j^\ast,\desh)\subset B^{(d)}(x_j^\ast,\desh) 
\; \text{ with }\;
x_j^\ast = \varphi_{g_j}(x_j)\;,$$
and analogously,
$$ \varphi_{g_{j-1}^\ast}( Y_j^0(\epsh))\subset B^{(\delta)}(x_{j-1},\desh)\subset B^{(d)}(x_{j-1},\desh) \; \text{ with }\;
x_{j-1} = \varphi_{g_{j-1}^\ast}(x_j^\ast)\;.$$
Hence, (d) of lemma~\ref{shadow:lemma} holds.
 
Therefore,
because $\varphi_{g_0^\ast}(x_1^\ast)=x_0$
and  $\varphi_{g_{n-1}}(x_{n-1})=x_n^\ast$,
conclusion (2) of lemma~\ref{shadow:lemma} holds
for both chains (A) and (B).
The projective points $\mostexp(g^{(n)})$
and $\mostexp(g^{(n)\ast})$ are the unique fixed points of the chains of mappings (A) and (B), respectively. Hence, by the shadowing lemma both distances $d(x_0, \mostexp(g^{(n)}))$ and 
$d(x_n^\ast, \mostexp(g^{(n)\ast}))$ are bounded above by 
$$\frac{\desh}{(1-\kash)\,(1-\kash^{2n})}
 \asymp \desh \asymp\frac{\kappa}{\varepsilon}\;.$$
This proves conclusions (1) and (2) of the AP.

\medskip


From proposition~\ref{derivative varphig}
we infer that for any $g\in \mathcal{L}(V)$,
$$  \norm{ (D\varphi_g)_{\mostexp(g)} } =\frac{s_2(g)}{\norm{g}} = \sgap(g)\;.$$
Hence, by (1) of the shadowing lemma,
\begin{align*}
\sgap(g^{(n)}) &= \norm{ (D\varphi_{g^{(n)}})_{\mostexp(g^{(n)})} } \leq {\rm Lip}(\varphi_{g^{(n)}}\vert_{ B(\mostexp(g_0),\epsh)} ) \\
&\leq (\kash)^n \leq 
\left( \frac{\kappa\,(4+ 2\,\varepsilon)}{\varepsilon^2}\right)^n\;. 
\end{align*}
This proves conclusion (3) of the AP.

Before proving (4), notice that applying
(3) to the chain of
 linear maps $g_0,\ldots, g_{i-1}$  we get that  $g^{(i)}:=g_{i-1}\,\ldots \, g_0$ has a first gap ratio  for all $i=1,\ldots, n$.

We claim that
\begin{equation}\label{claim}
 \abs{\aangle(g^{(i)},g_i)-\aangle(g_{i-1},g_i)} \lesssim \kappa\,\varepsilon^{-1} \;.
\end{equation} 
By (2) of the AP,  on the chain of
 linear maps $g_0,\ldots, g_{i-1}$,   
 $$ d(\mostexp(g^{(i)\ast}), \mostexp(g_{i-1}^\ast) )\leq \frac{\desh}{(1-\kash)(1- \kash^{2i})} \lesssim \kappa\,\varepsilon^{-1} \;. $$
Hence, by proposition~\ref{prop aangle continuity}

\begin{align*}
\abs{ \aangle(g^{(i)},g_i)-\aangle(g_{i-1},g_i) } &=
\abs{ \aangle(\mostexp(g^{(i)\ast}), \mostexp(g_i) )-\aangle(\mostexp(g_{i-1}^\ast), \mostexp(g_i ) )   }\\
& \leq  d(\mostexp(g^{(i)\ast}),\mostexp(g_{i-1}^\ast))
 \lesssim \kappa\,\varepsilon^{-1} \;.
\end{align*}
For any  $i$, the logarithm of any ratio between the four factors
$\aangle( g^{(i)}, g_i)$, $\bangle( g^{(i)}, g_i)$,
$\aangle( g_{i-1}, g_i)$ and  $\bangle( g_{i-1}, g_i)$
 is of order $\kappa\,\varepsilon^{-2}$.
In fact, by~(\ref{claim}) 
\begin{align*}
\abs{\log\frac{ \aangle(g^{(i)},g_i) }{ \aangle(g_{i-1},g_i) } }
&\leq \frac{1}{\varepsilon}\,\abs{ \aangle(g^{(i)},g_i) -\aangle(g_{i-1},g_i) }
\lesssim \kappa\,\varepsilon^{-2} \;.
\end{align*}
By Lemma~\ref{alpha:beta:bound}, 
and since $\sigma_{\tau_j}(g_i)\leq \kappa$, 
the other ratios are of the same magnitude as this one.
Thus, for some universal constant $C>0$,  each of these four ratios is inside the interval
$[e^{-C\,\kappa\,\varepsilon^ {-2}},e^{C\,\kappa\,\varepsilon^ {-2}}]$.

Finally, applying proposition~\ref{svp:lemma:norm}
 to the rifts $\rift(g_0,\ldots, g_{n-1})$,
 $\rift(g_0,g_1)$, $\rift(g_1,g_2)$, etc, we have
 
 $$  e^{-n\, C\,\kappa\,\varepsilon^{-2}} \leq \prod_{i=1}^{n-1} \frac{\aangle(g^{(i)},g_i)}{\bangle(g_{i-1},g_i)} \leq
 \frac{\rift(g_0,\ldots, g_{n-1})}{\prod_{i=1}^{n-1}\rift(g_{i-1},g_i)} \leq 
 \prod_{i=1}^{n-1} \frac{\bangle(g^{(i)},g_i)}{\aangle(g_{i-1},g_i)} \leq
 e^{ n\, C\,\kappa\,\varepsilon^{-2}}\;, $$
which by  remark~\ref{rmk on the conclusions}  is equivalent to (4).

\end{proof}

\bigskip

Next proposition is a practical reformulation of
the Avalanche Principle. 

\begin{proposition} \label{AP-practical}
There exists $c>0$ such that
 given $0<\epsilon<1$,  $0<\kappa\leq c\,\epsilon^ 2$ 
and  \,  $g_0, g_1,\ldots, g_{n-1}\in\gl(m,\R)$, \,
 if
\begin{align*}
\rm{(gaps)} \  & \rgap (g_i) >  \frac{1}{\ka} &  \text{for all }  & \ \  0 \le i \le n-1  
\\
\rm{(angles)} \  & \frac{\norm{ g_i \cdot g_{i-1} }}{\norm{g_i}  \, \norm{ g_{i-1}}}  >  \ep  & \
 \text{for all }   & \  \ 1 \le i \le n-1  
\end{align*}
then  
\begin{align*} 
& \max\left\{ \, d(\mostexp(g^{(n)\ast}), \mostexp(g_{n-1}^\ast)),\,
d(\mostexp(g^{(n)}), \mostexp(g_{0})) \, \right\}
 \lesssim    \kappa\,\ep^{-1} \\
& \sabs{ \log \norm{ g^{(n)} } + \sum_{i=1}^{n-2} \log \norm{g_i} -  \sum_{i=1}^{n-1} \log \norm{ g_i \cdot g_{i-1}} } \less n \cdot \frac{\ka}{\ep^2} \;.
\end{align*}
 \end{proposition}

 \begin{proof} 
 Consider the constant $c>0$ in theorem~\ref{Theorem:AP},
 let $c':=c\,(1-2\,c^2)$ and assume $0<\kappa \le c'\,\epsilon^2$.

 Assumption (gaps) here is equivalent to assumption (a) of
 theorem~\ref{Theorem:AP}. By  proposition~\ref{prop angle rift}, the assumption (angles) here implies
\begin{align*}
\aangle(g_{i-1}, g_i) & \geq \rift(g_{i-1}, g_i)\,
\sqrt{1-\frac{2\,\kappa^2}{\rift(g_{i-1}, g_i)^2} } \\
& \geq \epsilon\,
\sqrt{1-\frac{2\,\kappa^2}{\epsilon^2} }  \geq \epsilon\,
\sqrt{1- 2\,c^2\,\epsilon^2  } =:\epsilon'\;,
\end{align*}
Since  $0<\kappa\leq c'\,\epsilon^2$, \, and \,
$ c'\,\epsilon^2 \leq c\,(1-2\,c^2\,\epsilon^2)\,\epsilon^2 
= c\,(\epsilon')^2$ we have
 $0<\kappa\leq c \,(\epsilon')^2$.
Thus, because $\epsilon\asymp \epsilon'$, this proposition follows from conclusions (1), (2) and (4) of
theorem~\ref{Theorem:AP}.

 \end{proof}

\subsection{Consequences of the AP}

Given a chain of linear maps
$\{g_j:V_j\to V_{j+1}\}_{0\leq j\leq n-1}$ between Euclidean spaces $V_j$, and integers $0\leq i<j\leq n$ we 
define
$$ g^{(j,i)}:= g_{j-1}\circ \ldots \circ g_{i+1}\circ g_i\;. $$
With this notation the following relation holds for 
$0\leq i <k <j\leq n$,
$$ g^{(j,i)} =  g^{(j,k)}\circ  g^{(k,i)} \;.$$

Next proposition states, in a quantified way, that the most expanding directions $\mostexp(g^{n,i)})\in\Pp(V_i)$ are almost invariant under the adjoints of the chain mappings.

\begin{proposition} \label{prop::almost invariance}
Under the assumptions of  theorem~\ref{Theorem:AP}, where $0<\kappa\ll \varepsilon^2$,
$$ d( \varphi_{g_i^\ast}\,\mostexp(g^{(n,i+1)}), \mostexp(g^{(n,i)})) \lesssim \frac{\kappa}{\varepsilon}\, (\frac{\kappa\,(4+2\,\varepsilon)}{\varepsilon^2})^{n-i}\;.$$
\end{proposition}

\begin{proof}
Consider $\kappa$, $\varepsilon$, $\kash$ and $\epsh$
as in  theorem~\ref{Theorem:AP}.
From the proof of item (3) of the AP, applied to the chain of mappings $g_{n-1}^\ast,\ldots, g_{i}^\ast$,
we conclude that the composition
$g^{(n,i)}=g_{i}^\ast\circ \ldots \,\circ g_{n-1}^\ast$
is a $(\kash)^{n-i}$-Lipschitz contraction on the ball 
$B(\mostexp(g_{n-1}^\ast),\epsh)$.
On the other hand, by (2) of the AP we have
$d( \,\mostexp(g^{(n,i+1)\ast} , \mostexp(g_{n-1}^\ast)   \,)\lesssim\kappa\,\varepsilon^{-1}$  and \, $d( \,\mostexp(g_{n-1}^\ast) , \mostexp(g^{(n,i)\ast}   \,)\lesssim\kappa\,\varepsilon^{-1}$.
Since $\kappa\,\varepsilon^{-1}\ll \varepsilon \asymp\epsh$,
both projective points $\mostexp(g^{(n,i)\ast})$ and $\mostexp(g^{(n,i+1)\ast})$ belong to the ball $B(\mostexp(g_{n-1}^\ast),\epsh)$.
Thus, 
\begin{align*}
&d( \varphi_{g_i^\ast}\, \mostexp(g^{(n,i+1)}), \mostexp(g^{(n,i)})  )  =\\
&\qquad = d(\,   \varphi_{g_i^\ast}\circ\varphi_{g^{(n,i+1)\ast} } \,\mostexp(g^{(n,i+1)\ast} ),
\, \varphi_{g^{(n,i)\ast} }\,  \mostexp(g^{(n,i)\ast}  )\, ) \\
&\qquad = d(\,   \varphi_{g^{(n,i)\ast} } \,\mostexp(g^{(n,i+1)\ast} ),
\, \varphi_{g^{(n,i)\ast} }\,  \mostexp(g^{(n,i)\ast}  )\, ) \\
&\qquad   \leq   (\kash)^{n-i}  \, d( \,\mostexp(g^{(n,i+1)\ast} , \mostexp(g^{(n,i)\ast}   \,) \\
&\qquad   \leq   (\frac{\kappa\, (4+2\,\varepsilon)}{\varepsilon^2})^{n-i}  \, \left( d( \,\mostexp(g^{(n,i+1)\ast} , \mostexp(g_{n-1}^\ast)   \,)  +  d( \,\mostexp(g_{n-1}^\ast) , \mostexp(g^{(n,i)\ast}   \,)  \right)\\
&\qquad   \lesssim  \frac{2\,\kappa}{\varepsilon}\, (\frac{\kappa\, (4+2\,\varepsilon)}{\varepsilon^2})^{n-i} \;. 
\end{align*}
which proves the proposition.

\end{proof}

Most expanding directions and norms of products of chains matrices under an application of the AP
admit the following modulus of continuity.

\begin{proposition}
Let $c>0$ be the universal constant in theorem~\ref{Theorem:AP}.
Given numbers $0<\varepsilon<1$ and $0<\kappa<c\,\varepsilon^2$,
and given two chains of matrices $g_0,\ldots, g_{n-1}$ and  $g_0',\ldots, g_{n-1}'$ in $\Mat(m,\R)$,  both satisfying the assumptions of the AP for the given parameters $\kappa$ and $\varepsilon$, \, if \,
 $\drel(g_i,g_i')<\delta$ \, for all $i=0,1,\ldots, n-1$,
then
\begin{enumerate}
\item[(a)] $d( \, \mostexp(g_{n-1}\,\ldots\, g_0), \,
\mostexp(g_{n-1}'\,\ldots\, g_0')\, ) \lesssim  \frac{ \kappa}{\varepsilon} +  8\,\delta   $,
\item[(b)]\; 
$\displaystyle  \abs{\log \frac{\norm{g_{n-1}\,\ldots\, g_0}}{\norm{g_{n-1}'\,\ldots\, g_0'}} }\lesssim n\,\left( \frac{\kappa}{\varepsilon^2} + \frac{ \delta  }{\varepsilon} \right) $.

\end{enumerate}

\end{proposition} 

\begin{proof}
Item (a) follows from conclusion (1) of theorem~\ref{Theorem:AP}, and
proposition~\ref{lipschitz:eigendir},  
\begin{align*}
d(\, \mostexp( g_{n-1}\,\ldots\, g_0),\,
\mostexp( g_{n-1}'\,\ldots\, g_0')\,) &\leq 
d(\, \mostexp( g_{n-1}\,\ldots\, g_0),\,
\mostexp(g_0)\,) \\
&\qquad + d(\mostexp(g_0),\mostexp(g_0')) + 
d(\, \mostexp(g_0'),\,
\mostexp( g_{n-1}'\,\ldots\, g_0')\,)\\
&\lesssim 2\,\frac{\kappa}{\varepsilon} +
\frac{16\,\delta}{1-\kappa^2} \lesssim  \frac{ \kappa}{\varepsilon} +  8\,\delta \;.
\end{align*}

Assuming $\norm{g_i}\geq \norm{g_i'}$, we have
$$ \frac{\norm{g_i}}{\norm{g_i'}}
\leq 1+ \frac{\norm{g_i-g_i'}}{\norm{g_i'}}
\leq 1+ \frac{\norm{g_i}}{\norm{g_i'}}\, \drel(g_i,g_i') 
\leq  1+   \delta\,\frac{\norm{g_i}}{\norm{g_i'}}
$$
which implies
$$ \frac{\norm{g_i}}{\norm{g_i'}}\leq \frac{1}{1-\delta} \;. $$
Because the case $\norm{g_i}\leq \norm{g_i'}$ is analogous,
we conclude that
$$ \abs{\log \frac{\norm{g_i}}{\norm{g_i'}} }
\leq \log\left( \frac{1}{1-\delta}\right)
\leq \frac{\delta}{1-\delta} \asymp \delta \;. $$
Since the two chains of matrices  satisfy the assumptions of the AP we have
$$ \frac{\norm{g_i\,g_{i-1}}}{\norm{g_i}\,\norm{g_{i-1}}}
\geq \aangle(g_{i-1},g_i)\geq \varepsilon \quad \text{ and }\quad
\frac{\norm{g_i'\,g_{i-1}'}}{\norm{g_i'}\,\norm{g_{i-1}'}}
\geq \aangle(g_{i-1}',g_i')\geq \varepsilon \;. $$
A simple calculation gives
$$ \drel(\,g_i\,g_{i-1}, \, g_i'\,g_{i-1}' \,)\leq
\frac{2}{(1-\delta)^2}\, \frac{\delta}{\varepsilon}
\asymp \frac{\delta}{\varepsilon} \;.$$
Therefore
$$ \abs{\log \frac{\norm{g_i\,g_{i-1}}}{\norm{g_i'\,g_{i-1}'}} }
\lesssim \frac{\delta}{\varepsilon} \;.$$
Hence, by conclusion (4) of the AP we have
\begin{align*}
\abs{\log \frac{\norm{g_{n-1}\,\ldots \,g_{0}}}{\norm{g_{n-1}'\,\ldots \, g_{0}'}} } &\leq
\abs{\log \frac{\norm{g_{n-1}\,\ldots \,g_{0}}\,\norm{g_1}\,\ldots\, \norm{g_{n-2}}}{ \norm{g_1\,g_0}\,\ldots\,\norm{g_{n-1}\,g_{n-2}} }  } \\
&\qquad +  \abs{\log \frac{ \norm{g_1'\,g_0'}\,\ldots\,\norm{g_{n-1}'\,g_{n-2}'} }{\norm{g_{n-1}'\,\ldots \,g_{0}'}\,\norm{g_1'}\,\ldots\, \norm{g_{n-2}'}}  } \\
&\qquad +  \sum_{i=1}^{n-2} \abs{ \log \frac{\norm{g_i'} }{\norm{g_i}  } } 
+  \sum_{i=1}^{n-1} \abs{ \log \frac{\norm{g_i\,g_{i-1} } }{\norm{g_i'\,g_{i-1}' }}  }\\
&\lesssim 2\,n\,\frac{\kappa}{\varepsilon^2}
+(n-2)\,\delta + (n-1)\,\frac{\delta}{\varepsilon} \\
&\lesssim n\,\left( \frac{\kappa}{\varepsilon^2} + \frac{ \delta  }{\varepsilon} \right)\;,
\end{align*}
which proves (b).

\end{proof}

\smallskip

Next proposition is a flag version of the  AP.

Let $\tau=(\tau_1,\ldots, \tau_k)$ be a signature
with $0<\tau_1<\tau_2 <\ldots <\tau_k <m $.

We call $\tau$-block product to any of  the functions
$\pi_{\tau,j}:\Mat(m,\R)\to \R$,
$$
\pi_{\tau,j}(g):=  s_{\tau_{j-1}+1}(g) \,  \ldots  \, s_{\tau_j}(g)\;,\qquad 1\leq j\leq k \;,
$$
where by convention $\tau_0=0$.
A $\tau$-singular value product, abbreviated $\tau$-s.v.p., is any product of distinct $\tau$-block products.
By definition, $\tau$-block products are $\tau$-singular value products. Other examples of   $\tau$-singular value products are the functions 
$$p_{\tau_j} (g) = s_1 (g) \, \ldots  \, s_{\tau_j} (g)
= \norm{\wedge_{\tau_j} g} \;. $$
Note that for every $1 \le j \le k$ we have:
$$\pi_{\tau,j}(g) = \frac{p_{\tau_j} (g)}{p_{\tau_{j-1}} (g)}\;,$$
and
$$p_{\tau_j} (g) = \pi_{\tau,1}(g) \, \ldots \, \pi_{\tau,j}(g)\;. $$

\begin{proposition} 
\label{Flag:AP} 
Let $c>0$ be the universal constant in theorem~\ref{Theorem:AP}.
Given numbers  $0<\varepsilon<1$,  $0<\kappa\leq c\,\varepsilon^ 2$ 
and  a chain of matrices
$g_j\in \Mat(m,\R)$, with $j=0,1,\ldots, n-1$, \,   if 
\begin{enumerate}
\item[(a)] $\sgap_\tau(g_i)\leq \kappa$,\, 
for $0\leq i\leq n-1$, and 
\item[(b)] $\aangle_\tau(g_{i-1}, g_{i})\geq \varepsilon$,\; 
for $1\leq i\leq n-1$,
\end{enumerate}
then 
\begin{enumerate}
\item[(1)]  $d(\mostexp_\tau(g^{(n)\ast}), \mostexp_\tau(g_{n-1}^\ast)) 
  \lesssim   \kappa\,\varepsilon^{-1}$   
\item[(2)]  $d(\mostexp_\tau(g^{(n)}), \mostexp_\tau(g_{0})) 
 \lesssim    \kappa\,\varepsilon^{-1}$  
 \item[(3)]
$  \sigma_\tau(g^{(n)})  \leq   \left(\frac{\kappa\,(4+ 2\,\varepsilon)}{\varepsilon^{2}}\right)^n$

\item[(4)]  for any $\tau$-s.v.p. function $\pi$,
$$
 \abs{\log \pi(g^{(n)}) +
\sum_{i=1}^{n-2} \log \pi(g_i) -
\sum_{i=1}^{n-1}
\log  \pi(g_i\,g_{i-1}) } \lesssim  n\,\frac{\kappa}{\varepsilon^2}
\; . $$
\end{enumerate}
\end{proposition}

\begin{proof}
For each $j=1,\ldots, k$, consider the chain of matrices 
$\wedge_{\tau_j}g_0, \wedge_{\tau_j}g_1, \ldots, \wedge_{\tau_j} g_{n-1}$.
Assumptions (a) and (b) here imply the corresponding assumptions of theorem~\ref{Theorem:AP} for all these chains of exterior power matrices.
Hence, by (1) of the AP
\begin{align*}
d(\mostexp_{\tau_j}(g^{(n)\ast}), \mostexp_{\tau_j}(g_{n-1}^\ast)) &=
d(\Psi(\mostexp_{\tau_j}( g^{(n)\ast})), \Psi(\mostexp_{\tau_j}( g_{n-1}^\ast)))\\
&=d( \mostexp(\wedge_{\tau_j} g^{(n)\ast}),  \mostexp(\wedge_{\tau_j} g_{n-1}^\ast)) \lesssim
 \kappa\,\varepsilon^{-1}\;.
\end{align*}
Thus, taking the maximum in $j$ we get
$d(\mostexp_{\tau}(g^{(n)\ast}), \mostexp_{\tau}(g_{n-1}^\ast))\lesssim
 \kappa\,\varepsilon^{-1}$, which proves (1).
Conclusion (2) follows in the same way.

Similarly, from (3) of theorem~\ref{Theorem:AP},
we infer the corresponding conclusion here
$$
\sgap_{\tau}(g^{(n)}) =
\max_{1\leq j\leq k} \sgap_{\tau_j}(g^{(n)})
= \max_{1\leq j\leq k} \sgap(\wedge_{\tau_j} g^{(n)}) \leq   \left(\frac{\kappa\,(4+ 2\,\varepsilon)}{\varepsilon^{2}}\right)^n\;.
$$

Let us now prove (4).

For the $\tau$-s.v.p. $\pi(g)=p_{\tau,j}(g)=\norm{\wedge_{\tau_j} g}$ conclusion (4) is a consequence of the corresponding conclusion of theorem~\ref{Theorem:AP}.

For the $\tau$-block product
$\pi=\pi_{\tau,j}$, since
$$ \log \pi(g) = \log \norm{\wedge_{\tau_j}  g }
- \log \norm{\wedge_{\tau_{j-1}}  g } \;, $$
 conclusion (4) follows again from theorem~\ref{Theorem:AP} (4).

Finally, since any  $\tau$-s.v.p. is a finite product of $\tau$-block products we can  reduce (4) to the previous case.

\end{proof}

We finish this section with a version of the AP for complex matrices.

The  {\em singular values} of a complex matrix $g\in\Mat(m,\C)$ are defined to be the eigenvalues of the positive semi-definite hermitian matrix
$g^\ast\,g$, where $g^\ast$ stands for the transjugate of $g$, i.e., the conjugate transpose of $g$. Similarly, the 
 {\em singular vectors} of $g$
 are defined as the eigenvectors of $g^\ast\,g$.
The sorted singular values of $g\in \Mat(m,\C)$ are denoted  by
\, $ s_1(g)\geq s_2(g) \geq \ldots \geq s_m(g) $.
The top singular value of $g$ coincides with its norm,
$s_1(g)=\norm{g}$.

The (first) gap ratio of $g$ is the quotient
$\sgap(g):=s_2(g)/s_1(g)\leq 1$. We say that $g\in\Mat(m,\C)$ has a  (first) gap ratio  when
$\sgap(g)<1$. When this happens the complex eigenspace
$$ \{\, v\in\C^m\,\colon\, g^\ast\,g\,v=\norm{g}\,v\,\}
= \{\, v\in\C^m\,\colon\, \norm{g\,v}=\norm{g}\,\norm{v} \,\} $$
has complex dimension one, and determines a point
in $\Pp(\C^m)$, denoted by $\mostexp(g)$ and referred as the $g$-most expanding direction.


Given points $\hatv,\hatu\in\Pp(\C^m)$, we set
\begin{equation}\label{complex aangle def} 
 \aangle(\hatv,\hatu):= \frac{\abs{\langle v,u \rangle}}{\norm{v}\,\norm{u}} \qquad \text{ where } \quad 
v\in\hatv,\; u\in\hatu\;.
\end{equation}

Given $g,g'\in \Mat(m,\C)$, both with (first) gap ratios,
we define the {\em angle between} $g$ and $g'$ to be
$$ \aangle(g,g'):= \aangle(\mostexp(g^\ast),\mostexp(g'))\;. $$

With these definitions,  the real version of the AP leads in a straightforward manner to a slightly weaker complex version, stated and proved below.
However, adapting the original proof to the complex case,
 replacing each real concept by its complex analog,
would lead  to the same stronger estimates as in theorem~\ref{Theorem:AP}.

\begin{proposition}[Complex AP]
\label{complex:AP} 
Let $c>0$ be the universal constant in theorem~\ref{Theorem:AP}.
Given numbers  $0<\varepsilon<1$,  $0<\kappa\leq c\,\varepsilon^ 4$ 
and  a chain of matrices
$g_j\in \Mat(m,\C)$, with $j=0,1,\ldots, n-1$, \,   if 
\begin{enumerate}
\item[(a)] $\sgap(g_i)\leq \kappa$,\, 
for $0\leq i\leq n-1$, and 
\item[(b)] $\aangle(g_{i-1}, g_{i})\geq \varepsilon$,\; 
for $1\leq i\leq n-1$,
\end{enumerate}
then 
\begin{enumerate}
\item[(1)]  $d(\mostexp(g^{(n)\ast}), \mostexp(g_{n-1}^\ast)) 
  \lesssim   \kappa\,\varepsilon^{-2}$   
\item[(2)]  $d(\mostexp(g^{(n)}), \mostexp(g_{0})) 
 \lesssim    \kappa\,\varepsilon^{-2}$  
 \item[(3)]
$  \sigma(g^{(n)})  \leq   \left(\frac{\kappa\,(4+ 2\,\varepsilon^2)}{\varepsilon^{4}}\right)^n$

\item[(4)]  
$\displaystyle
 \abs{\log \norm{ g^{(n)} } +
\sum_{i=1}^{n-2} \log \norm{g_i} -
\sum_{i=1}^{n-1}
\log  \norm{ g_i\,g_{i-1} } } \lesssim  n\,\frac{\kappa}{\varepsilon^4}
\; . $ 
\end{enumerate}
\end{proposition}

\begin{proof}
Make the identification $\C^m\equiv \R^{2m}$,
and given  $g\in\Mat(m,\C^m)$
 denote by $g^\R\in\Mat(2m,\R)$ the matrix representing the linear operator $g:\R^{2m}\to\R^{2m}$ in the canonical basis. 
 
We make explicit the relationship between gap ratios and angles of the complex matrices and $g,g'\in\Mat(m,\C)$, and the gap ratios and angles of their real analogues   $g^\R$ and $(g')^\R$.

Given $g\in\Mat(m,\C)$, for each eigenvalue $\lambda$ of $g$,
the matrix $g^\R$ has a corresponding pair of eigenvalues $\lambda,\overline{\lambda}$.
Since $g\mapsto g^\R$ is a $C^\ast$-algebra homomorphism, we have
$(g^\ast\,g)^\R=(g^\R)^\ast\,(g^\R)$.
Therefore, for all $i=1,\ldots, m$,
 $s_i(g)=s_{2i-1}(g^\R) = s_{2i}(g^\R)$.
 In particular, considering the signature $\tau=(2)$,
\begin{equation}
\label{sgap real-complex}
\sgap_{(2)}(g^\R)
  = \frac{s_3(g^\R)}{s_1(g^\R)}
  =\frac{s_2(g)}{s_1(g)} = \sgap(g)
  \;. 
\end{equation} 
The $g$-most expanding direction 
$\mostexp(g)\in\Pp(\C^m)$ is a complex line
which we can identify with the real $2$-plane
$\mostexp_{(2)}(g^\R)$.
This identification, $\mostexp(g)\equiv \mostexp_{(2)}(g^\R)$,  comes from a natural isometric embedding
$\Pp(\C^m) \hookrightarrow \Gr_2(\R^{2m})$.

Consider two points $\hatv,\hatu\in\Pp(\C^m)$
and take unit vectors $v\in \hatv$ and $u\in \hatu$. Denote by $U,V\subset \C^m$ the complex lines spanned by these vectors, which are planes in $\Gr_{2}(\R^{2m})$.
Consider  the complex orthogonal projection  onto the complex line $V$,
$\pi_{u,v}:U\to V$, defined by
$\pi_{u,v}(x):=\langle x,v\rangle\,v$.
By~\eqref{complex aangle def} we have  $\aangle(\hatv,\hatu)=\norm{\pi_{u,v}}$.
On the other hand, since $\pi_{u,v}\circ \pi_{u,v}=\pi_{u,v}$ and
$\langle x- \pi_{u,v}(x), v\rangle =0$ for all $x\in U$, it follows that $\pi_{u,v}$ is the restriction to $U$ of the (real) orthogonal projection onto the $2$-plane $V$.
Thus, by  proposition~\ref{prop: alpha = det Pi(E F)}(b),
$$ \aangle_2(U,V)=\abs{\det{}_\R(\pi_{u,v})} =\abs{\det{}_\C(\pi_{u,v})}^2 = \norm{\pi_{u,v}}^2=\aangle(\hatv,\hatu)^2\;.  $$
In particular,
\begin{equation}
\label{angle real-complex}
\aangle_{(2)}(g^\R, (g')^\R) =
\aangle_{(2)}(\mostexp((g^\R)^\ast) , \mostexp((g')^\R)) =
\aangle(\mostexp(g^\ast) , \mostexp(g'))^2 = \aangle(g,g')^2\;.
\end{equation}

Take $\kappa,\varepsilon>0$ such that
$\kappa <c\,\varepsilon^4$, $0<\varepsilon<1$, and consider a chain of matrices $g_j\in\Mat(m,\C)$, $j=0,1,\ldots, n-1$ satisfying the assumptions (a) and (b) of the complex AP.
By~\eqref{sgap real-complex} and~\eqref{angle real-complex},
the assumptions (a) and (b) of proposition~\ref{Flag:AP} hold for the chain of real matrices
 $ g_j^\R\in\Mat(2m,\R)$, $j=0,1,\ldots, n-1$,   with parameters $\kappa$ and $\varepsilon^2$, and with $\tau=(2)$.
Therefore conclusions (1)-(4) of the complex AP follow from the corresponding conclusions of proposition~\ref{Flag:AP}.
In conclusion (4) we use the $(2)$-singular value product  $\pi(g):=\norm{g}^2 = \norm{\wedge_2 g^\R}$.

\end{proof}

\subsection*{Acknowledgments}

\bigskip

The first author was supported by 
Funda\c{c}\~{a}o  para a  Ci\^{e}ncia e a Tecnologia, 
UID/MAT/04561/2013.

The second author was supported by the Norwegian Research Council project no. 213638, "Discrete Models in Mathematical Analysis".

\bigskip

\bibliographystyle{amsplain} 

\providecommand{\bysame}{\leavevmode\hbox to3em{\hrulefill}\thinspace}
\providecommand{\MR}{\relax\ifhmode\unskip\space\fi MR }
\providecommand{\MRhref}[2]{%
  \href{http://www.ams.org/mathscinet-getitem?mr=#1}{#2}
}
\providecommand{\href}[2]{#2}

\end{document}